\documentclass[a4paper]{article}
\usepackage[latin1]{inputenc} % ou \usepackage[utf8]{inputenc}
\usepackage[OT1]{fontenc} % ou \usepackage[OT1]{fontenc}
\usepackage[english]{babel}
\usepackage{RR,RRthemes}
\usepackage[colorlinks=true,linkcolor=blue]{hyperref}
\usepackage{amsfonts}
\usepackage{amssymb}
\usepackage{amsmath}
\usepackage{amsthm}
\usepackage{seceqn}             % numérotation des équations par section
\usepackage{dsfont}             % pour pouvoir écrire l'indicatrice
\usepackage{array}
\usepackage{color}
\RRdate{Juillet 2010}

\RRauthor{
Pascal Massart
\thanks[sfn,1]{Université Paris-Sud, Laboratoire de Mathématiques, UMR 8628, 91405 Orsay, France}
\thanks[sfn,2]{INRIA Saclay Ile-de-France, Projet SELECT}
    \and
Caroline Meynet \thanksref{sfn,1} \thanksref{sfn,2}}

\authorhead{Massart \& Meynet}

\RRtitle{Une inégalité oracle $\ell_{1}$ pour le Lasso}

\RRetitle{An $\ell_{1}$-Oracle Inequality for the Lasso}

\titlehead{An $\ell_{1}$-Oracle Inequality for the Lasso}

\RRresume{Ces dernières années, alors que de nombreux efforts ont été faits pour prouver que le Lasso
agit comme une procédure de sélection de variables au prix d'hypothèses contraignantes sur
la structure géométrique de ces variables, peu de travaux analysant la performance du Lasso en tant qu'algorithme
de régularisation $\ell_{1}$ ont été réalisés. Notre premier objectif est de fournir un résultat dans cette
voie en prouvant que le Lasso se comporte presque aussi bien que le Lasso déterministe à condition que le
paramètre de régularisation soit bien choisi. Ce résultat ne nécessite aucune hypothèse, ni sur la structure
des variables, ni sur la fonction de régression.

Notre second objectif est de contruire un nouvel estimateur particulièrement adapté à l'utilisation
de dictionnaires infinis. Cet estimateur est construit par pénalisation $\ell_{0}$ d'une suite
d'estimateurs Lasso associés à une suite dyadique croissante de dictionnaires tronqués.
L'algorithme correspondant choisit automa\-tiquement le niveau de troncature garantissant le meilleur
compromis entre approximation, régularisation $\ell_{1}$ et parcimonie. D'un point de vue théorique, nous
établissons une inégalité oracle satisfaite par cet estimateur.

Toutes les inégalités oracles présentées dans cet article sont obtenues en appliquant un théorème
de sélection de modèles parmi un ensemble de modèles non linéaires, grâce à l'idée clé qui consiste
à envisager la régularisation $\ell_{1}$ comme une procédure de sélection de modèles parmi des boules $\ell_{1}$.

Enfin, nous déduisons de ces inégalités oracles des vitesses de convergence
sur de larges classes de fonctions montrant en particulier que les estimateurs Lasso sont
aussi performants que les algorithmes greedy.}

\RRabstract{These last years, while many efforts have been made to prove that the Lasso behaves like a variable
selection procedure at the price of strong assumptions on
the geometric structure of these variables,  much less attention has been paid
to the analysis of the performance of the Lasso as a regularization algorithm.
Our first purpose here is to provide a result in this
direction by proving that the Lasso works almost as well as the deterministic Lasso provided that
the regularization parameter is
properly chosen. This result does not require any assumption at all, neither on the structure of
the variables nor on the regression function.

Our second purpose is to introduce a new estimator particularly adapted to deal
with infinite countable dictionaries. This estimator is constructed as
an $\ell_{0}$-penalized estimator among a sequence of Lasso estimators
associated to a dyadic sequence of growing truncated dictionaries. The selection procedure
automatically chooses the
best level of truncation of the dictionary so as to make the best tradeoff
between approximation, $\ell_{1}$-regularization and sparsity. From a
theoretical point of view, we shall provide an oracle inequality satisfied by
this selected Lasso estimator.

All the oracle inequalities presented in this paper are obtained via the
application of a single general theorem of model selection among a collection
of nonlinear models. The key idea that enables us to apply this general theorem is to
see $\ell_{1}$-regularization as a model selection procedure among $\ell_{1}$-balls.

Finally, rates of convergence achieved by the Lasso and the selected Lasso estimators
on a wide class of functions are derived from these oracle inequalities, showing that these estimators
perform at least as well as greedy algorithms.}

\RRmotcle{Lasso, inégalités oracles $\ell_{1}$, sélection de modèles par pénalisation, boules $\ell_{1}$,
modèles linéaires gaussiens généralisés.}
\RRkeyword{Lasso, $\ell_{1}$-oracle inequalities, Model selection by
penalization, $\ell_{1}$-balls, Generalized linear Gaussian model.}

\RRdomaine{1}
\RRprojets{SELECT}
\RRthemeProj{select}
\RCSaclay

\begin{document}

\RRNo{7356}

\makeRR

\numberwithin{equation}{section} \theoremstyle{plain}
\newtheorem{thm}{Theorem}[section]
\newtheorem{theorem}{Theorem}[section]
\newtheorem{acknowledgement}[theorem]{Acknowledgement}
\newtheorem{algorithm}[theorem]{Algorithm}
\newtheorem{axiom}[theorem]{Axiom}
\newtheorem{case}[theorem]{Case}
\newtheorem{claim}[theorem]{Claim}
\newtheorem{conclusion}[theorem]{Conclusion}
\newtheorem{condition}[theorem]{Condition}
\newtheorem{conjecture}[theorem]{Conjecture}
\newtheorem{corollary}[theorem]{Corollary}
\newtheorem{criterion}[theorem]{Criterion}
\newtheorem{definition}[theorem]{Definition}
\newtheorem{example}[theorem]{Example}
\newtheorem{exercise}[theorem]{Exercise}
\newtheorem{lemma}[theorem]{Lemma}
\newtheorem{notation}[theorem]{Notation}
\newtheorem{problem}[theorem]{Problem}
\newtheorem{proposition}[theorem]{Proposition}
\newtheorem{remark}[theorem]{Remark}
\newtheorem{solution}[theorem]{Solution}
\newtheorem{summary}[theorem]{Summary}

\newcommand{\non}{\nonumber}
\newcommand{\droit}[1]{{\upshape #1}}
\newcommand{\n}{\noindent}
\newcommand{\p}{\lambda}
\newcommand{\argmin}{\mathop{\mathrm{arg\,min}}}
\newcommand{\e}{\xi_{i}}
\newcommand{\lp}{\hat{\theta}}
\newcommand{\lnp}{f_{\hat{\theta}}}
\newcommand{\pp}{p}
\newcommand{\lnpp}{f_{\hat{\theta}_{\pp}}}
\newcommand{\f}{f}
\newcommand{\ft}{h}
\newcommand{\fu}{g}
\newcommand{\pj}{\phi_{j}}
\newcommand{\norm}[1]{\Vert#1\Vert}
\newcommand{\norme}[1]{\Vert#1\Vert}
\newcommand{\deflp}{\lp \in \argmin_{\theta \in \mathbb{R}^{p}}
\norme{Y - X \theta}^2  + \p \norm{\theta}_{1}}
\newcommand{\deflnp}{\lnp \in \argmin_{\theta \in \mathbb{R}^{p}}
\norme{\f - \ft}^{2} + \p \norm{\ft}_{1}}
\newcommand{\lud}{\mathcal{L}_{1}(\mathcal{D})}
\newcommand{\dico}{\{\phi_{1},\dots,\phi_{p}\}}
\newcommand{\fc}{\hat{f}}
\newcommand{\vp}{\phi}
\newcommand{\ftheta}{\phi.\theta}
\newcommand{\lasso}{Lasso}
\newcommand{\tcp}{\hat{\theta}_{\pp}}
\newcommand{\fcp}{\hat{f}_{p}}
\newcommand{\fcpc}{\hat{f}_{\hat{p}}}
\newcommand{\ludp}{\mathcal{L}_{1}(\mathcal{D}_{\pp})}
\newcommand{\fcf}{\fc_{\hat{p}}}
\newcommand{\lambdap}{\lambda_{p}}
\newcommand{\lambdapc}{\lambda_{\hat{p}}}
\newcommand{\ftcpc}{f_{\hat{\theta}_{\hat{\pp}}}}
\newcommand{\normup}[1]{\Vert #1 \Vert_{\mathcal{L}_{1}(\mathcal{D}_{\pp})}}
\newcommand{\normupc}[1]{\Vert #1 \Vert_{\mathcal{L}_{1}(\mathcal{D}_{\hat{\pp}})}}
\newcommand{\crit}[1]{\gamma(#1)}
\newcommand{\pa}{\alpha}
\newcommand{\bq}{\mathcal{B}_{q}}
\newcommand{\fl}{\ell}
\newcommand{\thetal}{\theta^{\fl}}
\newcommand{\wlq}{w\mathcal{L}_{q}}
\newcommand{\flp}{\f_{p}}
\newcommand{\thetalp}{\theta^{p,\beta}}
\newcommand{\cf}{\theta^*}
\newcommand{\cfl}{\thetal}
\newcommand{\cflp}{\theta^L_{\pp}}
\newcommand{\cte}{C}
\newcommand{\bqr}{\mathcal{B}_{q,r}}
\newcommand{\eps}{\varepsilon}
\newcommand{\normu}[1]{\Vert #1\Vert_{\mathcal{L}_{1}(\mathcal{D})}}
\newcommand{\lu}{\mathcal{L}_{1}(\mathcal{D})}
\newcommand{\penp}{\operatorname*{pen}(\pp)}
\newcommand{\pf}{p_{\f}}
\newcommand{\nbqpf}{\norm{\pf}_{\bq}}
\newcommand{\nbqf}{\norm{\f}_{\bq}}
\newcommand{\nwlqpf}{\norm{\pf}_{\wlq}}
\newcommand{\nwlqf}{\norm{\f}_{\wlq}}
\newcommand{\lur}{\mathcal{L}_{1,r}}
\newcommand{\normlur}[1]{\Vert #1 \Vert_{\lur}}
\newcommand{\normbqr}[1]{\Vert #1 \Vert_{\bqr}}
\newcommand{\normbq}[1]{\Vert #1 \Vert_{\bq}}
\newcommand{\wlqc}{w\ell_{q}}
\newcommand{\normwlqc}[1]{\Vert #1 \Vert_{\wlqc}}
\newcommand{\normwlq}[1]{\Vert #1 \Vert_{\wlq}}
\newcommand{\ens}{\Lambda}
\newcommand{\ensn}{\mathbb{N}^{*}}
\newcommand{\lambdapo}{\lambda_{p_{0}}}
\newcommand{\normupo}[1]{\norm{#1}_{\mathcal{L}_{1}(\mathcal{D}_{p_{0}})}}
\newcommand{\ludpo}{\mathcal{L}_{1}(\mathcal{D}_{p_{0}})}
\newcommand{\normude}[1]{\Vert #1\Vert_{\mathcal{L}_{1}(D)}}
\newcommand{\lude}{\mathcal{L}_{1}(D)}
\newcommand{\besov}{\mathcal{B}^{\, r}_{2,\infty}}
\newcommand{\normbesov}[1]{\Vert #1 \Vert_{\besov}}
\newcommand{\ellipsoid}{\besov}
\newcommand{\pb}{\beta}
\newcommand{\fpbp}{\f_{\pp,\pb_{\pp}}}
\newcommand{\fg}{g}
\newcommand{\q}{a}
\newcommand{\w}{b}
\newcommand{\dd}{d}
\newcommand{\logn}{\ln}
\newcommand{\normud}[1]{\Vert #1\Vert_{\mathcal{L}_{1}(\mathcal{D})}}
\newcommand{\rw}{R}
\newcommand{\rb}{R}
\newcommand{\R}{M}
\newcommand{\swlq}{\mathcal{L}_{q}}
\newcommand{\fct}{\tilde{\f}}
\newcommand{\enso}{\Lambda_{0}}
\newcommand{\blambdap}{\beta}
\newcommand{\lassotype}{selected Lasso }
\newcommand{\fcpo}{\hat{\f}_{p_{0}}}

\setcounter{tocdepth}{2}

\tableofcontents

\section{Introduction}

We consider the problem of estimating a regression function $\f$ belonging to a
Hilbert space $\mathbb{H}$ in a fairly general Gaussian framework which
includes the fixed design regression or the white noise frameworks. Given a
dictionary $\mathcal{D} = \left\{\vp_{j}\right\}_{j}$ of functions
in~$\mathbb{H}$, we aim at constructing an estimator $\fc = \hat{\theta}.\vp :=
\sum_{j} {\hat{\theta}}_{j}\, \vp_{j}$ of~$\f$ which enjoys both good
statistical properties and computational performance even for large or infinite
dictionaries.

\medskip
For high-dimensional dictionaries, direct minimization of the empirical risk
can lead to overfitting and we need to add a complexity penalty to avoid it.
One could use an $\ell_{0}$-penalty, i.e.\ penalize the number of non-zero
coefficients $\hat{\theta}_{j}$ of~$\fc$ (see \cite{bm-min} for instance) so as
to produce interpretable sparse models but there is no efficient algorithm to
solve this non-convex minimization problem when the size of the dictionary
becomes too large. On the contrary, $\ell_{1}$-penalization leads to convex
optimization and is thus computationally feasible even for high-dimensional
data. Moreover, due to its geometric properties, $\ell_{1}$-penalty tends to
produce some coefficients that are exactly zero and hence often behaves like an
$\ell_{0}$-penalty. These are the main motivations for introducing
$\ell_{1}$-penalization rather than other penalizations.

\medskip In the linear regression framework, the idea of
$\ell_{1}$-penalization was first introduced by Tibshirani~\cite{tibshirani}
who considered the so-called \lasso\ estimator (Least Absolute Shrinkage and
Selection Operator). Then, lots of studies on this estimator have been carried
out, not only in the linear regression framework but also in the nonparametric
regression setup with quadratic or more gene\-ral loss functions (see
\cite{bic-tsy}, \cite{koltchinskii}, \cite{vdg1} among others). In the
particular case of the fixed design Gaussian regression models, if we observe
$n$ i.i.d.\! random couples $(x_{1},Y_{1}),\dots,(x_{n},Y_{n})$ such that
\begin{equation}\label{modelintro}Y_{i} = \f(x_{i})+\sigma
\xi_{i},\quad i=1,\dots, n,\end{equation} and if we consider a dictionary
$\mathcal{D}_{p} = \{\vp_{1},\dots,\vp_{p}\}$ of size $p$, the \lasso\
estimator is defined as the following $\ell_{1}$-penalized least squares
estimator
\begin{equation}\label{lassointro}\fcp := \fcp(\lambdap) = \argmin_{\ft \in \ludp}
\norm{Y-\ft}^2+\lambdap
\norm{\ft}_{\mathcal{L}_{1}(\mathcal{D}_{p})},\end{equation} where
$\norm{Y-\ft}^2 := \sum_{i=1}^{n} \left(Y_{i}-\ft(x_{i})\right)^2/n$ is the
empirical risk of $\ft$, $\ludp$ is the linear span of $\mathcal{D}_{p}$
equipped with the $\ell_{1}$-norm
$\norm{\ft}_{\mathcal{L}_{1}(\mathcal{D}_{p})}  := \inf\{\norm{\theta}_{1} =
\sum_{j=1}^{p} \vert\theta_{j}\vert :\ \ft = \theta.\vp = \sum_{j=1}^{p}
\theta_{j}\, \vp_{j}\}$ and $\lambdap>0$ is a regularization parameter.

\medskip Since $\ell_{1}$-penalization can be seen as a ``convex relaxation'' of
$\ell_{0}$-penalization, many efforts have been made to prove that the \lasso\
behaves like a variable selection procedure by esta\-blishing sparsity oracle
inequalities showing that the $\ell_{1}$-solution mimicks the
``$\ell_{0}$-oracle'' (see for instance~\cite{bic-tsy} for the prediction loss
in the case of the quadratic nonparametric Gaussian regression model).
Nonetheless, all these results require strong restrictive assumptions on the
geometric structure of the variables. We refer to \cite{vdg2} for a detailed
overview of all these restrictive assumptions.

\bigskip In this paper, we shall explore another approach by analyzing
the performance of the \lasso\ as a regularization algorithm rather than a
variable selection procedure. This shall be done by providing an
$\ell_{1}$-oracle type inequality satisfied by this estimator (see Theorem
\ref{zlasso}). In the particular case of the fixed design Gaussian regression
model, this result says that if $\mathcal{D}_{p}=\{\vp_{1},\dots,\vp_{p}\}$
with $\max_{j=1,\dots,p}\Vert\pj\Vert \leq1$, then there exists an absolute
constant $C>0$ such that for all $\lambdap \geq 4 \sigma\, n^{-1/2} (\sqrt{\ln
p} +1)$, the \lasso\ estimator defined by \eqref{lassointro} satisfies
\begin{equation}\label{thmintro}\mathbb{E}\left[\norm{\f-\fcp}^{2} + \lambdap \normup{\fcp}\right]
\leq C\left[ \inf_{\ft \in \ludp} \left(\norm{\f-\ft}^{2}+\lambdap
\normup{\ft}\right) +\frac{\sigma \lambdap}{\sqrt{n}}\right].\end{equation}
This simply means that, provided that the regularization parameter $\lambdap$
is pro\-perly chosen, the \lasso\ estimator works almost as well as the
deterministic \lasso. Notice that, unlike the sparsity oracle inequalities, the
above result does not require any assumption neither on the target function
$\f$ nor on the structure of the variables $\vp_{j}$ of the
dictionary~$\mathcal{D}_{p}$, except simple normalization that we can always
assume by considering $\vp_{j}/\norm{\vp_{j}}$ instead of~$\vp_{j}$. This
$\ell_{1}$-oracle type inequality is not entirely new. Indeed, on the one hand,
Barron and al.\! \cite{bar-huang} have provided a similar risk bound but in the
case of a truncated \lasso\ estimator under the assumption that the target
function is bounded by a constant. On the other hand, Rigollet and Tsybakov
\cite{rig-tsy} are proposing a result with the same flavour but with the subtle
difference that it is expressed as a probability bound which does not imply
\eqref{thmintro} (see a more detailed explanation in Section
\ref{soussectionlassofini}).

We shall derive \eqref{thmintro} from a fairly general model selection theorem
for non linear models, interpreting $\ell_{1}$-regularization as an
$\ell_{1}$-balls model selection criterion (see Section \ref{mst}). This
approach will allow us to go one step further than the analysis of the \lasso\
estimator for finite dictionaries. Indeed, we can deal with infinite
dictionaries in various situations.

\medskip

In the second part of this paper, we shall thus focus on infinite countable
dictionaries.  The idea is to order the varia\-bles of the infinite dictionary
$\mathcal{D}$ thanks to the a priori knowledge we can have of these variables,
then write the dictionary $\mathcal{D}=\{\vp_{j}\}_{j\in
\mathbb{N}^*}=\{\vp_{1},\vp_{2},\dots\}$ according to this order, and consider
the dyadic sequence of truncated dictionaries $\mathcal{D}_{1} \subset \dots
\subset \mathcal{D}_{\pp}\subset \dots \subset \mathcal{D}$ where
$\mathcal{D}_{\pp} = \{\vp_{1},\dots,\vp_{\pp}\}$ for $p \in \{2^{J}, J \in
\mathbb{N}\}$. Given this sequence $\left(\mathcal{D}_{\pp}\right)_{\pp}$, we
introduce an associated sequence of \lasso\ estimators $(\fcp)_{\pp}$ with
regularization parameters~$\lambdap$ depending on $\pp$, and choose
$\fc_{\hat{p}}$ as an $\ell_{0}$-penalized estimator among this sequence by
penalizing the size of the truncated dictionaries $\mathcal{D}_{p}$.
 This \lassotype estimator $\fc_{\hat{p}}$ is thus based on an algorithm choosing
automatically the best level of truncation of the dictionary and is constructed
to make the best tradeoff between approximation, $\ell_{1}$-regularization and
sparsity. From a theoretical point of view, we shall establish an oracle
inequality satisfied by this \lassotype estimator. Of course, although
introduced for infinite dictionaries, this estimator remains well defined for
finite dictionaries and it may be profitable to exploit its good properties and
to use it rather than the classical \lasso\ for such dictionaries.

\medskip

In a third part of this paper, we shall focus on the rates of convergence of
the sequence of the \lasso s and the \lassotype estimator introduced above. We
shall provide rates of convergence of these estimators for a wide range of
function classes described by mean of interpolation spaces $\bqr$ that are
adapted to the truncation of the dictionary and constitute an extension of the
intersection between weak-$\mathcal{L}_{q}$ spaces and Besov spaces $\besov$
for non orthonormal dictiona\-ries. Our results will prove that the \lasso\
estimators $\fcp$ for $p$ large enough and the \lassotype estimator $\fcf$
perform as well as the greedy algorithms described by Barron and al.\! in
\cite{bcdd}. Besides, our convergence results shall highlight the advantage of
using the \lassotype estimator rather than \lasso s. Indeed, we shall prove
that the \lasso\ estimators $\fcp$, like the greedy algorithms in \cite{bcdd},
are efficient only for $p$ large enough compared to the unknown parameters of
smoothness of $\f$ whereas $\fcf$ always achieves good rates of convergence
whenever the target function $\f$ belongs to some interpolation space $\bqr$.
In particular, we shall check that these rates of convergence are optimal by
establishing a lower bound of the minimax risk over the intersection between
$\swlq$ spaces and Besov spaces $\besov$ in the orthonormal case.

\medskip We shall end this paper by providing some theoretical results on the
performance of the \lasso\ for particular infinite uncountable dictionaries
such as those used for neural networks. Although \lasso\ solutions can not be
computed in practice  for such dictionaries, our purpose is just to point out
the fact that the \lasso\ theoretically performs as well as the greedy
algorithms in \cite{bcdd}, by establi\-shing rates of convergence based on an
$\ell_{1}$-oracle type inequality similar to~\eqref{thmintro} satisfied by the
\lasso\ for such dictionaries.

\bigskip The article is organized as follows. The notations and the genera\-lized linear
Gaussian framework in which we shall work throughout the paper are introduced
in Section 2. In Section 3, we consider the case of finite dictionaries and
analyze the performance of the \lasso\ as a regulari\-zation algorithm by
providing an $\ell_{1}$-oracle type inequality which highlights the fact that
the \lasso\ estimator works almost as well as the deterministic \lasso\
provided that the regularization parameter is large enough. In section 4, we
study the case of infinite counta\-ble dictionaries and establish a similar
oracle inequality for the \lassotype estimator $\fcf$. In section 5, we derive
from these oracle inequalities rates of convergence of the \lasso s and the
\lassotype estimator for a variety of function classes. Some theoretical
results on the performance of the \lasso\ for the infinite uncountable
dictionaries used to study neural networks in the artificial intelligence field
are mentioned in Section 6. Finally, Section 7 is devoted to the explanation of
the key idea that enables us to derive all our oracle inequalities from a
single general model selection theorem and to the statement of this general
theorem. The proofs are postponed until Section 8.

%%%%%%%%%%%%%%%%%%%%%%%%%%%%%%%%%%%%%%%%%%%%%%%%%%%%%%%%%%%%%%%%%%%%%%%%%%%%%%%%%%%%%%%
%%%%%%%%%%%%%%%%%%%%%%%%%%%%%%%%%%%%%%%%%%%%%%%%%%%%%%%%%%%%%%%%%%%%%%%%%%%%%%%%%%%%%%%
\section{Models and notations} \label{section1}

%%%%%%%%%%%%%%%%%%%%%%%%%%%%%%%%%%%%%%%%%%%%%%%%%%%%%%%%%%%%%%%%%%%%%%%%%%%%%%%%%%%%%%
\subsection{General framework and statistical problem}

Let us first describe the generalized linear Gaussian model we shall work with.
We consider a separable Hilbert space $\mathbb{H}$ equipped with a scalar
product $\langle ., .\rangle$ and its associated norm $\norm{.}$.

\begin{definition}\label{defw}\droit{\textbf{[Isonormal Gaussian process]}} A
centered Gaussian process $(W(\ft))_{\ft \in \mathbb{H}}$ is isonormal if its
covariance is given by $\mathbb{E}[W(\fu)W(\ft)] = \langle\fu,\ft\rangle$ for
all $\fu, \ft \in \mathbb{H}$.\end{definition} The statistical problem we
consider is to approximate an unknown target function $\f$ in $\mathbb{H}$ when
observing a process $\left(Y(\ft)\right)_{\ft \in \mathbb{H}}$ defined by
\begin{equation}\label{statpbg} Y (\ft) = \langle \f, \ft \rangle +
\varepsilon W(\ft), \quad \ft \in \mathbb{H},\end{equation} where
$\varepsilon
> 0$ is a fixed parameter and $W$ is an isonormal process.
This framework is convenient to cover both finite-dimensional models and the
infinite-dimensional white noise model as described in the following examples.

\begin{example}\label{excglrm}\droit{\textbf{[Fixed design Gaussian regression
model]} Let $\mathcal{X}$ be a measurable space. One observes $n$ i.i.d.\!
random couples $(x_{1},Y_{1}),\dots,(x_{n},Y_{n})$ of $\mathcal{X} \times
\mathbb{R}$ such that}
\begin{equation}\label{modelesection1}Y_{i} = \f(x_{i}) +
\sigma \e, \quad i=1,\dots,n,\end{equation} \droit{where  the covariates
$x_{1},\dots, x_{n}$ are deterministic elements of $\mathcal{X}$,  the errors
$\e$ are
 i.i.d. $\mathcal{N}\left(0, 1\right)$, $\sigma>0$  and $\f
: \mathcal{X} \mapsto \mathbb{R}$ is the unknown regression function to be
estimated. If one considers $\mathbb{H} = \mathbb{R}^n$ equipped with the
scalar product $\langle u,v\rangle = \sum_{i=1}^n u_{i}\, v_{i}/n$, defines $y
= (Y_{1},\dots,Y_{n})^{T}$, $\xi = (\xi_{1},\dots,\xi_{n})^{T}$ and
improper\-ly denotes $h=(h(x_{1}),\dots,h(x_{n}))$ for every $h : \mathcal{X}
\mapsto \mathbb{R}$, then $W(\ft) := \sqrt{n}\, \langle \xi, \ft \rangle$
defines  an isonormal Gaussian process on $\mathbb{H}$ and $Y (\ft) := \langle
y,\ft\rangle$ satisfies \eqref{statpbg} with $\varepsilon := \sigma/\sqrt{n}$.}

\n \droit{Let us notice that}\begin{equation}\label{notnorme}\norme{h} :=
\sqrt{\frac{1}{n} \sum_{i=1}^{n} h^{2}(x_{i})}\end{equation} \droit{corresponds
to the $\mathbb{L}_{2}$-norm with respect to the measure $\nu_{x} :=
\sum_{i=1}^{n} \delta_{x_{i}}/n$ with $\delta_{u}$ the Dirac measure at $u$. It
depends on the sample size $n$ and on the training sample via
$x_{1},\dots,x_{n}$ but we omit this dependence in notation \eqref{notnorme}.}
\end{example}

\begin{example}\label{whitenoise}\droit{\textbf{[The white noise framework]}
In this case, one observes $\zeta(x)$ for $x\in[0,1]$ given by the stochastic
differential equation}
\[
d\zeta(x)=\f(x)\, dx+\varepsilon\, dB(x)  \text{ \droit{with} }\zeta(0)
=0\text{,}%
\]
\droit{where $B$ is a standard Brownian motion, $\f$ is a square-integrable
function and $\eps > 0$. If we define $W(\ft) =\int_{0}^{1}\ft(x)\, dB(x)$ for
every $\ft\in\mathbb{L}_{2}([0,1]),$ then $W$ is an isonormal process
 on $\mathbb{H} = \mathbb{L}_{2}([0,1])$, and $Y(\ft) =\int_{0}^{1}\ft(x)\,
d\zeta(x)$ obeys to (\ref{statpbg}) provided that $\mathbb{H}$ is equipped with
its usual scalar product $\left\langle \f,\ft\right\rangle =\int_{0}^{1}\f(x)
\ft(x)\, dx$. Typically, $\f$ is a signal and $d\zeta(x)$ represents the noisy
signal received at time $x$. This framework easily extends to a $d$-dimensional
setting if one considers some multivariate Brownian sheet $B$ on $[0,1]^{d}$
and takes $\mathbb{H}=\mathbb{L}_{2}\left([0,1]^{d}\right)$.}\end{example}

%%%%%%%%%%%%%%%%%%%%%%%%%%%%%%%%%%%%%%%%%%%%%%%%%%%%%%%%%%%%%%%%%%%%%%%%%%%%%%%%%%%
\subsection{Penalized least squares estimators}

To solve the general statistical problem \eqref{statpbg}, one can consider a
dictionary $\mathcal{D}$, i.e.\! a given finite or infinite set of functions
$\vp_{j} \in \mathbb{H}$ that arise as candidate basis functions for estimating
the target function $\f$, and construct an estimator $\fc = \hat{\theta}.\vp :
= \sum_{j,\, \vp_{j} \in \mathcal{D}}\, \hat{\theta}_{j}\, \vp_{j}$ in the
linear span of $\mathcal{D}$. All the matter is to choose a ``good'' linear
combination in the following meaning. It makes sense to aim at constructing an
estimator as the best approximating point of $\f$ by minimizing
$\Vert\f-\ft\Vert$ or, equivalently, $-2\langle \f, \ft \rangle +
\Vert\ft\Vert^2$. However $\f$ is unknown, so one may instead minimize the
empirical least squares criterion
\begin{equation}\label{defgammah}\gamma(\ft) := -2 Y(\ft) + \norm{\ft}^2.\end{equation}
But since we are mainly interested in very large dictionaries, direct
minimization of the empirical least squares criterion can lead to overfitting.
To avoid it, one can rather consider a penalized risk minimization problem and
consider \begin{equation}\label{plse}\fc \in \argmin_{\ft} \gamma(\ft) +
\operatorname*{pen}(\ft),\end{equation} where $\operatorname*{pen}(\ft)$ is a
positive penalty to be chosen. Finally, since the resulting estimator $\fc$
depends on the observations, its quality can be measured by its quadratic risk
$\mathbb{E}[\Vert\f - \fc\Vert^2]$.

\medskip The penalty $\operatorname*{pen}\left(\ft\right)$
can be chosen according to the statistical target. In the recent years, the
situation where the number of variables $\vp_{j}$ can be very large (as
compared to $\varepsilon^{-2}$) has received the attention of many authors due
to the increa\-sing number of applications for which this can occur.
Micro-array data analysis or signal reconstruction from a dictionary of
redundant wavelet functions are typical examples for which the number of
variables either provided by Nature or considered by the statistician is large.
Then, an interesting target is to select the set of the ``most significant''
variables $\vp_{j}$ among the initial collection. In this case, a convenient
choice for the penalty is the $\ell_{0}$-penalty that penalizes the number of
non-zero coefficients $\hat{\theta}_{j}$ of $\fc$, thus providing sparse
estimators and interpretable models. Nonetheless, except when the functions
$\vp_{j}$ are orthonormal, there is no efficient algorithm to solve this
minimization problem in practice when the dictionary becomes too large. On the
contrary, $\ell_{1}$-penalization, that is to say $\operatorname*{pen}(\ft)
\propto \normud{\ft} := \inf \left\{\norm{\theta}_{1} = \sum_{j,\, \vp_{j} \in
\mathcal{D}} \vert\theta_{j}\vert\ \  \text{ such that } \ft =
\theta.\vp\right\}$, leads to convex optimization and is thus computationally
feasible even for high-dimensional data. Moreover, due to its geometric
properties, $\ell_{1}$-penalty tends to produce some coefficients that are
exactly zero and thus often behaves like an $\ell_{0}$-penalty, hence the
popularity of $\ell_{1}$-penalization and its associated estimator the \lasso\
defined by
$$\fc(\p)  = \argmin_{\ft \in \lud}\ \crit{\ft} + \p \normu{\ft},\quad \p
>0,$$ where $\lud$ denotes the set of functions $\ft$  in the linear span of $\mathcal{D}$ with finite $\ell_{1}$-norm
$\normud{\ft}$.

%%%%%%%%%%%%%%%%%%%%%%%%%%%%%%%%%%%%%%%%%%%%%%%%%%%%%%%%%%%%%%%%%%%%%%%%%%%%%%%%%%%%%%%%
%%%%%%%%%%%%%%%%%%%%%%%%%%%%%%%%%%%%%%%%%%%%%%%%%%%%%%%%%%%%%%%%%%%%%%%%%%%%%%%%%%%%%%%%%
\section{The \lasso\ for finite dictionaries}\label{section2}

While many efforts have been made to prove that the \lasso\ behaves like a
varia\-ble selection procedure at the price of strong (though unavoidable)
assumptions on the geometric structure of the dictionary (see \cite{bic-tsy} or
\cite{vdg2} for instance), much less attention has been paid to the analysis of
the performance of the Lasso as a regularization algorithm. The analysis
 we propose below goes in this very direction. In this section,
we shall consider a finite dictionary $\mathcal{D}_{p}$ of size $p$ and provide
an $\ell_{1}$-oracle type inequality bounding the quadratic risk
 of the \lasso\ estimator  by the
infimum over $\ludp$ of the tradeoff between the approximation term
$\norm{\f-\ft}^2$ and the $\ell_{1}$-norm $\normup{\ft}$.

%%%%%%%%%%%%%%%%%%%%%%%%%%%%%%%%%%%%%%%%%%%%%%%%%%%%%%%%%%%%%%%%%%%%%%%%%%%%%%%%%%%%%%
\subsection{Definition of the \lasso\ estimator}\label{sectiondeflassocasfini}

We consider the generalized linear Gaussian model and the statistical
pro\-blem~\eqref{statpbg} introduced in the last section.  Throughout this
section, we assume that $\mathcal{D}_{p} = \{\vp_{1},\dots,\vp_{p}\}$ is a
finite dictionary of size $p$. In this case, any $\ft$ in the linear span of
$\mathcal{D}_{p}$ has finite $\ell_{1}$-norm
\begin{equation}\label{defnormemodelpu}\normup{\ft} := \inf
\left\{\norm{\theta}_{1} = \sum_{j=1}^{\pp} \vert\theta_{j}\vert\ ,\ \theta \in
\mathbb{R}^{\pp} \text{ such that } \ft = \theta.\vp\right\}\end{equation} and
thus belongs to $\ludp$. We propose to estimate $\f$ by a penalized least
squares estimator as introduced at \eqref{plse} with a penalty
$\operatorname*{pen}(\ft)$ proportional to $\normup{h}$. This estimator is the
so-called \lasso\ estimator $\fcp$ defined by
\begin{equation}\label{deflassosection3}\fcp := \fcp(\lambdap)  = \argmin_{\ft \in \ludp}\ \crit{\ft} + \lambdap
\normup{\ft},\end{equation} where $\lambdap>0$ is a regularization parameter
and $\gamma(h)$ is defined by \eqref{defgammah}.

\begin{remark}\droit{Let us notice that the general definition \eqref{deflassosection3}
coincides with the usual definition of the \lasso\ in the particular case of
the classical fixed design Gaussian regression model presented in Example
\ref{excglrm},}
$$Y_{i} = \f(x_{i}) + \sigma \e, \quad i=1,\dots,n.$$
\droit{Indeed, if we define $y=\left(Y_{1},\dots,Y_{n}\right)^T$, we have}
$$\gamma(\ft) = -2 Y(\ft) + \norm{\ft}^2 = -2 \langle y,\ft\rangle +
\norm{\ft}^2 = \norm{y-h}^2 - \norm{y}^2,$$ \droit{so we deduce from
\eqref{deflassosection3} that the \lasso\ satisfies}
\begin{equation}\label{egalite} \fcp =
\argmin_{\ft \in \ludp} \left(\norme{y - \ft}^{2} + \lambdap
\normup{\ft}\right).\end{equation} \droit{Let us now consider for all $\ft \in
\ludp$, $\Theta_{\ft} :=\{\theta = (\theta_{1},\dots,\theta_{p}) \in
\mathbb{R}^{p},\ \ft = \theta.\vp = \sum_{j=1}^{p} \theta_{j}\, \vp_{j}\}$.
Then, we get from \eqref{defnormemodelpu} that}
\begin{align*}  \inf_{\ft \in \ludp} \left(\norme{y - \ft}^{2} + \lambdap
\normup{\ft}\right) & = \inf_{\ft \in \ludp} \left(\norme{y - \ft}^{2}
+ \lambdap \inf_{\theta \in \Theta_{\ft}} \norm{\theta}_{1}\right)\\ & =
\inf_{\ft \in \ludp} \inf_{\theta \in \Theta_{\ft}} \left(\norme{y -
\ft}^{2} + \lambdap
\norm{\theta}_{1}\right)\\
& = \inf_{\ft \in \ludp} \inf_{\theta \in \Theta_{\ft}} \left(\norme{y
- \theta.\vp}^{2} + \lambdap \norm{\theta}_{1}\right)\\ & = \inf_{\theta
\in \mathbb{R}^p} \left(\norme{y - \theta.\vp}^{2} + \lambdap
\norm{\theta}_{1}\right).
\end{align*} \droit{Therefore, we get from \eqref{egalite} that $\fcp =
\hat{\theta}_{p}.\vp$  where $\hat{\theta}_{p} = \argmin_{\theta \in
\mathbb{R}^p} \norme{y - \theta.\vp}^{2} + \lambdap \norm{\theta}_{1},$ which
corresponds to the usual definition of the \lasso\ estimator for the fixed
design Gaussian regression models with finite dictionaries of size $p$ (see
\cite{bic-tsy} for instance).}\end{remark}

%%%%%%%%%%%%%%%%%%%%%%%%%%%%%%%%%%%%%%%%%%%%%%%%%%%%%%%%%%%%%%%%%%%%%%%%%%%%%%%%%%%%%%%%%%%%%
\subsection{The $\ell_{1}$-oracle inequality}\label{soussectionlassofini}

Let us now state the main result of this section.
\begin{theorem}
\label{zlasso} Assume that $\max_{j=1,\dots,p}\Vert\pj\Vert \leq 1$
 and that
\begin{equation}\label{inegp} \lambdap \geq 4 \eps \left(\sqrt{\logn p} +
1\right).\end{equation} Consider the corresponding Lasso estimator $\fcp$
defined by \eqref{deflassosection3}.

\n Then, there exists an absolute positive constant $C$ such that, for all $z
>0$,  with probability larger than
$1-3.4\, \text{\upshape{e}}^{-z}$,
\begin{equation}\label{inegaliteenproba}\norm{\f-\fcp}^{2} + \lambdap \normup{\fcp} \leq C
\left[ \inf_{\ft\in \ludp}\left(
\norm{\f-\ft}^{2}+\lambdap \normup{\ft}\right)
+\lambdap\, \varepsilon(1+z)\right].\end{equation}
Integrating \eqref{inegaliteenproba} with respect to $z$ leads to the following
$\ell_{1}$-oracle type inequality in expectation,
\begin{equation}\label{inegaliteoraclecasfini}\mathbb{E}\left[\norm{\f-\fcp}^{2} +
\lambdap
\normup{\fcp}\right] \leq  C \left[ \inf_{\ft \in \ludp} \left(
\norm{\f-\ft}^{2}+\lambdap \normup{\ft}\right) +\lambdap
\eps\right].\end{equation}
\end{theorem}

This $\ell_{1}$-oracle type inequality highlights the fact that  the \lasso\
(i.e.\!\! the ``noisy'' \lasso) behaves almost as well as the deterministic
\lasso\ provided that the regularization parameter $\lambdap$ is properly
chosen. The proof of Theorem \ref{zlasso} is detailed in Section \ref{proofs}
and we refer the reader to Section \ref{mst} for the description of the key
observation that has enabled us to establish it. In a nutshell, the basic idea
is to view the \lasso\ as the solution of a penalized least squares model
selection procedure over a countable collection of models consisting of
$\ell_{1}$-balls. Inequalities \eqref{inegaliteenproba} and
\eqref{inegaliteoraclecasfini} are thus deduced from a general model selection
theorem borrowed from \cite{blm} and presented in Section \ref{mst} as Theorem
\ref{zmaingaussnl}.

\begin{remark}\droit{\ }

\begin{enumerate}
\droit{\item Notice that unlike the sparsity oracle inequalities with
$\ell_{0}$-penalty esta\-blished by many authors (\cite{bic-tsy},
\cite{vdg1}, \cite{koltchinskii} among others), the above result does not
require any assumption neither on the target function $\f$ nor on the
structure of the variables $\vp_{j}$ of the dictionary $\mathcal{D}_{p}$,
except simple norma\-lization that we can always assume by considerating
$\vp_{j}/\norm{\vp_{j}}$ instead of~$\vp_{j}$.} \droit{\item Although such
$\ell_{1}$-oracle type inequalities have already been studied by a few
authors, no such general risk bound has yet been put forward. Indeed,
Barron and al.\! \cite{bar-huang} have provided a risk bound like
\eqref{inegaliteoraclecasfini} but they restrict to the case of a truncated
\lasso\ estimator under the assumption that the target function is bounded
by a constant. For their part, Rigollet and Tsybakov \cite{rig-tsy} are
proposing an oracle inequality for the \lasso\, similar to
\eqref{inegaliteenproba} which is valid under the same assumption as the
one of Theorem~\ref{zlasso}, i.e.\! simple normalization of the variables
of the dictionary, but their bound in probability can not be integrated to
get an bound in expectation as the one we propose at
\eqref{inegaliteoraclecasfini}. Indeed, first notice that the constant
measuring the level of confidence of their risk bound appears inside the
infimum term as a multiplicative factor of the $\ell_{1}$-norm whereas the
constant $z$ measu\-ring the level of confidence of our risk bound
\eqref{inegaliteenproba} appears as an additive constant outside the
infimum term so that the bound in probability \eqref{inegaliteenproba} can
easily be integrated with respect to $z$, which leads to the bound in
expectation \eqref{inegaliteoraclecasfini}. Besides, the main drawback of
the result given by Tsybakov and Rigollet is that the lower bound of the
regularization parameter~$\lambdap$ they propose (i.e.\! $\lambdap \geq
\sqrt{8(1+z/ \ln p)}\, \eps \sqrt{\ln  p}$) depends on the level of
condidence $z$, with the consequence that their choice of the \lasso\
estimator $\fcp = \fcp(\lambdap)$ also depends on this level of confidence.
On the contrary, our lower bound $\lambdap \geq 4\eps(\sqrt{\ln p}+1)$ does
not depend on $z$ so that we are able to get the result
\eqref{inegaliteenproba} satisfied with high probability by an estimator
$\fcp=\fcp(\lambdap)$ independent of the level of confidence of this
probability.} \droit{\item Theorem \ref{zlasso} is interesting from the
point of view of approximation theory. Indeed, as we shall see in
Proposition \ref{propnvestlasso}, it shows that the \lasso\ performs as
well as the greedy algorithms studied in \cite{bcdd} and \cite{bar-huang}.}
\droit{\item We can check that the upper bound
\eqref{inegaliteoraclecasfini}  is sharp. Indeed, assume that $p\geq 2$,
that $\f \in \ludp$ with $\normup{\f} \leq R$ and that $R\geq \eps$.
Consider the \lasso\ estimator $\fcp$ for $\lambdap = 4 \eps (\sqrt{\ln
p}+1)$. Then, by bounding the infimum term in the right-hand side of
\eqref{inegaliteoraclecasfini} by the value at $h=\f$, we get that
\begin{equation}\label{sal}\mathbb{E}\left[\norm{\f-\fcp}^{2}\right] \leq C \lambdap
\left(\normup{\f}+\eps\right) \leq 8 C R \eps \left(\sqrt{\ln
p}+1\right),\end{equation} where $C>0$.} \droit{Now, it is established in
Proposition 5 in \cite{BM99g} that there exists $\kappa>0$ such that the
minimax risk over the $\ell_{1}$-balls $S_{R,p} = \{h\in \ludp,
\normup{h}\leq R\}$ satisfies
\begin{equation}\label{tous}\inf_{\tilde{\ft}} \sup_{\ft \in S_{R,p}}
\mathbb{E}\left[\norm{\ft-\tilde{\ft}}^{2}\right] \geq \kappa \inf\left(R
\eps \sqrt{1+\ln\left(p \eps R^{-1}\right)},\, p \eps^{2},\,
R^{2}\right),\end{equation}}\droit{where the infimum is taken over all
possible estimators $\tilde{\ft}$. Comparing the upper bound \eqref{sal} to
the lower bound \eqref{tous}, we see that the ratio between them is bounded
independently of $\eps$ for all $S_{R,p}$ such that the signal to noise
ratio $R\eps^{-1}$ is between $\sqrt{\ln p}$ and $p$. This proves that the
\lasso\ estimator $\fcp$ is approximately minimax over such sets
$S_{R,p}$.}
\end{enumerate}
\end{remark}

%%%%%%%%%%%%%%%%%%%%%%%%%%%%%%%%%%%%%%%%%%%%%%%%%%%%%%%%%%%%%%%%%%%%%%%%%%%%%%%%%%%%%%
%%%%%%%%%%%%%%%%%%%%%%%%%%%%%%%%%%%%%%%%%%%%%%%%%%%%%%%%%%%%%%%%%%%%%%%%%%%%%%%%%%%%%
\section{A \lassotype estimator for infinite coun\-table dictionaries}\label{section3}

In many applications such as micro-array data analysis or signal
reconstruction, we are now faced with situations in which the number of
variables of the dictionary is always increasing and can even be infinite.
Consequently, it is desirable to find competitive estimators for such infinite
dimensional problems. Unfortunately, the \lasso\ is not well adapted to
infinite dictionaries. Indeed, from a practical point of view, there is no
algorithm to approximate the \lasso\ solution over an infinite dictionary
because it is not possible to evaluate the infimum of
$\gamma(h)+\p\normud{\ft}$ over the whole set~$\lud$ for an infinite dictionary
$\mathcal{D}$, but only over a finite subset of it. Moreover, from a
theoretical point of view, it is difficult to prove good results on the \lasso\
for infinite dictionaries, except in rare situations when the variables have a
specific structure (see Section \ref{section5} on neural networks).

\medskip In order to deal with an infinite countable dictionary $\mathcal{D}$,
one may order the variables  of the dictionary, write the dictionary
$\mathcal{D}=\{\vp_{j}\}_{j\in \mathbb{N}^*}=\{\vp_{1},\vp_{2},\dots\}$
according to this order, then truncate~$\mathcal{D}$ at a given level $p$ to
get a finite subdictionary $\{\vp_{1},\dots,\vp_{p}\}$ and finally estimate the
target function by the \lasso\ estimator $\fcp$ over this subdictionary. This
procedure implies two difficulties. First, one has to put an order on the
variables of the dictionary, and then all the matter is to decide at which
level one should truncate the dictionary to make the best tradeoff between
approximation and complexity. Here, our purpose is to resolve this last dilemma
by proposing a \lassotype estimator based on an algorithm choosing
automatically the best level of truncation of the dictionary once the variables
have been ordered. Of course, the algorithm and thus the estimation of the
target function will depend on which order the variables have been classified
beforehand. Notice that the classification of the variables can reveal to be
more or less difficult according to the problem under consideration.
Nonetheless, there are a few applications where there may be an obvious order
for the variables, for instance in the case of dictionaries of wavelets.

\medskip  In this section, we shall first introduce the \lassotype estimator that we propose
to approxi\-mate the target function in the case of infinite countable
dictionaries. Then, we shall provide an oracle inequality satisfied by this
estimator. This inequality is to be compared to Theorem~\ref{zlasso}
established for the \lasso\ in the case of finite dictionaries. Its proof is
again an application of the general model selection Theorem \ref{zmaingaussnl}.
Finally, we make a few comments on the possible  advantage of using this
\lassotype estimator for finite dictionaries in place of the classical \lasso\
estimator.

%%%%%%%%%%%%%%%%%%%%%%%%%%%%%%%%%%%%%%%%%%%%%%%%%%%%%%%%%%%%%%%%%%%%%%%%%%%%%
\subsection{Definition of the \lassotype estimator}\label{soussectionlassotype}

We still consider the generalized linear Gaussian model and the statistical
pro\-blem~\eqref{statpbg} introduced in Section \ref{section1}. We recall that,
to solve this problem, we use a dictionary $\mathcal{D}=\{\vp_{j}\}_{j}$ and
seek for an estimator $\fc = \hat{\theta}.\vp = \sum_{j,\, \vp_{j} \in
\mathcal{D}} \hat{\theta_{j}}\, \vp_{j}$ solution of the penalized risk
minimization problem,
\begin{equation}\label{sour}\fc \in \argmin_{\ft \in \lud}
\gamma(\ft) + \operatorname*{pen}(\ft),\end{equation} where
$\operatorname*{pen}(\ft)$ is a suitable positive penalty. Here, we assume that
the dictionary is infinite countable and that it is ordered,
$$\mathcal{D}=\{\vp_{j}\}_{j\in \mathbb{N}^*}=\{\vp_{1},\vp_{2},\dots\}.$$ Given this order, we
can consider the sequence of truncated dictionaries
$\left(\mathcal{D}_{p}\right)_{p \in \ensn}$ where
\begin{equation}\label{defdicotronque}\mathcal{D}_{\pp} := \{\vp_{1},\dots,\vp_{\pp}\}\end{equation} corresponds to
the subdictionary of $\mathcal{D}$ truncated at level $p$, and the associated
sequence of \lasso\ estimators $(\fcp)_{\pp \in \ensn}$ defined in Section
\ref{sectiondeflassocasfini},
\begin{equation}\label{lassoveu}\fcp := \fcp(\lambdap)  = \argmin_{\ft \in
\ludp}\ \crit{\ft} + \lambdap \normup{\ft},\end{equation} where
$\left(\lambdap\right)_{\pp \in \ensn}$ is a sequence of regularization
parameters whose values will be specified below.  Now, we shall choose a final
estimator as an $\ell_{0}$-penalized estimator among a subsequence of the
\lasso\ estimators $(\fcp)_{\pp \in \ensn}$. Let us denote by $\ens$ the set of
dyadic integers,
\begin{equation}\label{dyadicinteger}\ens = \{2^J, J \in \mathbb{N}\},\end{equation} and define
\begin{align}\label{ire2} \fcf
 & =  \argmin_{\pp \in \ens}\left[ \crit{\fcp} +
\lambdap \normup{\fcp} + \penp\right]
\\ \label{ire} & =  \argmin_{\pp \in \ens} \left[\argmin_{\ft \in
\ludp}\left( \crit{\ft} + \lambdap \normup{\ft}\right)+\penp\right],\end{align}
where $\penp$ penalizes the size $p$ of the truncated dictionary
$\mathcal{D}_{p}$ for all $p \in \ens$. From \eqref{ire} and the fact that
$\lud = \cup_{\pp\in \ens} \ludp$, we see that this \lassotype estimator $\fcf$
is a pena\-lized least squares estimator solution of \eqref{sour} where, for
any $p \in \ens$ and $h\in \ludp$, $\operatorname*{pen}(\ft) = \lambdap
\normup{\ft}+\penp$ is a combination of both $\ell_{1}$-regularization and
$\ell_{0}$-penalization. We see from \eqref{ire2} that the algorithm
automatically chooses the rank $\hat{\pp}$ so that $\fcf$ makes the best
tradeoff between approximation, $\ell_{1}$-regularization and sparsity.

\begin{remark}\droit{Notice that from
a theoretical point of view, one could have defined $\fcf$ as an
$\ell_{0}$-penalized estimator among  the whole sequence of \lasso\ estimators
$(\fcp)_{\pp \in \ensn}$ (or more generally among any subsequence of
$(\fcp)_{\pp \in \ensn}$) instead of $(\fcp)_{\pp \in \ens}$.  Nonetheless, to
compute $\fcf$ efficiently, it is interesting to limit the number of
computations of the sequence of \lasso\ estimators $\fcp$  especially if we
choose an $\ell_{0}$-penalty $\penp$ that does not grow too fast with~$p$,
typically $\penp \propto \logn p$, which will be the case in the next theorem.
That is why we have chosen to consider a dyadic truncation of the dictionary
$\mathcal{D}$.}
\end{remark}

%%%%%%%%%%%%%%%%%%%%%%%%%%%%%%%%%%%%%%%%%%%%%%%%%%%%%%%%%%%%%%%%%%%%%%%%%%%%%%%%%%%%%%%%%%
\subsection{An oracle inequality for the \lassotype estimator}

By applying the same general model selection theorem (Theorem
\ref{zmaingaussnl}) as for the establishment of Theorem \ref{zlasso}, we can
provide a risk bound satisfied by the estimator $\fcf$ with properly chosen
penalties $\lambdap$ and $\penp$ for all $\pp \in \ens$. The sequence of
$\ell_{1}$-regularization parameters $\left(\lambdap\right)_{\pp\in \ens}$ is
simply chosen from the lower bound given by \eqref{inegp}  while a convenient
choice for the $\ell_{0}$-penalty will be $\penp \propto \logn \pp$.
\begin{theorem}
\bigskip\label{zlassocasinfini} Assume that $\sup_{j \in \mathbb{N}^*}\Vert\pj\Vert
\leq1$. Set for all $p \in \ens$,
\begin{equation}\label{condparinf}\lambdap = 4 \eps \left(\sqrt{\logn \pp}
+1\right)\,\text{,}\quad \quad \penp = 5 \eps^2 \logn \pp
\text{,}\end{equation} and consider the corresponding \lassotype estimator $\fcf$ defined by
\eqref{ire}.

\n Then, there exists an absolute constant $C>0$ such that
\begin{align}\non \label{inegaliteoraclecasinfini}
 & \mathbb{E}\left[\norm{\f -\fcf}^{2} + \lambda_{\hat{p}}
\norm{\fcf}_{\mathcal{L}_{1}(\mathcal{D}_{\hat{p}})} +
\operatorname*{pen}(\hat{\pp})\right]\\  \leq {} & C \left[\inf_{\pp
\in \ens} \left( \inf_{\ft\in \ludp} \left(
\norm{\f-\ft}^{2}+\lambdap \normup{\ft}\right) +
\penp\right)+\eps^2\right].
\end{align}
\end{theorem}

\medskip
\begin{remark}\droit{Our primary motivation for introducing the \lassotype estimator described above
was to construct an estimator adapted from the \lasso\ and fitted to solve
problems of estimation dealing with infinite dictionaries. Nonetheless, we can
notice that such a \lassotype estimator remains well-defined and can also be
interesting for estimation in the case of finite dictionaries. Indeed, let
$\mathcal{D}_{p_{0}}$ be a given finite dictionary of size $p_{0}$. Assume for
simplicity that $\mathcal{D}_{p_{0}}$ is of cardinal an integer power of two:
$p_{0} = 2^{J_{0}}$. Instead of working with the \lasso\ estimator defined by}
$$\fcpo = \argmin_{\ft \in \ludpo} \crit{\ft}+\lambdapo \normupo{\ft},$$
\droit{with $\lambdapo = 4\eps\left(\sqrt{\logn p_{0}}+1\right)$ being chosen
from the lower bound of Theorem \ref{zlasso}, one can introduce a sequence of
dyadic truncated dictionaries $ \mathcal{D}_{1} \subset \dots \subset
\mathcal{D}_{\pp}\subset \dots \subset \mathcal{D}_{{p}_{0}},$ and consider the
associated \lassotype estimator defined by}
$$\fcf
  =  \argmin_{\pp \in \enso} \left[\argmin_{\ft \in
\ludp}\left( \crit{\ft} + \lambdap \normup{\ft}\right)+\penp\right],$$
\droit{where $\enso = \{2^J, J=0,\dots, J_{0}\}$ and where the sequences
$\lambdap = 4\eps\left(\sqrt{\logn \pp}+1\right)$ and $\penp = 5\eps^2 \logn
\pp$ are chosen from Theorem \ref{zlassocasinfini}. The estimator $\fcf$ can be
seen as an $\ell_{0}$-penalized estimator among the sequence of \lasso\
estimators $(\fcp)_{\pp \in \enso}$ associated to the truncated dictionaries
$\left(\mathcal{D}_{p}\right)_{\pp \in \enso}$,}
$$\fcp  = \argmin_{\ft \in \ludp}\ \crit{\ft} +
\lambdap \normup{\ft}.$$ \droit{In particular, notice that the \lassotype
estimator $\fcf$ and the \lasso\ estimator~$\fcpo$ coincide when $\hat{p} =
p_{0}$ and that in any case the definition of $\fcf$ guarantees that $\fcf$
makes a better tradeoff between approximation, $\ell_{1}$-regularization and
sparsity than $\fcpo$. Furthermore, the risk bound
\eqref{inegaliteoraclecasinfini} remains satisfied by $\fcf$ for a finite
dictionary $\mathcal{D}_{p_{0}}$ if we replace $\mathcal{D}$ by
$\mathcal{D}_{p_{0}}$ and $\ens$ by $\enso$.}
\end{remark}

%%%%%%%%%%%%%%%%%%%%%%%%%%%%%%%%%%%%%%%%%%%%%%%%%%%%%%%%%%%%%%%%%%%%%%%%%%%%%%%%%%%%%%%%%%%
%%%%%%%%%%%%%%%%%%%%%%%%%%%%%%%%%%%%%%%%%%%%%%%%%%%%%%%%%%%%%%%%%%%%%%%%%%%%%%%%%%%%%%%%%%%%
\section{Rates of convergence of the \lasso\ and \lassotype estimators}\label{section4}

In this section, our purpose is to provide rates of convergence of the \lasso\
and the \lassotype estimators introduced in Section \ref{section2} and Section
\ref{section3}. Since in learning theory one has no or not much a priori
knowledge of the smoothness of the unknown target function $\f$ in the Hilbert
space~$\mathbb{H}$, it is essential to aim at establishing performance bounds
for a wide range of function classes. Here, we shall analyze rates of
convergence whenever $\f$ belongs to some real interpolation space between a
subset of $\lud$ and the Hilbert space $\mathbb{H}$. This will provide a full
range of  rates of convergence related to the unknown smoothness of $\f$. In
particular, we shall prove that both the \lasso\ and the \lassotype estimators
perform as well as the greedy algorithms presented by Barron and al.\!
in~\cite{bcdd}. Furthermore, we shall check that the \lassotype estimator is
simultaneously approximately minimax when the dictionary is an orthonormal
basis of $\mathbb{H}$ for a suitable signal to noise ratio.

\medskip Throughout the section, we keep the same framework as in Section
\ref{soussectionlassotype}. In particular,  $\mathcal{D} = \{\vp_{j}\}_{j\in
\mathbb{N}^*}$ shall be a given infinite countable ordered dictionary. We
consider the sequence of truncated dictionaries
$\left(\mathcal{D}_{\pp}\right)_{\pp \in \ensn}$ defined
by~\eqref{defdicotronque}, the associated sequence of \lasso\ estimators
$(\fcp)_{\pp \in \ensn}$ defined by \eqref{lassoveu} and the \lassotype
estimator $\fcf$ defined by \eqref{ire} with $\lambdap = 4\eps(\sqrt{\ln p}+1)$
and $\penp = 5 \eps^2 \logn p$  and where $\ens$ still denotes the set of
dyadic integers defined by~\eqref{dyadicinteger}.

\medskip The rates of convergence for the sequence of the \lasso\ and the
\lassotype estimators will be derived from the oracle inequalities established
in Theorem~\ref{zlasso} and Theorem \ref{zlassocasinfini} respectively. We know
from Theorem \ref{zlasso} that, for all $p\in \ensn$, the quadratic risk of the
\lasso\ estimator $\fcp$ is bounded by
\begin{equation}\label{majriskl}\mathbb{E}\left[\norm{\f-\fcp}^{2}\right] \leq C\left[ \inf_{\ft
\in \ludp} \left( \norm{\f-\ft}^{2}+\lambdap \normup{\ft}\right) +\lambdap
\eps\right],\end{equation} where $C$ is an absolute positive constant, while we
know from Theorem \ref{zlassocasinfini} that the quadratic risk of the
\lassotype estimator $\fcf$ is bounded by
\begin{equation}\label{majrisklt}\mathbb{E}\left[\norm{\f -\fcf}^{2} \right] \leq C
\left[\inf_{p \in \ens} \left( \inf_{\ft\in \ludp} \left(
\norm{\f-\ft}^{2}+\lambdap \normup{\ft}\right) +
\operatorname*{pen}(\pp)\right)+\eps^2\right],\end{equation} where $C$ is an
absolute positive constant. Thus, to bound the quadratic risks of the
estimators $\fcf$ and $\fcp$ for all $p \in\ensn$, we can first focus on
bounding for all $\pp \in \ensn$, \begin{equation}\label{justelasdet} \inf_{\ft
\in \ludp}\left(\norm{\f-\ft}^2+\lambdap \normup{\ft}\right) =
\norm{\f-\flp}^{2}+\lambdap \normup{\flp},\end{equation} where we denote by
$\flp$ the deterministic \lasso\ for the truncated dictionary
$\mathcal{D}_{\pp}$ defined by
\begin{equation}\label{deflassodeterministic}\flp = \argmin_{\ft \in \ludp} \left(
\norm{\f-\ft}^{2}+\lambdap \normup{\ft}\right).\end{equation} This first step
will be handled in Section \ref{interspaces} by considering suitable
interpolation spaces. Then in Section \ref{rateest}, we shall pass on the rates
of convergence of the deterministic \lasso s to the \lasso\ and the \lassotype
estimators thanks to the upper bounds \eqref{majriskl} and \eqref{majrisklt}.
By looking at these upper bounds, we can expect the \lassotype estimator to
achieve much better rates of convergence than the \lasso\ estimators. Indeed,
for a fixed value of $p \in \ensn$, we can see that the risk of the \lasso\
estimator $\fcp$ is roughly of the same order as the rate of convergence of the
corresponding deterministic \lasso\ $\flp$, whereas the risk of $\fcf$ is
bounded by the infimum over \emph{all} $p \in \ens$ of penalized rates of
convergence of the deterministic \lasso s $\flp$.

\subsection{Interpolation spaces}\label{interspaces}

Remember that we are first looking for an upper bound of $\inf_{\ft \in
\ludp}\left(\norm{\f-\ft}^2+\right.$ $\left.\lambdap\normup{\ft}\right)$ for
all $p \in \ensn$. In fact, this quantity is linked to another one in the
approximation theory, which is the so-called $K_{\mathcal{D}_{p}}$-functional
defined below. This link is specified in the following essential lemma.

\begin{lemma}\label{propgenerale} Let $D$ be some finite or infinite
dictionary. For any $\p \geq 0$ and $\delta >0$, consider $$L_{D}(\f,\p) :=
\inf_{\ft \in \mathcal{L}_{1}(D)} \left(\norm{\f-\ft}^2+\p
\norm{\ft}_{\mathcal{L}_{1}(D)}\right)$$ and the $K_{D}$-functional defined by
\begin{equation}\label{defkdfunct}K_{D}(\f,\delta) := \inf_{\ft \in \mathcal{L}_{1}(D)}
\left(\norm{\f-\ft}+\delta\norm{\ft}_{\mathcal{L}_{1}(D)}\right).\end{equation}
 Then,
\begin{equation}\label{eq1}
\frac{1}{2}\, \inf_{\delta
>0}\left(K_{D}^2(\f,\delta)+\frac{\p^2}{2\delta^2}\right) \leq
 L_{D}(\f,\p) \leq \inf_{\delta
>0}\left(K_{D}^2(\f,\delta)+\frac{\p^2}{4\delta^2}\right).\end{equation}
\end{lemma}

\medskip Let us now introduce a whole range of interpolation spaces $\bqr$ that
are intermediate spaces between subsets of $\lud$ and the Hilbert space
$\mathbb{H}$ on which the $K_{\mathcal{D}_{p}}$-functionals (and thus the rates
of convergence of the deterministic \lasso s $\flp$) are controlled for all $p
\in \ensn$.
\begin{definition}\label{defbqr}\droit{\textbf{[Spaces $\lur$ and $\bqr$]}}
Let $R>0$, $r>0$, $1<q<2$ and $\pa = 1/q-1/2$.

\n We say that a function $g$ belongs to the space $\lur$ if there exists $C>0$
such that for all $\pp \in \ensn$, there exists $g_{\pp} \in \ludp$ such that
$$\normup{g_{\pp}} \leq C$$
and \begin{equation}\label{lus}\norm{g-g_{\pp}} \leq C\, \vert
\mathcal{D}_{p}\vert^{-r} = C \pp^{-r}.\end{equation}
 The smallest $C$ such that this holds defines a
norm $\normlur{g}$ on the space $\lur$.

\smallskip \n We say that $g$ belongs to $\bqr(R)$  if, for all $\delta >0$,
\begin{equation}\label{kflur} \inf_{\ft \in \lur}
\left(\norm{g-\ft}+\delta\normlur{\ft}\right) \leq R\,
\delta^{2\pa}.\end{equation} We say that $g \in \bqr$ if there exists $R
>0$ such that $g \in \bqr(R)$. In this case, the smallest $R$ such that $g \in \bqr(R)$
defines a norm on the space $\bqr$ and is denoted by $\normbqr{g}$.
\end{definition}

\begin{remark}\droit{Note that the spaces $\lur$ and $\bqr$ depend on
the choice of the whole dictionary $\mathcal{D}$ as well as on the way it is
ordered, but we shall omit this dependence so as to lighten the notations. The
set of spaces $\lur$ can be seen as substitutes for the whole space $\lu$ that
are adapted to the truncation of the dictionary. In particular, the spaces
$\lur$ are smaller than the space $\lu$ and the smaller the value of $r>0$, the
smaller the distinction between them. In fact, looking at \eqref{lus}, we can
see that working with the spaces $\lur$ rather than $\lu$ will enable us to
have a certain amount of control (measured by the parameter $r$) as regards
what happens beyond the levels of truncation.}
\end{remark}

\medskip
Thanks to the property of the interpolation spaces $\bqr$ and to
 the equivalence established in Lemma~\ref{propgenerale}
between the rates of convergence of the deterministic \lasso s and the
$K_{\mathcal{D}_{p}}$-functional, we are now able to provide the following
upper bound of the rates of convergence of the deterministic \lasso s when the
target function belongs to some interpolation space $\bqr$.
\begin{lemma}\label{majinf} Let $1<q<2$, $r>0$ and $R>0$. Assume that $\f \in \bqr(R)$.

\n Then, there exists $C_{q}>0$ depending only on $q$ such that, for all $\pp
\in \ensn$, \begin{equation}\label{colless}\inf_{\ft\in \ludp} \left(
\norm{\f-\ft}^{2}+\lambdap \normup{\ft}\right)  \leq C_{q} \max\left(R^{q}
\lambdap^{2-q}\, , \, \left(R p^{-r}\right)^{\frac{2q}{2-q}}
\lambdap^{\frac{4(1-q)}{2-q}}\right).\end{equation}
\end{lemma}

\begin{remark}\label{eqcocnor}\droit{\textbf{[Orthonormal case]} Let us point out that
the abstract interpolation spaces $\bqr$ are in fact natural extensions to
non-orthonormal dictionaries of function spaces that are commonly studied in
statistics to analyze the approximation performance of estimators in the
orthonormal case, that is to say Besov spaces, strong-$\mathcal{L}_{q}$ spaces
and weak-$\mathcal{L}_{q}$ spaces. More precisely, recall that if $\mathbb{H}$
denotes a Hilbert space and $\mathcal{D} = \{\vp_{j}\}_{j \in \mathbb{N}^*}$ is
an orthonormal basis of $\mathbb{H}$, then, for all $r>0$, $q>0$ and $R>0$, we
say that $g=\sum_{j=1}^{\infty}\theta_{j}\, \vp_{j}$ belongs to the Besov space
$\besov(R)$ if}
\begin{equation}\label{defbesovr}\sup_{J \in
\mathbb{N}^*}\left(J^{2r}\sum_{j=J}^{\infty} \theta_{j}^{2}\right) \leq
R^2,\end{equation} \droit{while $g$ is said to belong to $\swlq(R)$ if}
\begin{equation}\label{deflqr}\sum_{j=1}^{\infty} \vert\theta_{j}\vert^q \leq R^{q},\end{equation}
\droit{and a slightly weaker condition is that $g$ belongs to $\wlq(R)$, that
is to say}
\begin{equation}\label{defwlqr}\sup_{\eta>0} \left(\eta^q \sum_{j=1}^{\infty}
\mathds{1}_{\{\vert\theta_{j}\vert>\eta\}}\right) \leq  R^{q}.\end{equation}
\droit{Then, we prove in Section \ref{proofs} that for all $1<q<2$ and $r>0$,
there exists $C_{q,r}>0$ depending only on $q$ and $r$ such that the following
inclusions of spaces hold for all $R>0$ when $\mathcal{D}$ is an orthonormal
basis of $\mathbb{H}$:}
\begin{equation}\label{eqcocno}\swlq(R)\cap\besov(R)\, \subset \, \wlq(R) \cap \besov(R)\, \subset
\, \bqr(C_{q,r}\, R).\end{equation} \droit{In particular, these inclusions shall
turn out to be useful to check the optimality of the rates of convergence of
the \lassotype estimator in Section \ref{lowb}.}
\end{remark}

%%%%%%%%%%%%%%%%%%%%%%%%%%%%%%%%%%%%%%%%%%%%%%%%%%%%%%%%%%%%%%%%%%%%%%%%%%%%%%%%%%%%%%
\subsection{Upper bounds of the quadratic risk of the
estimators}\label{rateest}

The rates of convergence of the deterministic \lasso s $\flp$ given in Lemma
\ref{majinf} can now be passed on to the \lasso\ estimators $\fcp$,
$p\in\ensn$, and to the \lassotype estimator $\fcf$ thanks to the oracle
inequalities \eqref{majriskl} and \eqref{majrisklt} respectively.

\smallskip
\begin{proposition}\label{propnvestlasso} Let $1<q<2$, $r>0$ and $R>0$.
Assume that $\f \in \bqr(R)$.

\n Then, there exists $C_{q}>0$ depending only on $q$ such that, for all $p\in
\ensn$,
\begin{itemize}
\item if $\left(\sqrt{\ln p}+1\right)^{\frac{q-1}{q}} \leq R \eps^{-1} \leq
    p^{\frac{2r}{q}} \left(\sqrt{\ln p}+1\right)$, then
\begin{equation}\label{onzes}\mathbb{E}\left[\norm{\f -\fcp}^{2} \right]  \leq
C_{q}\,
R^{q}
\left(\eps \left(\sqrt{\ln p}+1\right)\right)^{2-q},\end{equation}
\item if $R \eps^{-1} > p^{\frac{2r}{q}} \left(\sqrt{\ln p}+1\right)$, then
\begin{equation}\label{douzes}\mathbb{E}\left[\norm{\f -\fcp}^{2} \right]  \leq  C_{q}
\left(R\, p^{-r}\right)^{\frac{2q}{2-q}}\, \left(\eps \left(\sqrt{\ln
p}+1\right)\right)^{\frac{4(1-q)}{2-q}},\end{equation}
\item if $R \eps^{-1} <\left(\sqrt{\ln p}+1\right)^{\frac{q-1}{q}}$, then
\begin{equation}\label{treizes}\mathbb{E}\left[\norm{\f -\fcp}^{2} \right]  \leq
C_{q}\,
\varepsilon^2 \left(\sqrt{\logn
p}+1\right).\end{equation} \end{itemize}\end{proposition}

\smallskip
\begin{proposition}\label{propnvest} Let $1<q<2$ and $r>0$. Assume that $\f \in
\bqr(R)$ with $R>0$ such that $R\eps^{-1}\geq
\max\left(\text{\upshape{e}},(4r)^{-1}q\right)$.

\n Then,  there exists $C_{q,r}>0$ depending only on $q$ and $r$ such that the
quadratic risk of $\fcf$ satisfies
\begin{equation}\label{quatorzes}\mathbb{E}\left[\norm{\f -\fcf}^{2} \right]  \leq  C_{q,r}\,
R^q \left(\eps\sqrt{\ln\left(R\eps^{-1}\right)}\right)^{2-q}.\end{equation}
\end{proposition}

\begin{remark} \droit{\ }
\begin{enumerate}
\droit{\item Notice that the assumption  $R \eps^{-1} \geq
\max\left(\text{\upshape{e}},(4r)^{-1}q\right)$ of Proposition
\ref{propnvest} is not restrictive since it only means that we consider
non-degenerate situ\-ations when the signal to noise ratio is large enough,
which is the only interesting case to use the \lassotype estimator. Indeed,
if $R \eps^{-1}$ is too small, then the estimator equal to zero will always
be better than any other non-zero estimators, in particular \lasso\
estimators.} \droit{\item Proposition \ref{propnvest} highlights the fact
that the \lassotype estimator can simultaneously achieve rates of
convergence of order $\left(\eps \sqrt{\ln\left(\normbqr{\f}
\eps^{-1}\right)}\right)^{2-q}$ for all classes $\bqr$ without knowing
which class contains $\f$. Besides, comparing the upper bound
\eqref{quatorzes} to the lower bound \eqref{bag} established in the next
section for the minimax risk when the dictionary~$\mathcal{D}$ is an
orthonormal basis of $\mathbb{H}$ and $r < 1/q-1/2$, we see that they can
match up to a constant if the signal to noise ratio is large enough. This
proves that the rate of convergence \eqref{quatorzes} achieved by $\fcf$ is
optimal.} \droit{\item Analyzing the different results of Proposition
\ref{propnvestlasso}, we can notice that, unlike the \lassotype estimator,
the \lasso\ estimators are not adaptative. In particular, comparing
\eqref{onzes} to \eqref{quatorzes}, we see that the \lasso s $\fcp$ are
likely to achieve the optimal rate of convergence \eqref{quatorzes} only
for $p$ large enough, more precisely $p$ such that $R \eps^{-1} \leq
p^{2r/q}(\sqrt{\ln p}+1)$. For smaller values of $p$, truncating the
dictionary at level $p$ affects the rate of convergence as it is shown at
\eqref{douzes}. The problem is that $q$ and $r$ are unknown since they are
the parameters characterizing the smoothness of the unknown target
function. Therefore, when one chooses a level $p$ of truncation of the
dictionary, one does not know if $R \eps^{-1} \leq p^{2r/q}(\sqrt{\ln
p}+1)$ and thus if the corresponding \lasso\ estimator $\fcp$ has a good
rate of convergence. When working with the \lasso s, the statistician is
faced with a dilemma since one has to choose $p$ large enough to get an
optimal rate of convergence, but the larger $p$ the less sparse and
interpretable the model. The advantage of using the \lassotype estimator
rather than the \lasso s is that, by construction of $\fcf$, we are sure to
get an estimator making the best tradeoff between approximation,
$\ell_{1}$-regularization and sparsity and achieving desirable rates of
convergence for any target function belonging to some interpolation space
$\bqr$.} \droit{\item Looking at the different results from \eqref{onzes}
to \eqref{quatorzes}, we can notice that the parameter $q$ has much more
influence on the rates of convergence than the parameter $r$ since the
rates are of order depending only on the parame\-ter~$q$ while the
dependence on $r$ appears only in the multiplicative factor. Nonetheless,
note that the smoother the target function with respect to the parameter
$r$, the smaller the number of variables necessary to keep to get a good
rate of convergence for the \lasso\ estimators. Indeed, on the one hand, it
is easy to check that $\bqr(R) \subset \mathcal{B}_{q,r'}(R)$ for $r
> r' >0$ which means that the
smoothness of $\f$ increases with $r$, while on the other hand,
$p^{2r/q}(\sqrt{\ln p}+1)$ increases with respect to $r$ so that the larger
$r$ the smaller $p$ satisfying the constraint necessary for the \lasso\
$\fcp$ to achieve the optimal rate of convergence \eqref{onzes}.}
\droit{\item Proposition \ref{propnvestlasso} shows that the \lasso s
$\fcp$ perform as well as the greedy algorithms stu\-died by Barron and
al.\! in \cite{bcdd}. Indeed, in the case of the fixed design Gaussian
regression model introduced in Example~\ref{excglrm} with a sample of size
$n$, we have $\eps = \sigma/\sqrt{n}$ and \eqref{onzes} yields that the
\lasso\ estimator $\fcp$ achieves a rate of convergence of order $R^{q}
\left(n^{-1}\ln p\right)^{1-q/2}$ provided that $R \eps^{-1}$ is
well-chosen, which corresponds to the rate of convergence esta\-blished by
    Barron and al.\! for the greedy algorithms. Similarly to our result,
Barron and al.\! need to assume that the dictionary is large enough so as
to ensure such rates of convergence. In fact, they consider truncated
dictionaries of size $p$ greater than $n^{1/(2r)}$ with $n^{1/q-1/2} \geq
\normbqr{\f}$. Under these assumptions, we recover the upper bound we
impose on $R\eps^{-1}$  to get the rate \eqref{onzes}.}\end{enumerate}
\end{remark}

\smallskip
\begin{remark}\droit{\textbf{[Orthonormal case]}}
\begin{enumerate}
\droit{\item Notice that the rates of convergence provided for the \lasso\
    estimators in Proposition \ref{propnvestlasso} are a generalization to
    non-orthonormal dictionaries of the well-known performance bounds of
    soft-thres\-holding estimators in the orthonormal case. Indeed, when
    the dictionary $\mathcal{D} = \{\vp_{j}\}_{j}$ is an orthonormal basis
    of $\mathbb{H}$, if we set $\Theta_{p} := \left\{\theta
    =(\theta_{j})_{j\in \mathbb{N}^*},\ \theta =
 \left(\theta_{1},\dots,\theta_{\pp},0,\dots,0,\dots\right)\right\}$ and
calculate the subdifferential of the function $\theta \in \Theta_{p}
\mapsto \crit{\theta.\vp}+\lambdap\norm{\theta}_{1}$, where the function
$\gamma$ is defined by \eqref{defgammah}, we easily get that $\fcp =
\hat{\theta}_{p}.\vp$ with $\hat{\theta}_{p} =
(\widehat{\theta}_{p,1},\dots,\widehat{\theta}_{p,p},0,\dots,0,\dots)$
where for all $j=1,\dots,\pp$,}
$$\widehat{\theta}_{p,j}\ = \
 \left\{\begin{array}{ll}
Y(\vp_{j}) - \lambdap/2 & \quad \text{\droit{if} } Y(\vp_{j}) > \lambdap/2 = 2\eps\left(\sqrt{\ln p}+1\right),\\
Y(\vp_{j}) + \lambdap/2 & \quad \text{\droit{if} } Y(\vp_{j}) < -\lambdap/2 = 2\eps\left(\sqrt{\ln p}+1\right),\\
0 & \quad \text{\droit{else}} ,\end{array}\right.$$ \droit{where $Y$ is
defined by \eqref{statpbg}. Thus, the \lasso\ estimators $\fcp$ correspond
to soft-thresholding estimators with thresholds of order $\eps\sqrt{\ln
p}$, and Proposition \ref{propnvestlasso} together with the inclusions of
spaces \eqref{eqcocno} enable to recover the well-known rates of
convergence of order $(\eps \sqrt{\ln p})^{2-q}$ for such thres\-holding
estimators when the target function belongs to $\wlq\cap\besov$ (see for
instance~\cite{cohendevorepicard} for the establishment of such rates of
convergence for estimators based on wavelet thresholding in the white noise
framework).} \droit{\item Let us stress that, in the orthonormal case,
since the \lasso\ estimators $\fcp$ correspond to soft-thresholding
estimators with thresholds of order $\eps\sqrt{\ln p}$, then the \lassotype
estimator $\fcf$ can be viewed as a soft-thresholding estimator with
adapted threshold $\eps\sqrt{\ln \hat{p}}$.}
\end{enumerate} \end{remark}

%%%%%%%%%%%%%%%%%%%%%%%%%%%%%%%%%%%%%%%%%%%%%%%%%%%%%%%%%%%%%%%%%%%%%%%%%%%%%%%%%%%
\subsection{Lower bounds in the orthonormal case}\label{lowb}

To complete our study on the rates of convergence, we propose to establish a
lower bound of the minimax risk in the orthonormal case so as to prove that the
\lassotype estimator is simultaneously approximately minimax over spaces
$\swlq(R)\cap\besov(R)$ in the orthonormal case for suitable signal to noise
ratio $R\eps^{-1}$.

\begin{proposition}\label{minimax} Assume that the dictionary $\mathcal{D}$ is
an orthonormal basis of~$\mathbb{H}$. Let $1<q<2$, $0<r<1/q-1/2$ and $R>0$ such
that $R\eps^{-1}\geq \max(\text{\upshape{e}}^2,u^2)$ where
\begin{equation}\label{defu}u:=\frac{1}{r}-q\left(1+\frac{1}{2r}\right)>0.\end{equation} Then, there exists an
absolute constant $\kappa>0$ such that the minimax risk over $\swlq(\rw) \cap
\besov(\rb)$ satisfies
    \begin{equation}\label{fls}\inf_{\fct} \sup_{\f \in \swlq(\rw) \cap \besov(\rb)}
\mathbb{E}\left[\norm{\f-\fct}^2\right]  \geq   \kappa\, u^{1-\frac{q}{2}}\, \rw^{q} \left(\eps
\sqrt{\ln\left(R \eps^{-1}\right)}\right)^{2-q},\end{equation} where the infimum
is taken over all possible estimators $\fct$.
\end{proposition}

\begin{remark}\droit{\ }

\begin{enumerate}
\droit{\item Notice that the lower bound \eqref{fls} depends much more on
the parameter~$q$ than on the parameter $r$ that only appears as a
    multiplicative factor through the term $u$. In fact, the assumption $\f
    \in \besov(R)$ is just added to the assumption $\f \in \swlq(R)$ in
    order to control the size of the high-level components of $\f$ in the
    orthonormal basis $\mathcal{D}$ (see the proof of Lemma~\ref{gentmam}
    to convince yourself), but this additional parameter of smoothness
    $r>0$ can be taken arbitrarily small and has little effect on the
    minimax risk.}

\droit{\item It turns out that the constraint $r<1/q-1/2$ of Proposition
\ref{minimax} is quite natural. Indeed, assume that $r>1/q-1/2$. Then, on
the one hand it is easy to check that, for all $R>0$, $\besov(R') \subset
\swlq(R)$ with $R' = (1-2^{ru})^{1/q}R$ where $u$ is defined by
    \eqref{defu}, and thus $\swlq(R) \cap \besov(R') = \besov(R')$. On the
    other hand, noticing that
 $R'<R$, we have $\besov(R') \subset \besov(R)$ and thus $\swlq(R) \cap
 \besov(R') \subset \swlq(R) \cap \besov(R)$. Consequently, $\besov(R')
\subset \swlq(R) \cap \besov(R) \subset \besov(R)$, and the intersection
space $\swlq(R) \cap \besov(R)$ is no longer a real intersection between a
strong-$\swlq$ space and a Besov space $\besov$ but rather a Besov space
$\besov$ itself. In this case, the lower bound of the minimax risk is known
to be of order $\eps^{4r/(2r+1)}$ (see \cite{donoho98} for instance), which
is no longer of the form~\eqref{fls}}.
\end{enumerate}
\end{remark}

Now, we can straightforwardly deduce from \eqref{eqcocno} and Proposition
\ref{minimax} the following result which proves that the rate of convergence
\eqref{quatorzes} achieved by the \lassotype estimator is optimal.

\begin{proposition}\label{vitmin}  Assume that the dictionary $\mathcal{D}$ is an
orthonormal basis of~$\mathbb{H}$. Let $1<q<2$, $0<r<1/q-1/2$ and $R>0$ such
that $R\eps^{-1}\geq \max(\text{\upshape{e}}^2,u^2)$ where
$$u:=\frac{1}{r}-q\left(1+\frac{1}{2r}\right)>0.$$
Then, there exists $C_{q,r}>0$ depending only on $q$ and $r$ such that the
minimax risk over $\bqr(R)$ satisfies
    \begin{equation}\label{bag}\inf_{\fct} \sup_{\f \in \bqr(R)}
\mathbb{E}\left[\norm{\f-\fct}^2\right]  \geq   C_{q,r}\, \rw^{q} \left(\eps\sqrt{
\ln\left(R \eps^{-1}\right)}\right)^{2-q},\end{equation}
where the infimum is taken over all possible estimators $\fct$.
\end{proposition}

\begin{remark}\droit{Looking at \eqref{eqcocno}, one could have obtained a result
similar to \eqref{bag} by bounding from below the minimax risk over $\wlq(R)
\cap \besov(R)$ instead of $\swlq(R)\cap\besov(R)$ as it is done in Proposition
\ref{minimax}. We refer the interested reader to Theorem 1 in \cite{rivoirard}
for the establishment of such a result.}
\end{remark}

%%%%%%%%%%%%%%%%%%%%%%%%%%%%%%%%%%%%%%%%%%%%%%%%%%%%%%%%%%%%%%%%%%%%%%%%%%%%%%%%%%
%%%%%%%%%%%%%%%%%%%%%%%%%%%%%%%%%%%%%%%%%%%%%%%%%%%%%%%%%%%%%%%%%%%%%%%%%%%%%%%

\section{The \lasso\ for uncountable dictionaries : neural
networks}\label{section5}

In this section, we propose to provide some theoretical results on the
performance of the \lasso\ when considering some particular infinite
uncountable dictionaries such as those used for neural networks in the fixed
design Gaussian regression models. Of course, there is no algorithm to
approximate the \lasso\ solution for infinite dictionaries, so the following
results are just to be seen as theoretical performance of the \lasso. We shall
provide an $\ell_{1}$-oracle type ine\-quality satisfied by the \lasso\ and
deduce rates of convergence of this estimator whenever the target function
belongs to some interpolation space between $\lud$ and the Hilbert space
$\mathbb{H}=\mathbb{R}^{n}$. These results will again prove that the \lasso\
theoretically performs as well as the greedy algorithms introduced in
\cite{bcdd}.

\medskip In the artificial intelligence field, the introduction of artificial
neural networks have been motivated by the desire to model the human brain by a
computer. They have been applied successfully to pattern recognition (radar
systems, face identification...), sequence recognition (gesture, speech...),
image analysis, adaptative control, and their study can enable the
reconstruction of software agents (in computer, video games...) or autonomous
robots for instance. Artificial neural networks receive a number of input
signals and produce an output signal. They consist of multiple layers of
weighted-sum units, called neurons, which are of the type
\begin{equation}\label{defdicorn}\vp_{\q,\w} :\  \mathbb{R}^d
\mapsto \mathbb{R},\ \ x \mapsto \chi\left(\langle \q,x\rangle +
\w\right),\end{equation} where $\q \in \mathbb{R}^d$, $\w \in \mathbb{R}$ and
$\chi$ is the Heaviside function $\chi(x) = \mathds{1}_{\{x >0\}}$ or more
generally a sigmoid function. Here, we shall restrict to the case of $\chi$
being the Heaviside function. In other words, if we consider the infinite
uncountable dictionary $\mathcal{D}=\{\vp_{\q,\w}\, ;\ \q \in \mathbb{R}^\dd,
\w \in \mathbb{R}\}$, then a neural network is a real-valued function defined
on $\mathbb{R}^d$
 belonging to the linear span of $\mathcal{D}$.

\medskip Let us now consider the fixed design Gaussian regression model
introduced
 in Example \ref{excglrm} with neural network regression function estimators.
 Given a training sequence $\{(x_{1},Y_{1}),\dots,(x_{n},Y_{n})\}$, we assume
 that $Y_{i} = \f(x_{i}) + \sigma \e$ for all $i=1,\dots,n$ and
 we study the \lasso\ estimator over the set of neural network regression function estimators in
 $\lud$,  \begin{equation}\label{dellassorn}\fc:= \fc(\p)  = \argmin_{\ft \in \lud}\ \norm{Y-\ft}^2 +
\p \normu{\ft},\end{equation} where $\p >0$ is a regularization parameter,
$\lud$ is the linear span of $\mathcal{D}$ equipped with the $\ell_{1}$-norm
\begin{equation}\label{defludrn}\norm{\ft}_{\mathcal{L}_{1}(\mathcal{D})} :=
\inf\left\{\norm{\theta}_{1} = \sum_{\q \in \mathbb{R}^d, \w \in \mathbb{R}}
\vert\theta_{\q,\w}\vert, \ \ \ft = \theta.\vp = \sum_{\q \in \mathbb{R}^d, \w
\in \mathbb{R}} \theta_{\q,\w}\, \vp_{\q,\w}\right\}\end{equation}   and
$\norm{Y-\ft}^2 := \sum_{i=1}^{n} \left(Y_{i}-\ft(x_{i})\right)^2/n$ is the
empirical risk of $h$.

\subsection{An $\ell_{1}$-oracle type inequality}

Despite the fact that the dictionary $\mathcal{D}$ for neural networks is
infinite uncounta\-ble, we are able to establish an $\ell_{1}$-oracle type
inequality satisfied by the \lasso\ which is similar to the one provided in
Theorem \ref{zlasso} in the case of a finite dictionary. This is due to the
very particular structure of the dictionary $\mathcal{D}$ which is only
composed of functions derived from the Heaviside function. This property
enables us to achieve theoretical results without truncating the whole
dictionary into finite subdictionaries contrary to the study developed in
Section~\ref{section3} where we considered arbitrary infinite countable
dictionaries (see Remark \ref{lien} for more details). The following
$\ell_{1}$-oracle type inequality is once again a direct application of the
general model selection Theorem \ref{zmaingaussnl} already used to prove both
Theorem~\ref{zlasso} and Theorem \ref{zlassocasinfini}.
\begin{theorem}\label{zlassorn}
Assume that $$\p \geq \frac{28\, \sigma}{\sqrt{n}} \left(\sqrt{\logn
\left((n+1)^{\dd+1}\right)}+4\right).$$ Consider the corresponding Lasso
estimator $\fc$ defined by \eqref{dellassorn}.

\n Then, there exists an absolute constant $C>0$ such that
$$\mathbb{E}\left[\norm{\f-\fc}^{2}
+ \p \normu{\fc}\right] \leq C\left[ \inf_{\ft \in \lu} \left(
\norm{\f-\ft}^{2}+\p \normu{\ft}\right) +\p \frac{\sigma}{\sqrt{n}}\right].$$
\end{theorem}

%%%%%%%%%%%%%%%%%%%%%%%%%%%%%%%%%%%%%%%%%%%%%%%%%%%%%%%%%%%%%%%%%%%%%%%%%%%%%%%%%%%%%%%%
\subsection{Rates of convergence in real interpolation spaces}

We can now deduce theoretical rates of convergence for the \lasso\ from
Theorem~\ref{zlassorn}. Since we do not truncate the dictionary $\mathcal{D}$,
we shall not consider the spaces $\lur$ and $\bqr$ that we introduced in the
last section because they were adapted to the truncation of the dictionary.
Here, we can work with the whole space $\lud$ instead of $\lur$ and the spaces
$\bqr$ will be replaced by bigger spaces~$\bq$ that are the real interpolation
spaces between $\lud$ and $\mathbb{H}=\mathbb{R}^n$.

\begin{definition}\droit{\textbf{[Space $\bq$]}} Let $1<q<2$, $\pa = 1/q-1/2$
and $R>0$. We say that a function $g$ belongs  to $\bq(R)$ if, for all $\delta
>0$,
\begin{equation}\label{kf}\inf_{\ft \in \lud}
\left(\norm{g-\ft}+\delta\normu{\ft}\right) \leq R\,
\delta^{2\pa}.\end{equation} We say that $g \in \bq$ if there exists $R>0$ such
that $g \in \bq(R)$. In this case, the smallest $R$ such that $g\in\bq(R)$
defines a norm on the space $\bq$ and is denoted by $\norm{g}_{\bq}$.
\end{definition}

The following proposition shows that the \lasso\ simultaneously achieves
desirable levels of performance on all classes $\bq$ without knowing which
class contains~$\f$.
\begin{proposition}\label{propcvlassorn} Let $1<q<2$.

\n Assume that $\f \in \bq(R)$ with  $R \geq  \sigma\left[\logn
    \left((n+1)^{d+1}\right)\right]^{\frac{q-1}{2q}}/\sqrt{n}$.

\n Consider the \lasso\ estimator $\fc$ defined by \eqref{dellassorn} with
$$\p = \frac{28\, \sigma}{\sqrt{n}} \left(\sqrt{\logn
\left((n+1)^{\dd+1}\right)}+4\right).$$

\n Then, there exists $C_{q}>0$ depending only on $q$ such that the quadratic
risk of $\fc$ satisfies
$$\mathbb{E}\left[\norm{\f-\fc}^{2}\right] \leq
 C_{q}\, R^{q} \left[\frac{\logn
 \left((n+1)^{d+1}\right)}{n}\right]^{1-\frac{q}{2}}.$$
\end{proposition}

\begin{remark} \droit{Notice that the above rates of convergence are of the same order
as those provided for the \lasso\ in Proposition \ref{propnvestlasso} for a
suitable signal to noise ratio in the case of an infinite countable dictionary
with $\eps = \sigma/\sqrt{n}$. Besides, we recover the same rates of
convergence as those obtained by Barron and al.\! in \cite{bcdd} for the greedy
algorithms when considering neural networks. Notice that our results can be
seen as the analog in the Gaussian framework of their results which are valid
under the assumption that the output variable $Y$ is bounded but not
necessarily Gaussian.}
\end{remark}

%%%%%%%%%%%%%%%%%%%%%%%%%%%%%%%%%%%%%%%%%%%%%%%%%%%%%%%%%%%%%%%%%%%%%%%%%%%%%%%%%%%%%%%%%%%%%
%%%%%%%%%%%%%%%%%%%%%%%%%%%%%%%%%%%%%%%%%%%%%%%%%%%%%%%%%%%%%%%%%%%%%%%%%%%%%%%%%%%%%%%%%%%%
\section{A model selection theorem}\label{mst}

Let us end this paper by describing the main idea that has enabled us to
establish all the oracle inequalities of Theorem \ref{zlasso}, Theorem
\ref{zlassocasinfini} and Theorem \ref{zlassorn} as an application of a single
general model selection theorem, and by presenting this general theorem.

\medskip We keep the notations introduced in Section 2. In particular,
recall that one observes a process $\left(Y(\ft)\right)_{\ft \in \mathbb{H}}$
defined by $ Y (\ft) = \langle \f, \ft \rangle + \varepsilon W(\ft)$ for all
$\ft \in \mathbb{H}$, where $\varepsilon
> 0$ is a fixed parameter and $W$ is an isonormal process, and that we define
$\gamma(\ft) := -2 Y(\ft) + \norm{\ft}^2$.

\medskip The basic idea is to view the \lasso\ estimator as the solution of a penalized
least squares model selection procedure over a properly defined countable
collection of models with $\ell_{1}$-penalty. The key observation that enables
one to make this connection is the simple fact that $\lu = \bigcup_{R>0} \{\ft
\in \lu,\ \normu{\ft}\leq R\}$, so that for any finite or infinite given
dictionary $\mathcal{D}$, the \lasso\ $\fc$ defined by $$ \fc = \argmin_{\ft
\in \lud}\left(\crit{\ft}+\p\normu{\ft}\right)$$ satisfies
$$\crit{\fc} +\p\normu{\fc} = \inf_{\ft \in \lu}
\crit{\ft}+ \p \normu{\ft} = \inf_{R>0}\left(\inf_{\normu{\ft}\leq R}\crit{\ft}
+\p R\right) \text{.}$$ Then, to obtain a countable collection of models, we
just discretize the family of $\ell_{1}$-balls $\{\ft \in \lu,\ \normu{\ft}\leq
R\}$ by setting for any integer $m\geq1$,
$$S_{m}=\left\{\ft \in \lu,\ \normu{\ft}
\leq m \eps\right\},$$ and define $\hat{m}$ as the smallest integer such that
$\fc$ belongs to $S_{\hat {m}}$, i.e.
\begin{equation}\label{defmchapeau}\hat{m} = \left\lceil \frac{\Vert\fc\Vert_{\lu}}{\eps} \right\rceil.\end{equation} It is now easy to derive from the definitions
of $\hat{m}$ and $\fc$ and from the fact that $\lud = \bigcup_{m \geq 1} S_{m}$
that
\begin{align*} \crit{\fc} +\p
\hat{m}\eps  &  \leq \crit{\fc} +\p
\left(\normu{\fc} + \eps\right)\\
& = \inf_{\ft \in \lu} \left(\crit{\ft} + \p
\normu{\ft}\right) + \p \eps\\
& = \inf_{m\geq1} \left(\inf_{\ft\in S_{m}}
\left(\crit{\ft}+\p \normu{\ft}\right)\right) +\p \eps\\
& \leq \inf_{m\geq1} \left(\inf_{\ft\in S_{m}} \crit{\ft}+\p m
\eps\right) +\p \eps,
\end{align*} that is to say
\begin{equation}\label{minip}\crit{\fc} + \operatorname*{pen}(\hat{m})  \leq
\inf_{m\geq1} \left(\inf_{\ft\in S_{m}} \crit{\ft}+
\operatorname*{pen}(m)\right) +\rho\end{equation} with $\operatorname*{pen}( m)
= \p m \eps$ and $\rho = \p \eps$. \label{idee} This means that $\fc$ is
equivalent to a $\rho$-approximate penalized least squares estimator over the
sequence of models given by the collection of $\ell_{1}$-balls $\left\{  S_{m}%
,\ m\geq1\right\}$. This property will enable us to derive $\ell_{1}$-oracle
type inequalities by applying a general model selection theorem that guarantees
such inequalities provided that the penalty $\operatorname*{pen}(m)$ is large
enough. This general theorem, stated below as Theorem \ref{zmaingaussnl}, is
borrowed from \cite{blm} and is a restricted version of an even more general
model selection theorem that the interested reader can find in
\cite{mas-stflour}, Theorem~4.18. For the sake of completeness, the proof of
Theorem \ref{zmaingaussnl} is recalled in Section \ref{proofs}.
\begin{theorem}
\label{zmaingaussnl}Let $\left\{S_{m}\right\}_{m\in\mathcal{M}}$ be a countable
collection of convex and compact subsets of a Hilbert space $\mathbb{H}$.
Define, for any $m\in\mathcal{M}$,
\begin{equation}
\Delta_{m} := \mathbb{E}\left[  \sup_{\ft \in S_{m}} W(\ft)\right],
\label{edefdim}%
\end{equation}
and consider weights $\left\{  x_{m}\right\}  _{m\in \mathcal{M}}$ such that
\[
\Sigma := \sum_{m\in\mathcal{M}}e^{-x_{m}} < \infty\text{.}%
\]
Let $K>1$ and assume that, for any $m\in\mathcal{M}$,
\begin{equation}
\operatorname*{pen}(m)  \geq2K\varepsilon\left(  \Delta
_{m}+\varepsilon x_{m}+\sqrt{\Delta_{m}\varepsilon x_{m}}\right)
\text{.}
\label{epengaussnl}%
\end{equation}
Given non negative $\rho_{m}$, $m \in\mathcal{M}$, define a
$\rho_{m}$-approximate penalized  least squares estimator as any $\fc \in
S_{\hat{m}}$, $\hat{m} \in \mathcal{M}$, such that
\[
\gamma( \fc) +\operatorname*{pen}(\hat{m})
\leq\inf_{m\in\mathcal{M}}\left(  \inf_{\ft\in S_{m}}\gamma(\ft) +\operatorname*{pen}(m) +\rho
_{m}\right)  \text{.}%
\]
Then, there is a positive constant $C(K)$ such that  for all $\f \in\mathbb{H}$
and $z >0$, with probability larger than $1- \Sigma\, \text{\upshape{e}}^{-z}$,
\begin{align}\label{ooispie}\non & \norm{\f-\fc}^{2} +
\operatorname*{pen}(\hat{m})\\
\leq {} &C(K) \left[ \inf_{m\in\mathcal{M}}\left( \inf_{\ft \in S_{m}} \norm{\f-\ft}^2
+\operatorname*{pen}(m) +\rho_{m}\right)
+(1+z)\eps^2  \right].\end{align} Integrating this
inequality with respect to $z$ leads to the following risk bound
\begin{align} \non &
\mathbb{E}\left[  \norm{\f-\fc}^2 + \operatorname*{pen}(\hat
{m})\right]\\  \leq {}& C(K) \left[
\inf_{m\in\mathcal{M}} \left(\inf_{\ft \in S_{m}} \norm{\f-\ft}^2
+\operatorname*{pen}(m) +\rho_{m}\right)
+(1+\Sigma)\eps^2  \right].  \label{eriskgaussnl}%
\end{align}
\end{theorem}

%%%%%%%%%%%%%%%%%%%%%%%%%%%%%%%%%%%%%%%%%%%%%%%%%%%%%%%%%%%%%%%%%%%%%%%%%%%%%%%%%%%%%
%%%%%%%%%%%%%%%%%%%%%%%%%%%%%%%%%%%%%%%%%%%%%%%%%%%%%%%%%%%%%%%%%%%%%%%%%%%%%%%%%%%%
\section{Proofs}\label{proofs}

%%%%%%%%%%%%%%%%%%%%%%%%%%%%%%%%%%%%%%%%%%%%%%%%%%%%%%%%%%%%%%%%%%%%%%%%%%%%%%%%%%%%%
\subsection{Oracle inequalities}

We first prove the general model selection Theorem \ref{zmaingaussnl}. Its
proof is based on the concentration inequality for the suprema of Gaussian
processes established in~\cite{blm}. Then, deriving Theorem \ref{zlasso},
Theorem \ref{zlassocasinfini} and Theorem \ref{zlassorn} from
Theorem~\ref{zmaingaussnl} is an exercise. Indeed, using the key observation
that the \lasso\ and the \lassotype estimators are approximate penalized least
squares estimators over a collection of $\ell_{1}$-balls with a convenient
penalty, it only remains to determine a lower bound on this penalty to
guarantee condition \eqref{epengaussnl} and then to apply the conclusion of
Theorem \ref{zmaingaussnl}.

%%%%%%%%%%%%%%%%%%%%%%%%%%%%%%%%%%%%%%%%%%%%%%%%%%%%%%%%%%%%%%%%%%%%%%%%%%%%%%%%%%
\subsubsection{Proof of Theorem \ref{zmaingaussnl}} \label{preuvezmaingaussnl}

Let $m \in \mathcal{M}$. Since $S_{m}$ is assumed to be a convex and compact
subset, we can consider $\f_{m}$ the projection of $\f$ onto $S_{m}$, that is
the unique element of $S_{m}$ such that $\norm{\f - \f_{m}}=\inf_{\ft \in
S_{m}} \norm{\f-\ft}$. By definition of $\fc$, we have
\[
\gamma(\fc) +\operatorname*{pen}(\hat{m}) \leq\gamma(
\f_{m}) +\operatorname*{pen}(m) +\rho_{m}.\] Since
$\norm{\f}^{2}+\gamma(\ft) =\norm{\f-\ft}^{2}-2\varepsilon W( \ft)$, this
implies that
\begin{equation}
\norm{\f-\fc}^{2} + \operatorname*{pen}(\hat{m})
\leq\norm{\f-\f_{m}}^{2}+2\varepsilon\left( W(\fc) -W(\f_{m})
\right)    +\operatorname*{pen}(m)  +\rho_{m}\text{.} \label{efund1}%
\end{equation}
For all $m^{\prime}\in\mathcal{M}$, let $y_{m^{\prime}}$ be a positive number
whose value will be specified below and define for every $\ft\in
S_{m^{\prime}}$
\begin{equation}\label{okcolles}
2w_{m^{\prime}}(\ft)  =\left( \norm{\f-\f_{m}}+\norm{\f-\ft}\right)
^{2}+y_{m^{\prime}}^{2}.
\end{equation}
Finally, set
\[
V_{m^{\prime}}=\sup_{\ft\in S_{m^{\prime}}}\left( \frac{W(\ft)
-W(\f_{m})}{w_{m^{\prime}}(\ft)}\right)  \text{.}%
\]
Taking these definitions into account, we get from (\ref{efund1}) that
\begin{equation}
\norm{\f-\fc}^{2}+\operatorname*{pen}(\hat{m}) \leq
\norm{\f-\f_{m}}
^{2}+2\varepsilon w_{\hat{m}}(\fc)  V_{\hat{m}%
} +\operatorname*{pen}(m)  +\rho_{m}\text{.} \label{efund2}%
\end{equation}

\medskip \n The essence of the proof is the control of the random variables
$V_{m^{\prime}}$ for all possible values of $m^{\prime}$. To this end, we may
use the concentration inequality for the suprema of Gaussian processes
(see \cite{blm}) which ensures that, given $z>0$, for all $m^{\prime}%
\in\mathcal{M}$,
\begin{equation}
\mathbb{P}\left[  V_{m^{\prime}}\geq\mathbb{E}\left[
V_{m^{\prime}}\right] +\sqrt{2v_{m^{\prime}}(x_{m^{\prime}}+z)}\right]  \leq
\text{\upshape{e}}^{-(x_{m^{\prime}}+z)}, \label{ecir}%
\end{equation}
where
\[
v_{m^{\prime}}=\sup_{\ft\in S_{m^{\prime}}}\operatorname*{Var}\left[
\frac{W(\ft)-W(\f_{m})}{w_{m^{\prime}}(\ft)}\right] =\sup_{\ft\in
S_{m^{\prime}}} \frac{\norm{\ft-\f_{m}}^{2}}{w_{m^{\prime}}^{2}(\ft)}\ \text{.}%
\]
From \eqref{okcolles}, $w_{m^{\prime}}(\ft)
\geq\left(\norm{\f-\f_{m}}+\norm{\f-\ft}\right)
y_{m^{\prime}}\geq\norm{\ft-\f_{m}} y_{m^{\prime}}$, so $v_{m^{\prime}}\leq
y_{m^{\prime} }^{-2}$ and summing the inequalities (\ref{ecir}) over $m^{\prime
}\in\mathcal{M}$, we get that for every $z > 0$ there is an event $\Omega_{z}$
with $\mathbb{P}(\Omega_{z}) > 1-\Sigma \text{\upshape{e}}^{-z}$ such that on
$\Omega_{z}$, for all $m^{\prime}\in\mathcal{M}$,
\begin{equation}
V_{m^{\prime}}\leq\mathbb{E}\left[  V_{m^{\prime}}\right]  +y_{m^{\prime}%
}^{-1}\sqrt{2 (x_{m^{\prime}}+z)}\text{.} \label{e3cont1}%
\end{equation}
Let us now bound $\mathbb{E}\left[ V_{m^{\prime}}\right] $. We may write
\begin{equation}
\mathbb{E}\left[  V_{m^{\prime}}\right]  \leq\mathbb{E}\left[  \frac
{\sup_{\ft\in S_{m^{\prime}}}\left( W(\ft)-W(\f_{m^{\prime}})
\right)  }{\inf_{\ft\in S_{m^{\prime}}} w_{m^{\prime}}\left(
\ft\right)  }\right] +\mathbb{E}\left[ \frac{\left(
W(\f_{m^{\prime}}) -W( \f_{m}) \right) _{+}}{\inf_{\ft\in
S_{m^{\prime}}} w_{m^{\prime}}(\ft)   }\right]\, .
\label{eespsup}
\end{equation}
But from the definition of $\f_{m^{\prime}}$, we have for all $\ft \in
S_{m^{\prime}}$ \begin{align*}2 w_{m^{\prime}}(\ft) &
\geq\left(\norm{\f-\f_{m}} + \norm{\f-\f_{m^{\prime}}}\right)
^{2}+y_{m^{\prime}}^{2}\\ & \geq\norm{\f_{m^{\prime}}-\f_{m}}
^{2}+y_{m^{\prime}}^{2}\\ & \geq\left( y_{m^{\prime}}^{2}\vee2y_{m^{\prime}}
\norm{\f_{m^{\prime}}-\f_{m}}\right).\end{align*} Hence, on the one hand via
(\ref{edefdim}) and recalling that $W$ is centered, we get \begin{align*}
\mathbb{E}\left[ \frac{\sup_{\ft\in S_{m^{\prime}}}\left( W(\ft)
-W(\f_{m^{\prime}})\right) }{\inf_{\ft\in S_{m^{\prime}}}
w_{m^{\prime}}(\ft)}\right] &   \leq 2\,y_{m^{\prime}}^{-2}\,\mathbb{E}\left[
\sup_{\ft\in S_{m^{\prime}}}\left( W(\ft) -W(\f_{m^{\prime}})\right)  \right]\\
& = 2\,y_{m^{\prime}}^{-2}\,\Delta_{m^{\prime}} ,
\end{align*}
and on the other hand, using the fact that $\left( W(\f_{m^{\prime}})-W(\f_{m})
\right) /\norm{\f_{m}-\f_{m^{\prime}}}$ is a standard normal variable, we get
\[
\mathbb{E}\left[  \frac{\left(  W(\f_{m^{\prime}}) -W\
(\f_{m})  \right)  _{+}}{\inf_{\ft\in S_{m^{\prime}}}  w_{m^{\prime}%
}(\ft)}\right]  \leq y_{m^{\prime}}^{-1}\,\mathbb{E}%
\left[  \frac{W(\f_{m^{\prime}})  -W(\f_{m})
}{\norm{\f_{m}-\f_{m^{\prime}}}}\right]  _{+} \leq
y_{m^{\prime}}^{-1}\left( 2\pi\right) ^{-1/2}\text{.}\] Collecting
these inequalities, we get from (\ref{eespsup}) that for all
$m^{\prime}\in\mathcal{M}$,
\[
\mathbb{E}\left[  V_{m^{\prime}}\right]
\leq2\Delta_{m^{\prime}}y_{m^{\prime }}^{-2}+\left(  2\pi\right)
^{-1/2}y_{m^{\prime}}^{-1}\text{ .}\] Hence, setting $\delta=\left(
\left(  4\pi\right) ^{-1/2}+\sqrt{z}\right) ^{2}$, (\ref{e3cont1}) implies that
on the event $\Omega_{z}$, for all $m^{\prime}\in\mathcal{M}$,
\begin{align}
\label{vmp}\non
V_{m^{\prime}}& \leq y_{m^{\prime}}^{-1}\left[  2\Delta_{m^{\prime}}%
y_{m^{\prime}}^{-1}+\sqrt{2x_{m^{\prime}}}+\left(  2\pi\right)  ^{-1/2}%
+\sqrt{2z}\right]\\ & = y_{m^{\prime}}^{-1}\left[  2\Delta_{m^{\prime}}%
y_{m^{\prime}}^{-1}+\sqrt{2x_{m^{\prime}}}+\sqrt{2\delta}\right].\end{align}

\medskip \n Given $K^{\prime}\in \left(1, \sqrt{K}\right]$ to be chosen later,
we now define
\[
y_{m^{\prime}}^{2}=2{K^{\prime}}^{2}\varepsilon^{2}\left[  \left(  \sqrt{x_{m^{\prime}}%
}+\sqrt{\delta}\right)  ^{2}+{K^{\prime}}^{-1}\varepsilon^{-1}\Delta_{m^{\prime}}%
+\sqrt{{K^{\prime}}^{-1}\varepsilon^{-1}\Delta_{m^{\prime}}} \left(\sqrt{x_{m^{\prime}}}+\sqrt{\delta}\right) \right]  \text{.}%
\] With this choice of $y_{m^{\prime}}$, it is not hard to check that
\eqref{vmp}  warrants that on the event $\Omega_{z}$, $\varepsilon
V_{m^{\prime}}\leq {K^{\prime}}^{-1}$ for all $m^{\prime}\in\mathcal{M}$, which
in particular implies that $\varepsilon V_{\hat{m}}\leq {K^{\prime}}^{-1}$, and
we get from \eqref{efund2} and \eqref{okcolles} that
\begin{align}\label{ine} \non
& \norm{\f -\fc}^{2}+\operatorname*{pen}(\hat{m})\\ \non
\leq {} & \norm{\f-\f_{m}}^{2}+2{K^{\prime}}^{-1}w_{\hat{m}}(\fc)
  +\operatorname*{pen}(m)  +\rho_{m}\\   = {} &
\norm{\f-\f_{m}}^{2}+{K^{\prime}}^{-1}\left[ \left(\norm{\f-\f_{m}}
+\norm{\f-\fc}\right) ^{2}+y_{\hat{m}}^{2}\right]
+\operatorname*{pen}(m)  +\rho_{m}\text{.}
\end{align}
Moreover, using repeatedly the elementary inequalities $(a+b)^2 \leq (1+\theta)
a^2 + (1+\theta^{-1}) b^{2}$ or equi\-valently $2ab\leq\theta
a^{2}+\theta^{-1}b^{2}$ for various values of $\theta>0$, we derive that on the
one hand
\[
\left(  \norm{\f-\f_{m}}+\norm{\f-\fc}
\right)  ^{2}\leq \sqrt{K^{\prime}}\left(\norm{\f-\fc}^{2}%
+\frac{\norm{\f-\f_{m}}^{2}}{\sqrt{K^{\prime}}-1}\right)  \text{,}%
\]
and on the other hand
\[
{K^{\prime}}^{-1}y_{\hat{m}}^{2}\leq 2 {K^{\prime}}^{2}\varepsilon^{2}\left[  \varepsilon^{-1}%
\Delta_{\hat{m}}+x_{\hat{m}}+\sqrt{\varepsilon^{-1}\Delta_{\hat{m}}x_{\hat{m}%
}}+B(K^{\prime})\left( \frac{1}{2\pi}+2z\right) \right],\] where $B(K^{\prime})
= \left(K^{\prime}-1\right)^{-1}+\left(4 K^{\prime}
({K^{\prime}}^{2}-1)\right)^{-1}$.

\n Hence, setting $A(K^{\prime})= 1+{K^{\prime}}^{-1/2}\left(
\sqrt{K^{\prime}}-1\right) ^{-1}$, we deduce from \eqref{ine} that on the event
$\Omega_{z}$,
\begin{align*}
 & \norm{\f-\fc}^{2} + \operatorname*{pen}(\hat{m})\\  \leq {}
&
A(K^{\prime})\norm{\f-\f_{m}}^{2}+{K^{\prime}}^{-1/2}\norm{\f-\fc}^{2}
  +2{{K}^{\prime}}^{2}\varepsilon\left[
\Delta_{\hat{m}}+\varepsilon x_{\hat{m}}
+\sqrt{\varepsilon\Delta_{\hat{m}}x_{\hat{m}}}\right]\\ &
+\operatorname*{pen}(m)  +\rho_{m}+ 2
\varepsilon^{2} {K^{\prime}}^{2} B(K^{\prime}) \left(  \frac{1}{2\pi}+2z\right)  \text{,}%
\end{align*} or equivalently \begin{align*}
\nonumber  & \left(1-{K^{\prime}}^{-1/2}\right) \norm{\f-\fc}^{2} +
\operatorname*{pen}(\hat{m}) - 2{{K}^{\prime}}^{2}\varepsilon\left[
\Delta_{\hat{m}}+\varepsilon x_{\hat{m}}+\sqrt{\varepsilon\Delta_{\hat{m}}x_{\hat{m}}}\right]\\
 \leq {} & A(K^{\prime})\left\Vert \f-\f_{m}\right\Vert ^{2}+
\operatorname*{pen}(m)  +\rho_{m}
 + 2\varepsilon^{2} B(K^{\prime}) \left( \frac{1}{2\pi}+2z\right)
\text{.}
\end{align*}
Because of condition (\ref{epengaussnl}) on the penalty function, this implies
that
\begin{align*}
\nonumber  & \left(1-{K^{\prime}}^{-1/2}\right) \norm{\f-\fc}^{2}
+\left(1-{K^{\prime}}^{2}K^{-1}\right) \operatorname*{pen}(\hat{m})\\   \leq {} & A(K^{\prime})\left\Vert
\f-\f_{m}\right\Vert ^{2}+
\operatorname*{pen}(m)  +\rho_{m}
 + 2\varepsilon^{2} B(K^{\prime}) \left( \frac{1}{2\pi}+2z\right)
\text{.}
\end{align*}
Now choosing $K^{\prime} = K^{2/5}$, we get that
\begin{align*}
\nonumber  & \left(1-{K}^{-1/5}\right) \left(\norm{\f-\fc}^{2} +
\operatorname*{pen}(\hat{m})\right)\\   \leq {} & A(K^{2/5})\norm{\f-\f_{m}}^{2}+\operatorname*{pen}(m)+\rho_{m}
 + 2\varepsilon^{2} B(K^{2/5}) \left( \frac{1}{2\pi}+2z\right)
\text{.}
\end{align*}
So, there exists a positive constant $C := C(K)$ depending only on $K$ such
that for all $z >0$, on the event $\Omega_{z}$,
$$\norm{\f-\fc}^{2} + \operatorname*{pen}(\hat{m})
\leq C \left( \inf_{m\in\mathcal{M}}\left( \norm{\f-\f_{m}}^2
+\operatorname*{pen}(m) +\rho_{m}\right) +\varepsilon^{2}(1+z)
\right),$$ which proves \eqref{ooispie}. Integrating this
inequality with respect to $z$ straightforwardly leads to the risk bound
(\ref{eriskgaussnl}). \hfill $\Box$

%%%%%%%%%%%%%%%%%%%%%%%%%%%%%%%%%%%%%%%%%%%%%%%%%%%%%%%%%%%%%%%%%%%%%%%%%%%%%%%%%%%%%%%%%
\subsubsection{Proof of Theorem \ref{zlasso}} \label{preuvethm}

Fix $p \in \ensn$. Let $\mathcal{M} = \mathbb{N}^*$ and consider the collection
of $\ell_{1}$-balls for $m \in \mathcal{M}$,
$$S_{m}=\left\{ \ft\in \ludp \text{, }\normup{\ft} \leq
m \varepsilon\right\}\text{.}$$ We have noticed at \eqref{minip} that the
\lasso\ estimator $\fcp$ is a $\rho$-approximate penali\-zed least squares
estimator over the sequence $\left\{S_{m}%
,\ m\geq1\right\}$ for $\operatorname*{pen}(m) = \lambdap m \eps$ and $\rho =
\lambdap \eps.$ So, it only remains to determine a lower bound on $\lambdap$
that gua\-rantees that $\operatorname*{pen}(m)$ satisfies condition
\eqref{epengaussnl}.

\medskip \n Let $\ft \in S_{m}$ and consider $\theta = (\theta_{1},\dots,\theta_{p})$
such that $\ft = \theta.\vp = \sum_{j=1}^{p} \theta_{j}\, \vp_{j}$ and
$\normup{\ft} = \norm{\theta}_{1}$. The linearity of $W$ implies that
\begin{equation} \label{ch}
W(\ft) = \sum_{j=1}^{p} \theta_{j}\, W(\vp_{j})
\leq \sum_{j=1}^{p} \vert\theta_{j}\vert\,
\vert W(\vp_{j})\vert \leq m\varepsilon\max_{j =
1,\dots,p}\left\vert W(\vp_{j})  \right\vert \text{.}%
\end{equation}
From Definition \ref{defw}, $\text{Var}\left[ W(\vp_{j})\right] =
\mathbb{E}\left[{W^2(\vp_{j})}\right] = \norme{\vp_{j}}^2 \leq 1$ for all $j
=1,\dots,p$. So, the variables $W(\vp_{j})$ and $\left(-W(\vp_{j})\right)$,
$j=1,\dots,p$, are $2p$ centered normal variables with variance less than 1 and
thus (see Lemma 2.3 in \cite{mas-stflour} for instance),
$$\mathbb{E}\left[ \max_{j=1,\dots,p}\left\vert W(
\vp_{j}) \right\vert \right] =\mathbb{E}\left[
\left(\max_{j=1,\dots,p}W(\vp_{j})\right)
\vee\left(\max_{j=1,\dots,p}\left(-W(\vp_{j})\right)\right) \right] \leq\sqrt{2\logn(2p)}\ .$$
Therefore, we deduce from \eqref{ch} that
\begin{equation} \label{deltam} \Delta_{m}:=\mathbb{E}\left[
\sup_{\ft\in S_{m}}W(\ft) \right]  \leq
m\varepsilon\sqrt{2\logn(2p)}\leq \sqrt{2}
m\varepsilon \left( \sqrt{\logn p }+\sqrt{\logn  2}\right)
\text{.}\end{equation}

\medskip \n Now, choose the weights of the form
$x_{m}= \gamma m$ where $\gamma>0$ is specified below. Then, $\sum_{m\geq1}
e^{-x_{m}}=1/\left(e^{\gamma}-1\right):=\Sigma_{\gamma} < +\infty$.

\medskip \n Defining $K = 4 \sqrt{2} / 5
> 1$ and $\gamma=(1-\sqrt{\logn
2})/K$, and using the inequality $2\sqrt{ab}\leq\eta a+\eta^{-1}b$ with $\eta =
1/2$, we get that
\begin{align}\nonumber
2K\varepsilon\left(  \Delta_{m}+\varepsilon
x_{m}+\sqrt{\Delta_{m}\varepsilon
x_{m}}\right) &  \leq  K\varepsilon\left(  \frac{5}{2}\Delta_{m}+4x_{m}%
\varepsilon\right)\\ \non &  \leq  4 m\varepsilon^{2}\left(
\sqrt{\logn p}+\sqrt{\logn  2}+K\gamma\right)\\ \non & \leq
 4m\varepsilon^{2}\left( \sqrt{\logn p}+1\right)\\
\non & \leq  \lambdap m \eps
\end{align}
as soon as \begin{equation}\label{mamach}\lambdap \geq 4 \varepsilon
\left(\sqrt{\logn p}+1\right).\end{equation}  For such values of $\lambdap$,
condition (\ref{epengaussnl}) on the penalty function is satisfied and we may
apply Theorem~\ref{zmaingaussnl}. Taking into account the definition of
$\hat{m}$ at \eqref{defmchapeau} and noticing that $\varepsilon^{2} \leq
\lambdap \varepsilon / 4$ for $\lambdap$ satisfying \eqref{mamach}, we get from
\eqref{ooispie} that there exists some $C>0$ such that for all $z
>0$, with probability larger than $1-\Sigma_{\gamma}\, \text{\upshape{e}}^{-z} \geq 1-3.4\, \text{\upshape{e}}^{-z}$,
\begin{align}\label{ille}\non & \norm{\f-\fcp}^{2} +\lambdap \normup{\fcp}\\ \non   \leq {} &   C \left[ \inf_{m\geq
1}\left(\inf_{\normup{\ft} \leq m\varepsilon}\norm{\f-\ft}^2 + \lambdap m
\eps\right)+\lambdap \eps +(1+z)\eps^2 \right]\\   \leq {} &  C \left[
\inf_{m\geq 1}\left(\inf_{\normup{\ft} \leq m\varepsilon}\norm{\f-\ft}^2 +
\lambdap m \eps\right)+\lambdap \eps(1+z) \right],\end{align} while
the risk bound (\ref{eriskgaussnl}) leads to
\begin{align}\label{inter}\non & \mathbb{E}\left[\norm{\f-\fcp}^{2} + \lambdap
\normup{\fcp}\right]\\ \non  \leq {} &  C \left[ \inf_{m\geq1}\left( \inf_{\normup{\ft}
\leq m\varepsilon}\norm{\f-\ft}^{2}+\lambdap m\varepsilon\right)
+\lambdap\varepsilon +(1+\Sigma_{\gamma}) \varepsilon ^{2}\right]\\
 \leq {} &  C \left[ \inf_{m\geq1}\left( \inf_{\normup{\ft} \leq
m\varepsilon}\norm{\f-\ft}^{2}+\lambdap m\varepsilon\right)
+\lambdap\varepsilon\right].\end{align}

\medskip \n Finally, to get the desired bounds \eqref{inegaliteenproba} and \eqref{inegaliteoraclecasfini},
just notice that for all $R>~0$, by considering $m_{R} = \lceil R/\varepsilon
\rceil \in \mathbb{N}^{*}$, we have for all $g \in \ludp$ such that $\normup{g}
\leq R$,
\begin{align*} \inf_{m\geq1} \left(\inf_{\normup{\ft}  \leq
m\varepsilon}\norm{\f-\ft}^{2}+\lambdap m\varepsilon\right) & \leq
\norm{\f-g}^{2}+\lambdap\, m_{R}\, \varepsilon\\ & \leq \norm{\f-g}^{2}+\lambdap
R + \lambdap \varepsilon,\end{align*} so that
\begin{align}\label{feui}\non & \inf_{m\geq1} \left(\inf_{\normup{\ft} \leq
m\varepsilon}\norm{\f-\ft}^{2}+\lambdap m\varepsilon\right)\\ \non   \leq {} &  \inf_{R
> 0}\left(\inf_{\normup{g}\leq R} \norm{\f-g}^{2}+\lambdap
R\right) + \lambdap \varepsilon\\   = {} & \inf_{g\in \ludp} \left(
\norm{\f-g}^{2}+\lambdap \normup{g}\right) +\lambdap \varepsilon,\end{align}
and combining \eqref{feui} with \eqref{ille} and \eqref{inter} leads to
$$\norm{\f-\fcp}^{2} +\lambdap \normup{\fcp}  \leq  C \left[ \inf_{g \in \ludp}\left(\norm{\f-g}^2 + \lambdap
\normup{g}\right) +\lambdap \varepsilon(1+z) \right]$$ and
$$\mathbb{E}\left[\norm{\f -\fcp}^{2} + \lambdap \normup{\fcp}\right]  \leq C
\left[ \inf_{g\in \ludp}\left( \norm{\f-g}^{2}+\lambdap \normup{g}\right)
+\lambdap\varepsilon\right],$$ where $C>0$ is some absolute constant. \hfill $\Box$

%%%%%%%%%%%%%%%%%%%%%%%%%%%%%%%%%%%%%%%%%%%%%%%%%%%%%%%%%%%%%%%%%%%%%%%%%%%%%%%%%%
\subsubsection{Proof of Theorem \ref{zlassocasinfini}}
\label{preuvezlassocasinfini}

Let $\mathcal{M} = \mathbb{N}^* \times \ens$ and consider the set of
$\ell_{1}$-balls for all $(m,\pp) \in \mathcal{M}$,
$$S_{m,p}=\left\{ \ft\in \ludp\text{, }\normup{\ft}\leq
m \varepsilon\right\}\text{.}$$ Define $\hat{m}$ as the smallest integer such
that $\fcf$ belongs to $S_{\hat {m},\hat{\pp}}$, i.e.
\begin{equation}\label{defmchapeaucasinfini}\hat{m} = \left\lceil\frac{
\Vert\fcf\Vert_{\mathcal{L}_{1}(\mathcal{D}_{\hat{\pp}})}}{\varepsilon}
\right\rceil.\end{equation} Let $\alpha >0$ be a constant to be chosen later.
 From the definitions of $\hat{m}$, $\lambdapc$ and
$\operatorname*{pen}(\hat{\pp})$, and using the fact that for all $\pp\in
\ens$, $\sqrt{\logn \pp} \leq (\logn \pp)/\sqrt{\logn 2}$, we have
\begin{align*} \crit{\fcf} +\lambdapc \hat{m} \varepsilon +
\alpha\operatorname*{pen}(\hat{\pp})
 &  \leq \crit{\fcf} +\lambdapc
\normupc{\fcf} + \lambdapc \varepsilon + \alpha\operatorname*{pen}(\hat{\pp})\\
 &  \leq \crit{\fcf} +\lambdapc \normupc{\fcf} + 4
 \varepsilon^2 \sqrt{\logn \hat{\pp}} + 4 \varepsilon^2 +
  5 \alpha \varepsilon^2 \logn \hat{\pp} \\
 &  \leq \crit{\fcf} +\lambdapc \normupc{\fcf} +
 \left(\frac{4}{5\sqrt{\logn
 2}} + \alpha\right) 5 \varepsilon^2
 \logn \hat{\pp}  + 4 \varepsilon^2 \\
&  \leq \crit{\fcf} +\lambdapc \normupc{\fcf} +
\left(\frac{4}{5\sqrt{\logn
 2}} + \alpha\right)
\operatorname*{pen}(\hat{\pp}) + 4 \varepsilon^2.\end{align*} Now, if we choose
$\alpha=1-4/(5\sqrt{\logn 2}) \in\, ]0,1[$, we get from the definition of
$\fcf$ and the fact that $\ludp = \bigcup_{m\in \mathbb{N}^{*}} S_{m,p}$, that
\begin{align*} \crit{\fcf} +\lambdapc \hat{m} \varepsilon +
\alpha\operatorname*{pen}(\hat{\pp})
 &  \leq \crit{\fcf} +\lambdapc \normupc{\fcf} +
\operatorname*{pen}(\hat{\pp}) + 4 \varepsilon^2\\ & \leq
\inf_{\pp \in \ens} \left[\inf_{\ft \in \ludp} \left(\crit{\ft} +
\lambdap \normup{\ft}\right) +
\operatorname*{pen}(\pp)\right]+
4 \varepsilon^2\\
& \leq \inf_{p\in\ens}\left[\inf_{m\in\mathbb{N}^*} \left(\inf_{\ft\in
S_{m,p}} \crit{\ft}+\lambdap m \varepsilon\right) +
\operatorname*{pen}(\pp) \right]+ 4 \varepsilon^2\\ &
\leq \inf_{(m,\pp) \in \mathcal{M}}\left[\inf_{\ft\in S_{m,p}}
\crit{\ft}+\lambdap m \varepsilon +
\operatorname*{pen}(\pp) \right]+ 4
\varepsilon^2,\end{align*} that is to say
\[\crit{\fcf} + \operatorname*{pen}(\hat{m},\hat{\pp})  \leq
\inf_{(m,\pp) \in \mathcal{M}} \left[\inf_{\ft\in S_{m,\pp}} \crit{\ft}+
\operatorname*{pen}( m,\pp) +\rho_{\pp}\right],\] with $\operatorname*{pen}(
m,\pp) := \lambdap m \varepsilon + \alpha \operatorname*{pen}(\pp)$ and
$\rho_{\pp} := (1-\alpha) \operatorname*{pen}(\pp)+ 4 \varepsilon^2$. This
means that $\fcf$ is equivalent to a $\rho_{p}$-approximate penalized least
squares estimator over the sequence of models $\left\{S_{m,p},\ (m,\pp) \in
\mathcal{M}\right\}$. By applying Theorem \ref{zmaingaussnl}, this property
will enable us to derive a performance bound satisfied by $\fcf$ provided that
$\operatorname*{pen}(m,\pp)$ is large enough. So, it remains to choose weights
$x_{m,p}$ so that condition (\ref{epengaussnl}) on the penalty function is
satisfied with $\operatorname*{pen}(m,\pp) = \lambdap m \varepsilon + \alpha
\operatorname*{pen}(\pp)$.

\medskip \n Let us choose the weights of the form $x_{m,\pp}= \gamma m + \beta \logn \pp$
where $\gamma >0$ and $\beta>0$ are  numerical constants specified later. Then,
\begin{align*}\Sigma_{\gamma, \beta} := \sum_{(m,p)\in \mathcal{M}}
e^{-x_{m,\pp}} & = \left(\sum_{m \in \mathbb{N}^*} e^{-\gamma m}\right)
\left(\sum_{\pp \in \ens} e^{-\beta \logn \pp}\right)\\ & = \left(\sum_{m \in
\mathbb{N}^*} e^{-\gamma
m}\right) \left(\sum_{J \in \mathbb{N}} e^{- \beta \logn 2^J}\right)\\
&  = \frac{1}{\left(e^{\gamma}-1\right)\left(1-2^{-\beta}\right)} \ <
+\infty.\end{align*} Moreover, for all $(m,\pp) \in \mathcal{M}$, we can prove
similarly as \eqref{deltam} that $$ \Delta_{m,\pp}:=\mathbb{E}\left[
\sup_{\ft\in S_{m,\pp}}W(\ft) \right]  \leq \sqrt{2} m\varepsilon \left(
\sqrt{\logn p }+\sqrt{\logn 2}\right) \text{.}$$Now, defining $K = 4 \sqrt{2} /
5
> 1$, $\gamma=( 1-\sqrt{\logn
2}) /K >0$ and $\beta = (5\alpha)/(4K) >0$, and using the inequality
$2\sqrt{ab}\leq\eta a+\eta^{-1}b$ with $\eta = 1/2$, we have
\begin{align*}\nonumber &
2K\varepsilon\left(  \Delta_{m,\pp}+\varepsilon
x_{m,\pp}+\sqrt{\Delta_{m,\pp} \varepsilon x_{m,\pp}}\right)\\
\leq {} & K \varepsilon \left(\frac{5}{2}\, \Delta_{m,\pp}+4
x_{m,\pp} \varepsilon\right)\\ \non   \leq  {} & 4
\varepsilon^{2}\left(m \sqrt{\logn \pp}+ m
\sqrt{\logn 2}+K \gamma m + K \beta \logn \pp\right)\\
\non  \leq {} & 4
\varepsilon^{2}\left(m\left(\sqrt{\logn \pp}+1\right) + K \beta \logn \pp\right)\\
\non  \leq {} &  4
\varepsilon^{2}\left(m \left(\sqrt{\logn \pp}+1\right)+\frac{5\alpha}{4} \logn \pp\right)\\
\non  \leq {} & \lambdap m \varepsilon +\alpha
\operatorname*{pen}(\pp).\non
\end{align*}
Thus, condition (\ref{epengaussnl}) is satisfied and we can apply
Theorem~\ref{zmaingaussnl} with $\operatorname*{pen}(m,\pp) = \lambdap m
\varepsilon +\alpha \operatorname*{pen}(\pp)$ and $\rho_{\pp} = (1-\alpha)
\operatorname*{pen}(\pp) + 4 \varepsilon^2$, which leads to the following risk
bound:
\begin{align}\label{pluie}\non  & \mathbb{E}\left[\norm{\f-\fcf}^{2}+\lambda_{\hat{p}} \hat{m}\varepsilon +
\alpha \operatorname*{pen}(\hat{\pp})\right]\\ \non   \leq {} & C \left[
\inf_{(m,\pp)\in \mathcal{M}} \left(\inf_{\ft \in S_{m,\pp}}\left\Vert
\f-\ft\right\Vert ^{2}+\lambdap m\varepsilon +
\operatorname*{pen}(\pp)\right) +\left( 5+\Sigma_{\gamma,
\beta}\right) \varepsilon ^{2}\right]\\  \leq {} & C \left[\inf_{(m,\pp)\in
\mathcal{M}} \left(\inf_{\ft \in S_{m,\pp}}\left\Vert \f-\ft\right\Vert
^{2}+\lambdap m\varepsilon + \operatorname*{pen}(\pp)\right)
+\varepsilon ^{2}\right],\end{align} where $C>0$ denotes some numerical
constant. The infimum of this risk bound can easily be extended to $\inf_{p \in
\ens} \inf_{\ft\in \ludp}$. Indeed, let $\pp_{0} \in \ens$ and $R
> 0$, and consider $m_{R} = \lceil R/\varepsilon \rceil \in
\mathbb{N}^*$. Then for all $g \in
\mathcal{L}_{1}\left(\mathcal{D}_{\pp_{0}}\right)$ such that $\normupo{g} \leq
R$, we have $g \in S_{m_{R},\pp_{0}}$, and thus
\begin{align}\label{soleil}\non  &  \inf_{(m,\pp)\in \mathcal{M}}
\left(\inf_{\ft \in S_{m,\pp}}\left\Vert \f-\ft\right\Vert
^{2}+\lambdap m\varepsilon +
\operatorname*{pen}(\pp)\right)\\ \non  \leq {} &  \left\Vert \f-g\right\Vert ^{2}+
\lambda_{\pp_{0}} m_{R} \varepsilon +
\operatorname*{pen}(\pp_{0})\\ \non  \leq {} &  \left\Vert
\f-g\right\Vert ^{2}+ \lambda_{\pp_{0}} \left(R + \varepsilon\right)
+ \operatorname*{pen}(\pp_{0})\\   \leq {} &
\left\Vert \f-g\right\Vert ^{2}+ \lambda_{\pp_{0}} R
 +\left(\frac{4}{5\sqrt{\log 2}}+1\right) \operatorname*{pen}(\pp_{0}) + 4\varepsilon^2.
 \end{align} So, we deduce from \eqref{pluie} and \eqref{soleil} that there exists $C>0$ such that
\begin{align}\label{congg}\non
 & \mathbb{E}\left[\norm{\f -\fcf}^{2} +\lambda_{\hat{p}} \hat{m}\varepsilon +
\alpha \operatorname*{pen}(\hat{\pp})\right]\\ \non  \leq {} & C \left[
\inf_{\pp \in \ens} \left(\inf_{R
> 0}\left(\inf_{\underset{\normup{g}\leq R}{g \in
\ludp,}}\ \left\Vert \f-g\right\Vert ^{2} +\lambdap R\right) +
\operatorname*{pen}(\pp)\right)
+ \varepsilon ^{2}\right]\\
 \leq {} & C \left[\inf_{p \in \ens} \left( \inf_{g\in \ludp} \left(
\left\Vert \f-g\right\Vert ^{2}+\lambdap \normup{g}\right) +
\operatorname*{pen}(\pp)\right)+\varepsilon^{2}\right].
\end{align} Finally, let us notice that from the fact that
$\alpha \in\, ]0,1[$ and from \eqref{defmchapeaucasinfini}, we have
\begin{align}\label{merp} \non & \mathbb{E}\left[\norm{\f -\fcf}^{2} +
\lambda_{\hat{p}} \norm{\fcf}_{\mathcal{L}_{1}(\mathcal{D}_{\hat{p}})} +
\operatorname*{pen}(\hat{\pp})\right]\\ \non \leq {} &  \frac{1}{\alpha}\
\mathbb{E}\left[\norm{\f -\fcf}^{2} + \lambda_{\hat{p}}
\norm{\fcf}_{\mathcal{L}_{1}(\mathcal{D}_{\hat{p}})} + \alpha
\operatorname*{pen}(\hat{\pp})\right]\\  \leq {} &  \frac{1}{\alpha}\
\mathbb{E}\left[\norm{\f -\fcf}^{2} +\lambda_{\hat{p}} \hat{m}\varepsilon +
\alpha \operatorname*{pen}(\hat{\pp})\right].\end{align} Combining
\eqref{congg} with \eqref{merp} leads to the result. \hfill $\Box$

%%%%%%%%%%%%%%%%%%%%%%%%%%%%%%%%%%%%%%%%%%%%%%%%%%%%%%%%%%%%%%%%%%%%%%%%%%%%%%%%%%%%%
\subsubsection{Proof of Theorem \ref{zlassorn}}

The proof of Theorem \ref{zlassorn} is again an application of Theorem
\ref{zmaingaussnl} and it is thus very similar to the proof of Theorem
\ref{zlasso}. In particular, it is still based on the key idea that the \lasso\
estimator~$\fc$ is an approximate penalized least squares estimator over the
collection of $\ell_{1}$-balls for $m \in \mathbb{N}^*$, $S_{m}=\{\ft\in \lu
\text{, }\normu{\ft} \leq m \sigma/\sqrt{n}\}\text{.}$ The main difference is
that the dictionary $\mathcal{D}$ considered for Theorem~\ref{zlasso} was
finite while the dictionary $\mathcal{D}=\{\vp_{\q,\w},\, \q
\in~\mathbb{R}^\dd, \w \in~\mathbb{R}\}$ is infinite. Consequently, we can not
use the same tools to check the assumptions of Theorem~\ref{zmaingaussnl}, more
precisely to provide an upper bound of $\mathbb{E}\left[\sup_{\ft \in S_{m}}
W(h)\right]$. Here, we shall bound this quantity by using Dudley's criterion
(see Theorem 3.18 in \cite{mas-stflour} for instance) and we shall thus first
establish an upper bound of the $t$-packing number of $\mathcal{D}$ with
respect to $\norm{.}$.

\begin{definition}\textbf{\droit{[$t$-packing numbers]}}
Let $t>0$ and let $\mathcal{G}$ be a set of functions $\mathbb{R}^d \mapsto
\mathbb{R}$. We call $t$-packing number of $\mathcal{G}$ with respect to
$\norme{.}$, and denote by $N \left(t,\mathcal{G},\norme{.}\right)$, the
maximal $m \in \mathbb{N}^{*}$ such that there exist functions
$g_{1},\dots,g_{m} \in \mathcal{G}$ with $\norme{g_{i}-g_{j}} \geq t$ for all
$1\leq i < j \leq m$.
\end{definition}

\begin{lemma}\label{covering} Let $t >0$. Then, the $t$-packing number of $\mathcal{D}$
with respect to $\norm{.}$ is upper bounded by
$$N\left(t,\mathcal{D},\norme{.}\right) \leq (n+1)^{d+1}\ \frac{4+t}{t}.$$\end{lemma}

\begin{proof}\label{preuvecovering}The inequality can easily be deduced from the
intermediate result (9.10) in the proof of Lemma 9.3 in \cite{walk}. We recall
here this result. Let $\mathcal{G}$ be a set of functions $\mathbb{R}^d \mapsto
\mathbb{R}$. If $\mathcal{G}$ is a linear vector space of dimension $D$, then,
for every $R>0$ and $t>0$, the $t$-packing number of $\{g \in \mathcal{G},
\norme{g}\leq R\}$ with respect to $\norme{.}$ is upper bounded by
$$N\left(t,\left\{g \in \mathcal{G}, \norme{g}\leq
R\right\},\norme{.}\right) \leq \left(\frac{4R+t}{t}\right)^D.$$ We can apply
this result to the linear span of $\vp_{\q,\w}$, that we denote by
$\mathcal{F}_{\q,\w}$, for all $\q \in \mathbb{R}^d$ and $\w \in \mathbb{R}$.
From \eqref{defdicorn}, we have $\sup_{x}\vert \vp_{\q,\w}(x)\vert \leq 1$, so
 $\norme{\vp_{\q,\w}}^2 = \sum_{i=1}^{n} \vp_{\q,\w}^{2}(x_{i})/n
\leq 1$ and we get that
$$\mathcal{D} = \bigcup_{\q \in \mathbb{R}^d, \w \in\mathbb{R}} \{\vp_{\q,\w}\}
\ \subset\  \bigcup_{\q \in \mathbb{R}^d, \w \in\mathbb{R}} \left\{g \in
\mathcal{F}_{\q,\w}, \norme{g}\leq 1\right\},$$ with
$$N\left(t,\left\{g \in \mathcal{F}_{\q,\w}, \norme{g}\leq
1\right\},\norme{.}\right) \leq \frac{4+t}{t},$$ for all $\q \in \mathbb{R}^d$
and $\w \in \mathbb{R}$. To end the proof, just notice that there are at most
$(n+1)^{d+1}$ hyperplanes in $\mathbb{R}^d$ separating the points
$(x_{1},\dots,x_{n})$ in different ways (see Chapter 9 in \cite{walk} for
instance), with the result that there are at most $(n+1)^{d+1}$ ways of
selecting $\vp_{\q,\w}$ in $\mathcal{D}$ that will be different on the sample
$(x_{1},\dots,x_{n})$. Therefore, we get that for all $t>0$,
$$N\left(t,\mathcal{D},\norme{.}\right) \leq (n+1)^{d+1}\ \frac{4+t}{t}.$$
\end{proof}

\begin{remark}\label{lien}\droit{Let us point out the fact that we are able
to get such an upper bound of $N \left(t,\mathcal{D},\norme{.}\right)$ thanks
to the particular structure of the dictionary $\mathcal{D}$. Indeed, for all
$\q \in \mathbb{R}^d$, $\w \in \mathbb{R}$ and $x \in \mathbb{R}^d$,
$\vp_{\q,\w}(x) \in \{0,1\}$, but there are at most $(n+1)^{d+1}$ hyperplanes
in $\mathbb{R}^d$ separating the observed points $(x_{1},\dots,x_{n})$ in
different ways, so there are at most $(n+1)^{d+1}$ ways of selecting
$\vp_{\q,\w} \in \mathcal{D}$ which will give different functions on the sample
$(x_{1},\dots,x_{n})$. In particular, this pro\-perty enables us to bound the
packing numbers of $\mathcal{D}$ without truncation of the dictionary. This
would not be possible for an arbitrary infinite (countable or uncountable)
dictionary and truncation of the dictionary into finite subdictionaries was
necessary to achieve our theoretical results in Section~\ref{section3} when
considering an arbitrary infinite countable dictionary.}
\end{remark}

\n The following technical lemma will also be used in the proof of Theorem
\ref{zlassorn}.
\begin{lemma}\label{majintln} $$\int_{0}^{1} \sqrt{\logn \left(\frac{1}{t}\right)}\ dt \leq
\sqrt{\pi}.$$
\end{lemma}
\begin{proof} By integration by parts and by defining $u=\sqrt{2\logn(1/t)}$,
we have
$$\int_{0}^{1} \sqrt{\logn \left(1/t\right)}\ dt\,  = \,
 \left[t
\sqrt{\logn\left(1/t\right)}\right]_{0}^{1}+\int_{0}^{1}\frac{1}{2\sqrt{\logn(1/t)}}\
\ dt\, =\, \frac{1}{\sqrt{2}}\int_{0}^{+\infty}
\text{\droit{e}}^{-u^2/2}\ du.$$ But, if $Z$ is a standard
Gaussian variable, we have
$$\int_{0}^{+\infty} \frac{1}{\sqrt{2\pi}}\ \text{\droit{e}}^{-u^2/2}\ du\,
= \, \mathbb{P}(Z \geq 0) \leq  1,$$ hence the result.\end{proof}

\bigskip

\n \textbf{Proof of Theorem \ref{zlassorn} }\label{preuvezlassorn} Let us
define $\eps=\sigma/\sqrt{n}.$ Consider the collection of $\ell_{1}$-balls for
$m \in \mathbb{N}^*$,
$$S_{m}=\left\{ \ft\in \lu \text{, }\normu{\ft} \leq
m \varepsilon\right\}\text{.}$$ We have noticed in Section \ref{mst} that the
\lasso\ estimator $\fc$ is a $\rho$-approximate penalized least squares
estimator over the sequence $\left\{S_{m}%
,\ m\geq1\right\}$ for $\operatorname*{pen}(m) = \p m \eps$ and $\rho = \p
\eps$. So, it only remains to determine a lower bound on $\p$ that guarantees
that $\operatorname*{pen}(m)$ satisfies condition \eqref{epengaussnl} and to
apply the conclusion of Theorem \ref{zmaingaussnl}.

\medskip \n Let $h \in S_{m}$. From \eqref{defludrn}, for all $\delta >0$, there exist
coefficients $\theta_{\q,\w}$ such that $h = \sum_{\q,\w} \theta_{\q,\w}\,
\vp_{\q,\w}$ and $\sum_{\q,\w}\vert \theta_{\q,\w}\vert \leq m\eps+\delta$. By
using the linearity of $W$, we get that
$$W(h) = \sum_{\q,\w} \theta_{\q,\w}\, W(\vp_{\q,\w}) \leq
\sup_{\q,\w}\vert W(\vp_{\q,\w})\vert \sum_{\q,\w}\vert \theta_{\q,\w}\vert
\leq \left(m\eps+\delta\right) \sup_{\q,\w}\vert W(\vp_{\q,\w})\vert.$$ Then,
by Dudley's criterion (see Theorem 3.18 in \cite{mas-stflour} for instance), we
have \begin{align*}\Delta_{m} := \mathbb{E}\left[\sup_{\ft \in S_{m}}
W(h)\right] & \leq \left(m \eps+\delta\right) \mathbb{E}\left[\sup_{\q,\w}
\vert W(\vp_{\q,\w})\vert\right]\\ & \leq 12 (m \eps+\delta) \int_{0}^{\alpha}
\sqrt{\ln\left(N\left(t,\mathcal{D},\norme{.}\right)\right)}\ \,
dt,\end{align*} where $\alpha^2=\sup_{\q,\w}
\mathbb{E}\left[W^2(\vp_{\q,\w})\right] = \sup_{\q,\w} \norm{\vp_{\q,\w}}^2 =
\sup_{\q,\w}\left(\sum_{i=1}^{n} \vp_{\q,\w}^{2}(x_{i})/n\right) \leq~1$ from
\eqref{defdicorn}. So,
$$\Delta_{m} \leq 12 (m \eps+\delta) \int_{0}^{1} \sqrt{\logn
\left(N\left(t,\mathcal{D},\norme{.}\right)\right)}\ \, dt.$$ Moreover, by using
Lemma \ref{covering} and Lemma \ref{majintln}, we get that
\begin{align*} &  \int_{0}^{1} \sqrt{\logn\left(
N\left(t,\mathcal{D},\norm{.}\right)\right)}\ \, dt\\  \leq {} &  \int_{0}^{1}
\sqrt{\logn \left[(n+1)^{d+1}\ \frac{4+t}{t} \right]}\ \, dt \\  = {} &  \int_{0}^{1}
\sqrt{\logn \left((n+1)^{d+1}\right) + \logn\left(4+t\right) + \logn\left(\frac{1}{t}\right)}
\ \, dt\\ \leq {} &  \sqrt{\logn \left(
(n+1)^{\dd+1}\right)} + \int_{4}^{5} \sqrt{\logn
\left(t\right)}\ \, dt + \int_{0}^{1} \sqrt{\logn
\left(\frac{1}{t}\right)}\ \, dt \\   \leq {} & \sqrt{\logn \left(
(n+1)^{\dd+1}\right)} + \sqrt{\logn 5} + \sqrt{\pi}.
\end{align*} Thus,
$$\Delta_{m} \leq 12 (m \eps+\delta) \left[
\sqrt{\logn \left((n+1)^{\dd+1}\right)}+C\right],$$ where $C=\sqrt{\logn 5} + \sqrt{\pi} \in\, ]0,4[$.

\medskip \n Now, choose the weights of the form $x_{m}= \gamma m$ where
$\gamma$ is a positive numerical constant specified below. Then $\sum_{m\geq1}
e^{-x_{m}}=1/\left(e^{\gamma}-1\right):=\Sigma_{\gamma} < +\infty$.

\medskip \n  Defining $K = 14/13
> 1$, $\gamma=13\, ( 4-C) /4 >0$, and using the inequality
$2\sqrt{ab}\leq\eta a+\eta^{-1}b$ with $\eta = 1/6$, we get that
\begin{align}\nonumber &
2K\varepsilon\left(  \Delta_{m}+\varepsilon
x_{m}+\sqrt{\Delta_{m}\varepsilon
x_{m}}\right)\\ \non  \leq {} &  K\varepsilon\left(  \frac{13}{6}\, \Delta_{m}+8x_{m}%
\varepsilon\right)\\ \non   \leq {} &  K (m\varepsilon+\delta) \eps\left( 26
\left[\sqrt{\logn \left((n+1)^{\dd+1}\right)}+C\right] +8\gamma\right)\\ \non
\leq {} &  28 (m\varepsilon+\delta)\eps\left(\sqrt{\logn \left(
(n+1)^{\dd+1}\right)}+C + 4-C\right)  \\ \non  < {} & 28
(m\varepsilon+\delta)\eps\left(\sqrt{\logn \left( (n+1)^{\dd+1}\right)} +
4\right).\end{align} Since this inequality is true for all $\delta >0$, we get
when $\delta$ tends to 0 that
$$2K\varepsilon\left(  \Delta_{m}+\varepsilon
x_{m}+\sqrt{\Delta_{m}\varepsilon x_{m}}\right)   \leq  28 m\varepsilon^{2}\left(\sqrt{\logn \left(
(n+1)^{\dd+1}\right)} + 4\right)  \leq  \p m \eps$$ as soon as
\begin{equation}\label{condlambda} \p \geq 28
\varepsilon \left(\sqrt{\logn
\left((n+1)^{\dd+1}\right)}+4\right).\end{equation} For such values of $\p$,
condition (\ref{epengaussnl}) on the penalty function is satisfied and me may
apply Theorem~\ref{zmaingaussnl} with $\operatorname{pen}(m) = \p m
\varepsilon$ and $\rho = \p \varepsilon$ for all $m \geq 1$. Taking into
account the definition of $\hat{m}$ at \eqref{defmchapeau} and noticing that
$\eps^2 \leq \p \eps/112$ for $\p$ satisfying \eqref{condlambda}, the risk
bound (\ref{eriskgaussnl}) leads to
\begin{align*} & \mathbb{E}\left[\norm{\f-\fc}^{2} + \p \normu{\fc}\right]\\  \leq {} & C
\left[ \inf_{m\geq1}\left( \inf_{\normu{\ft} \leq
m\varepsilon}\norm{\f-\ft}^{2}+\p m\varepsilon\right) +\p\varepsilon
+(1+\Sigma_{\gamma}) \varepsilon ^{2}\right]\\  \leq {} & C \left[
\inf_{m\geq1}\left( \inf_{\normu{\ft} \leq m\varepsilon}\norm{\f-\ft}^{2}+\p
m\varepsilon\right) +\p\varepsilon\right],\end{align*} where $C>0$ is some
absolute constant. We end the proof as the one of Theorem~\ref{zlasso}. \hfill
$\Box$

%%%%%%%%%%%%%%%%%%%%%%%%%%%%%%%%%%%%%%%%%%%%%%%%%%%%%%%%%%%%%%%%%%%%%%%%%%%%%%%%%%%%%%%%%%%

\subsection{Rates of convergence}

\subsubsection{Proofs of the upper bounds}

We first prove a crucial equivalence between the rates of convergence of the
deterministic \lasso s and $K_{\mathcal{D}_{p}}$-functionals, which shall able
us to provide an upper bound of the rates of convergence of the deterministic
\lasso s when the target function belongs to some interpolation space $\bqr$.
Then, we shall pass on these rates of convergence to the \lasso\ and the
\lassotype estimators. Looking at the proofs of Proposition
\ref{propnvestlasso} and Proposition \ref{propnvest}, we can see that   the
rate of convergence of a \lasso\ estimator  is nothing else than the rate of
convergence of the corresponding deterministic \lasso, whereas we can choose
the best penalized rate of convergence of all the deterministic \lasso s to get
the rate of convergence of the \lassotype estimator, which explains why this
estimator can achieve a much better rate of convergence than any \lasso\
estimator.

\paragraph{Proof of Lemma \ref{propgenerale}.} \label{equikl} Let us first
prove the right-hand side inequality of~\eqref{eq1}. For all $\ft \in \lude$,
$\p \geq 0$ and $\delta
>0$, we have
$$\norm{\f-\ft}^2+\delta^2 \norm{\ft}_{\lude}^2+
\frac{\p^2}{4\delta^2}\ \leq\ \left(\norm{\f-\ft}+\delta
\norm{\ft}_{\lude}\right)^2 + \frac{\p^2}{4\delta^2}.$$ Taking the infimum on
all $\delta > 0$ on both sides, and noticing that the infimum on the left-hand
side is achieved for $\delta^2 = \p/\left(2\norm{\ft}_{\lude}\right)$, we get
that
$$\norm{\f-\ft}^2+ \p
\norm{\ft}_{\lude}\ \leq\ \inf_{\delta
>0}\left[\left(\norm{\f-\ft}+\delta
\norm{\ft}_{\lude}\right)^2 + \frac{\p^2}{4\delta^2}\right].$$ Then, taking the
infimum on all $\ft \in \lude$, we get that\begin{align*} & \inf_{\ft \in
\lude}\left(\norm{\f-\ft}^2+ \p \norm{\ft}_{\lude}\right)\\
  \leq {} & \inf_{\delta
>0}\left[\inf_{\ft \in \lude}\left(\norm{\f-\ft}+\delta
\norm{\ft}_{\lude}\right)^2 + \frac{\p^2}{4\delta^2}\right]\\
 = {} & \inf_{\delta
>0}\left[\left(\inf_{\ft \in \lude}\left(\norm{\f-\ft}+\delta
\norm{\ft}_{\lude}\right)\right)^2 + \frac{\p^2}{4\delta^2}\right]\\
 = {} & \inf_{\delta
>0}\left(K^2_{D}(\f,\delta) + \frac{\p^2}{4\delta^2}\right),\end{align*}
which proves the right-hand side inequality of \eqref{eq1}. Let us now prove
similarly the left-hand side inequality of \eqref{eq1}. By definition of
$L_{D}(\f,\p)$, for all $\eta
>0$, there exists $h_{\eta}$ such that
$L_{D}(\f,\p) \leq \norm{\f-h_{\eta}}^2+\p\norm{h_{\eta}}_{\lude} \leq
L_{D}(\f,\p) + \eta$. For all $\delta
>0$, we have
\begin{align*} K^{2}_{D}(\f,\delta) + \frac{\p^2}{2\delta^2}\ & =
\inf_{h \in \lude} \left(\norm{\f-h}+\delta\norm{h}_{\lude}\right)^2
+ \frac{\p^2}{2\delta^2}
\\ & \leq
\left(\norm{\f-h_{\eta}}+\delta\norm{h_{\eta}}_{\lude}\right)^2 +
\frac{\p^2}{2\delta^2} \\ & \leq
2\left(\norm{\f-h_{\eta}}^2+\delta^2\norm{h_{\eta}}_{\lude}^2\right)+
\frac{\p^2}{2\delta^2}\ .\end{align*} Taking the infimum on all
$\delta > 0$ on both sides, and noticing that the infimum on the right-hand
side is achieved for $\delta^2 = \p/\left(2\norm{\ft_{\eta}}_{\lude}\right)$,
we get that
$$ \inf_{\delta > 0}\left(K^{2}_{D}(\f,\delta) +
\frac{\p^2}{2\delta^2}\right)  \leq
2\left(\norm{\f-h_{\eta}}^2+\p\norm{h_{\eta}}_{\lude}\right) \leq
2\left(L_{D}(\f,\p) + \eta\right).$$ We get the expected inequality when $\eta$
tends to zero. \hfill $\Box$

\paragraph{Proof of Lemma \ref{majinf}.} \label{preuvemajinf} Let $p \in
\ensn$ and $\delta >0$. Applying \eqref{eq1} with $D = \mathcal{D}_{\pp}$ and
$\p=\lambdap$, we have
\begin{equation}\label{vendr}\inf_{\ft\in \ludp} \left(
\norm{\f-\ft}^{2}+\lambdap \normup{\ft}\right) \leq
 \inf_{\delta
>0}\left(K_{\mathcal{D}_{p}}^2(\f,\delta)+\frac{\lambdap^2}{4\delta^2}\right).\end{equation}
So, it just remains to bound $K_{\mathcal{D}_{p}}^2(\f,\delta)$ when $\f \in
\bqr(R)$. Let $\pa := 1/q-1/2$. By definition of $\f \in \bqr(R)$, for all
$\delta
>0$, there exists $g \in \lur$ such that $\norm{\f-g}+\delta\normlur{g} \leq
R\, \delta^{2\pa}$. So, we have
\begin{equation}
\norm{\f-g} \leq R\,
\delta^{2\pa}\label{equation1}\end{equation} and
\begin{equation} \normlur{g} \leq R\,
\delta^{2\pa-1}.\label{equation2}
\end{equation} Then, by definition of $g \in \lur$, there exists $g_{\pp} \in \ludp$ such that
\begin{equation}\label{equation3} \normup{g_{p}} \leq \normlur{g}\end{equation}
and
\begin{equation}\norm{g-g_{\pp}} \leq \normlur{g}\,
\pp^{-r}.\label{equation4}\end{equation} Then, we get from
\eqref{equation1}, \eqref{equation4} and \eqref{equation2} that
\begin{equation}\label{equation5}\norm{\f-g_{p}} \leq
\norm{\f-g}+\norm{g-g_{p}} \leq R\, \left(\delta^{2\pa}+
\delta^{2\pa-1}\, \pp^{-r}\right),\end{equation} and we deduce
from \eqref{defkdfunct}, \eqref{equation5}, \eqref{equation3} and
\eqref{equation2} that
$$
K_{\mathcal{D}_{p}}(\f,\delta) \leq
\norm{\f-g_{p}}+\delta\,\normup{g_{p}} \leq R
\left(2\delta^{2\pa}+\delta^{2\pa-1}\, \pp^{-r}\right).
$$ So, we get from \eqref{vendr} that
\begin{align}\label{dinpf}\non \inf_{\ft\in \ludp} \left( \norm{\f-\ft}^{2}+\lambdap
\normup{\ft}\right) & \leq \inf_{\delta
>0}\left(R^2
\left(2\delta^{2\pa}+\delta^{2\pa-1}\, \pp^{-r}\right)^2+\frac{\lambdap^2}{4\delta^2}\right)\\
& \leq \inf_{\delta
>0}\left(2R^2
\left(4\delta^{4\pa}+\delta^{4\pa-2}\,
\pp^{-2r}\right)+\frac{\lambdap^2}{4\delta^2}\right).\end{align} Let us now
consider $\delta_{0}$ such that $8 R^2 {\delta_{0}}^{4\pa} = \lambdap^2
\left({4\delta_{0}}^2\right)^{-1}$, and $\delta_{1}$ such that $2 R^2
\delta_{1}^{4\pa-2}\, \pp^{-2r} = \lambdap^2
\left({4\delta_{1}}^2\right)^{-1}$, that is to say $\delta_{0} = \left(\lambdap
\left(4\sqrt{2}R\right)^{-1}\right)^{1/(2\alpha+1)}$ and $\delta_{1} =
\left(\lambdap\, p^r \left(2\sqrt{2}R\right)^{-1}\right)^{1/(2\alpha)}$. We can
notice that there exists $C_{q}>0$ depending only on $q$ such that
$\delta_{0}^{4\pa-2}\, \pp^{-2r} \leq C_{q}\delta_{0}^{4\pa}$ for all $p$
checking $\lambdap\, p^{r(2\pa+1)} \geq R$, while $\delta_{1}^{4\pa}\leq
C_{q}\delta_{1}^{4\pa-2}\, \pp^{-2r}  $ for all $p$ checking $\lambdap\,
p^{r(2\pa+1)} < R$. Therefore, we deduce from \eqref{dinpf} that there exists
$C_{q}>0$ depending only on $q$ such that for all $p$ checking $\lambdap\,
p^{r(2\pa+1)} \geq R$,
\begin{align}\label{bonne}\non
\inf_{\ft\in \ludp} \left( \norm{\f-\ft}^{2}+\lambdap \normup{\ft}\right) &
 \leq  2R^2
\left(4\delta_{0}^{4\pa}+\delta_{0}^{4\pa-2}\,
\pp^{-2r}\right)+\frac{\lambdap^2}{4\delta_{0}^2}\\ \non & \leq  C_{q}\, R^2
\delta_{0}^{4\pa}\\ \non  & \leq  C_{q}\, R^{\frac{2}{2\pa+1}} \lambdap^{\frac{4\pa}{2\pa+1}}
\\  & =  C_{q}\, R^{q} \lambdap^{2-q},\end{align} while for all
$p$ checking $\lambdap\, p^{r(2\pa+1)} < R$,
\begin{align}\label{nvelle}\non
\inf_{\ft\in \ludp} \left( \norm{\f-\ft}^{2}+\lambdap \normup{\ft}\right) & \leq  2R^2
\left(4\delta_{1}^{4\pa}+\delta_{1}^{4\pa-2}\,
\pp^{-2r}\right)+\frac{\lambdap^2}{4\delta_{1}^2}\\ \non & \leq  C_{q}\, R^2
\delta_{1}^{4\pa-2}\,
\pp^{-2r}\\ \non  & \leq  C_{q}\, R^{\frac{1}{\pa}}\, p^{-\frac{r}{\pa}}\, \lambdap^{2-\frac{1}{\pa}}
\\  & =  C_{q} \left(R p^{-r}\right)^{\frac{2q}{2-q}} \lambdap^{\frac{4(1-q)}{2-q}}.\end{align}
Inequalities \eqref{bonne} and \eqref{nvelle} can be summarized into the
following result: $$\inf_{\ft\in \ludp} \left( \norm{\f-\ft}^{2}+\lambdap
\normup{\ft}\right)  \leq C_{q} \max\left(R^{q} \lambdap^{2-q}\, , \, \left(R
p^{-r}\right)^{\frac{2q}{2-q}} \lambdap^{\frac{4(1-q)}{2-q}}\right).$$\hfill
$\Box$

\paragraph{Proof of Proposition
\ref{propnvestlasso}.}\label{preuvepropnvestlasso} From \eqref{majriskl} and
\eqref{colless}, we know that there exists some constant $C_{q}>0$ depending
only on $q$ such that, for all $p \in \ensn$, the quadratic risk of $\fcp$ is
bounded by
\begin{align}\non \mathbb{E}\left[\norm{\f -\fcp}^{2}
\right] & \leq  C_{q} \left(\max\left(R^{q} \lambdap^{2-q}\, , \, \left(R
p^{-r}\right)^{\frac{2q}{2-q}} \lambdap^{\frac{4(1-q)}{2-q}}\right)+ \lambdap
\varepsilon\right)\\ \non & \leq  C_{q} \max\left(R^{q} \lambdap^{2-q}\, , \,
\left(R p^{-r}\right)^{\frac{2q}{2-q}} \lambdap^{\frac{4(1-q)}{2-q}}\, , \,
\lambdap \varepsilon\right).\end{align} By remembering that $\lambdap =
4\eps(\sqrt{\ln p}+1)$ and by comparing the three terms inside the maximum
according to the value of $p$, we get \eqref{onzes}, \eqref{douzes} and
\eqref{treizes}. \hfill $\Box$

\paragraph{Proof of Proposition \ref{propnvest}.}\label{preuvepropnvest} From
\eqref{majrisklt} and \eqref{colless}, we know that there exists some constant
$C_{q}>0$ depending only on $q$ such that the quadratic risk of $\fcf$ is
bounded by
\begin{align}\label{rosebis}\non & \mathbb{E}\left[\norm{\f -\fcf}^{2}
\right]\\ \non \leq {} &  C_{q} \left[\inf_{p \in \ens}\left(\max\left(R^{q}
\lambdap^{2-q}\, , \, \left(R p^{-r}\right)^{\frac{2q}{2-q}}
\lambdap^{\frac{4(1-q)}{2-q}}\right)+ \varepsilon^2 \logn p\right)
+\varepsilon^2\right]\\   \leq {} &  C_{q} \inf_{p \in \ens
\setminus\{1\}}\left(\max\left(R^{q} \left(\eps\sqrt{\ln p}\right)^{2-q}\, , \,
\left(R p^{-r}\right)^{\frac{2q}{2-q}} \left(\eps \sqrt{\ln
p}\right)^{\frac{4(1-q)}{2-q}}\right)+ \varepsilon^2 \logn p\right),
\end{align} where we use the fact that for all $p \geq 2$, we have $\lambdap
= 4\eps(\sqrt{\ln p}+1) \leq 4(1+1/\sqrt{\ln 2}) \eps \sqrt{\ln p}$ and $\eps^2
\leq \eps^2 (\ln p)/ \ln 2$. We now choose $p$ such that the two terms inside
the maximum are approximately of the same order. More precisely, let us define
$$J_{q,r} = \left\lceil \frac{q}{2r} \log_{2} \left(R\eps^{-1}\right)\right\rceil,$$
where $\lceil x\rceil$ denotes the smallest integer greater than $x$, and
$p_{q,r} := 2^{J_{q,r}}$. Since we have assumed $R\eps^{-1} \geq
\text{\upshape{e}}$, we have $p_{q,r} \in \ens \setminus \{1\}$ and we deduce
from \eqref{rosebis} that
\begin{align}\label{roseter} \non & \mathbb{E}\left[\norm{\f -\fcf}^{2}
\right]\\  \leq {} & C_{q} \left(\max\left(R^{q} \left(\eps\sqrt{\ln
p_{q,r}}\right)^{2-q}\, , \, \left(R p_{q,r}^{-r}\right)^{\frac{2q}{2-q}}
\left(\eps \sqrt{\ln p_{q,r}}\right)^{\frac{4(1-q)}{2-q}}\right)+ \varepsilon^2
\logn p_{q,r}\right).\end{align} Now, let us give an upper bound of each
term of the right-hand side of \eqref{roseter} by assuming that $R\eps^{-1}
\geq \max\left(\text{\upshape{e}},(4r)^{-1}q\right)$. First, we have by
definition of $p_{q,r}$ that $$\ln p_{q,r} \leq \ln
2+\frac{q}{2r}\ln\left(R\eps^{-1}\right).$$ Moreover, for all $x>0$, $\ln 2
\leq 2\, x \ln x$ and by assumption $R\eps^{-1}\geq q/(4r)$, so $$\ln 2 \leq
\frac{q}{2r} \ln \left(\frac{q}{4r}\right) \leq \frac{q}{2r} \ln
\left(R\eps^{-1}\right),$$ and thus we get that
\begin{equation}\label{paine}\ln p_{q,r} \leq
\frac{q}{r}\ln\left(R\eps^{-1}\right).\end{equation} Then, we deduce from
\eqref{paine} that the first term of \eqref{roseter} is upper bounded by
\begin{equation}\label{lever}R^{q} \left(\eps\sqrt{\ln p_{q,r}}\right)^{2-q}
 \leq R^q \left(\frac{q}{r}\right)^{1-\frac{q}{2}} \left(\eps
 \sqrt{\ln\left(R\eps^{-1}\right)}\right)^{2-q}.\end{equation} For
the second term of \eqref{roseter}, using \eqref{paine}, the fact that $p_{q,r}
\geq (R\eps^{-1})^{\frac{q}{2r}}$, that $\frac{4(1-q)}{2-q} \leq 2-q$ and that
$\ln(R\eps^{-1}) \geq 1$, we get that
\begin{align}\label{mamia} \non  \left(R
p_{q,r}^{-r}\right)^{\frac{2q}{2-q}} \left(\eps \sqrt{\ln
p_{q,r}}\right)^{\frac{4(1-q)}{2-q}} &  \leq   R^q \eps^{2-q}
\left(\sqrt{\frac{q}{r}\ln\left(R\eps^{-1}\right)}\right)^{\frac{4(1-q)}{2-q}}\\  &  \leq
R^q \left(\frac{q}{r}\right)^{\frac{2(1-q)}{2-q}}\left(\eps
\sqrt{\ln\left(R\eps^{-1}\right)}\right)^{2-q}.\end{align}
For the third term of \eqref{roseter},  we have $$\eps^{2} \ln p_{q,r}  \leq
 \frac{q}{r}\, \eps^{2} \ln\left(R\eps^{-1}\right)
 =  \frac{q}{r}
 \left(\frac{\ln \left[(R\eps^{-1})^{2}\right]}{2(R\eps^{-1})^{2}}\right)^{q/2} R^q \left(\eps
\sqrt{\ln\left(R\eps^{-1}\right)}\right)^{2-q}.$$ Now, let us introduce
$$g:\ ]0,+\infty[ \mapsto \mathbb{R},\ x \mapsto \frac{\ln x}{x}.$$ It is easy
to check that $g(x^2) \leq 1/x$ for all $x>0$. Using this property and the fact
that $R\eps^{-1} \geq \text{\upshape{e}}$, we get that
$$\frac{\ln \left[(R\eps^{-1})^{2}\right]}{(R\eps^{-1})^{2}} = g\left((R\eps^{-1})^{2}\right) \leq \frac{1}{R\eps^{-1}}
\leq \frac{1}{\text{\upshape{e}}},$$ and thus
\begin{equation}\label{france}  \eps^{2} \ln p_{q,r}  \leq  \frac{q}{r} \left(\frac{1}{2\text{\upshape{e}}}\right)^{q/2}
 R^q \left(\eps \sqrt{\ln\left(R\eps^{-1}\right)}\right)^{2-q}.\end{equation}
Then, we deduce from \eqref{roseter}, \eqref{lever}, \eqref{mamia} and
\eqref{france} that there exists $C_{q,r}>0$ depending only on $q$ and $r$ such
that
$$\mathbb{E}\left[\norm{\f -\fcf}^{2} \right] \leq C_{q,r}\, R^q \left(\eps
\sqrt{\ln\left(R\eps^{-1}\right)}\right)^{2-q}.$$  \hfill $\Box$

\paragraph{Proof of Proposition \ref{propcvlassorn}.}\label{preuvepropcvlassorn} Set $\eps = \sigma/\sqrt{n}$. From Theorem \ref{zlassorn}, we have
\begin{equation}\label{take}\mathbb{E}\left[\norm{\f-\fc}^{2}\right] \leq C\left[
\inf_{\ft \in \lu} \left( \norm{\f-\ft}^{2}+\p \normu{\ft}\right) +\p
\eps\right],\end{equation} where $C$ is an absolute positive constant. Then, if
$\f \in \bq(R)$, we get from \eqref{kf} and \eqref{eq1} with $D=\mathcal{D}$
that $$\inf_{\ft \in \lu} \left(\norm{\f-\ft}^2+\p \normu{\ft}\right) \leq
\inf_{\delta
>0}\left(R^2
\delta^{4\pa}+\frac{\p^2}{4\delta^2}\right).$$ The infimum on the right-hand
side is achieved for $\delta = \left(\p/(2R)\right)^{1/(2\pa+1)}$ and the last
inequality leads to \begin{align*}\inf_{\ft \in \lud} \left(\norm{\f-\ft}^2+\p
\normu{\ft}\right)  & \leq  2 R^{\frac{2}{2\pa+1}}
\left(\frac{\p}{2}\right)^{\frac{4\pa}{2\pa+1}}\\ & =
2^{\frac{1-2\pa}{1+2\pa}}\, R^{\frac{2}{2\pa+1}}\, \p^{\frac{4\pa}{2\pa+1}}\\ &
= 2^{q-1} R^{q} \p^{2-q}\, \text{.}\end{align*} Thus, we deduce from
\eqref{take} that there exists some $C_{q}>0$ depending only on~$q$ such that
\begin{align}  \non & \mathbb{E}\left[\norm{\f-\fc}^{2}\right]\\ \non \leq {} &
C_{q}\left[R^{q} \p^{2-q} +\p \eps\right]\\ \non
 \leq {} &  C_{q}\left[R^{q} \left(\eps\left(\sqrt{\logn
\left((n+1)^{d+1}\right)}+4\right)\right)^{2-q} +\eps^2 \left(\sqrt{\logn
\left((n+1)^{d+1}\right)}+4\right)\right]\\ \label{jatrep}
\leq {} &  C_{q}\left[R^{q}
\left(\eps\sqrt{\logn \left((n+1)^{d+1}\right)}\right)^{2-q}
+\eps^2 \sqrt{\logn \left((n+1)^{d+1}\right)}\right]\\ \non
 \leq {} &
 C_{q} \max\left(R^{q}
\left(\eps\sqrt{\logn \left((n+1)^{d+1}\right)}\right)^{2-q}\, , \, \eps^{2}
\sqrt{\logn \left((n+1)^{d+1}\right)}\right)\\ \label{repjat} \leq {} & C_{q}\,
R^{q} \left(\eps\sqrt{\logn \left((n+1)^{d+1}\right)}\right)^{2-q},\end{align}
where we get \eqref{jatrep} by using the fact $4 \leq 5 \sqrt{\ln2} \leq 5
\sqrt{\logn \left((n+1)^{d+1}\right)}$ for $n\geq 1$ and $d\geq1$ and
\eqref{repjat} thanks to the assumption $R\eps^{-1} \geq \left[\logn
\left((n+1)^{d+1}\right)\right]^{\frac{q-1}{2q}}.~\Box$

%%%%%%%%%%%%%%%%%%%%%%%%%%%%%%%%%%%%%%%%%%%%%%%%%%%%%%%%%%%%%%%%%%%%%%%%%%%%%%%%%%%%%%%%%%%%%%
\subsubsection{Proofs of the lower bounds in the orthonormal case}

To prove that the rates of convergence \eqref{quatorzes} achieved by the
\lassotype estimator on the classes $\bqr$ are optimal, we propose to establish
a lower bound of the minimax risk over $\bqr$ when the dictionary is an
orthonormal basis of~$\mathbb{H}$ and to check that it is of the same order as
the rates \eqref{quatorzes}. The first central point is to prove Remark
\ref{eqcocnor}, that is to say the inclusion in the orthonormal case of the
space $\wlq(R)\cap \besov(R)$ in the space $\bqr(C_{q,r}R)$  for all $R>0$ and
some $C_{q,r}>0$ depending only on $q$ and $r$. Taking this inclusion into
account, we shall then focus on establishing a lower bound of the minimax risk
over the smaller space $\swlq(R)\cap \besov(R)$, which shall reveal to be an
easy task, and which entails the same lower bound over the bigger space $\bqr$.

\paragraph{Proof of Remark \ref{eqcocnor}.}\label{preuveeqcocnor} Let $R>0$. The
first inclusion comes from the simple inclusion $\swlq(R) \subset \wlq(R)$. Let
us prove the second inclusion here. Assume that $\f \in \wlq(R) \cap
\besov(R)$. For all $p \in \ensn$ and $\pb
>0$, define \begin{equation}\label{deffpbp}\f_{p,\pb} := \argmin_{\ft \in
\ludp}\left(\norm{\f-\ft}^2+\pb\normup{\ft}\right).\end{equation} The proof
will be divided in two main parts. First, we shall choose $\beta$ such that
$\f_{p,\pb} \in \lur$. Secondly, we shall choose $p$ such that
$\norm{\f-f_{p,\pb}} + \delta \normlur{f_{p,\pb}} \leq C_{q,r} R\,
\delta^{2\alpha}$ for all $\delta>0$, some $C_{q,r}>0$ and $\alpha=1/q-1/2$,
which shall prove that $\f \in \bqr(C_{q,r}R)$. To establish our results, we
shall need an upper bound of $\norm{\f-f_{p,\pb}}$ and $\normup{f_{p,\pb}}$.
These bounds are provided by Lemma \ref{gentmam} stated below.

\medskip \n  Let us first choose $\beta$ such that $\f_{p,\pb} \in \lur$.
From Lemma \ref{gentmam}, we have $$\norm{\f-\f_{p,\pb}} \leq R
(p+1)^{-r}+\sqrt{C_{q}}\, R^{q/2} \pb^{1-q/2}.$$ Now choose $\pb$ such that $R
(p+1)^{-r} = \sqrt{C_{q}}\, R^{q/2} \pb^{1-q/2}$, that is to say
\begin{equation}\label{defbp}\pb_{\pp} = R \left(\sqrt{C_{q}}\, (p+1)^{r}\right)^{-\frac{2}{2-q}}.\end{equation} Then, we
have
\begin{equation}\label{lj} \norm{\f-\f_{p,\pb_{p}}} \leq 2 R
(p+1)^{-r} \leq 2 R
(2p)^{-r} = 2^{1-r} R p^{-r}.
\end{equation}
Let us now check that $\fpbp \in \lur$.  Define
\begin{equation}\label{defcp}C_{p}:= \max\left\{2^{2-r} R,
\max_{p' \in \ensn, p' \leq
p}\norm{\f_{p',\beta_{p'}}}_{\mathcal{L}_{1}(\mathcal{D}_{p'})}\right\}.\end{equation}
Let $p'\in \ensn$. By definition of $\f_{p',\beta_{p'}}$, we have
$\f_{p',\beta_{p'}} \in \mathcal{L}_{1}(\mathcal{D}_{p'})$. If $p' \leq p$,
then we deduce from \eqref{lj} that
$$\norm{\fpbp-\f_{p',\beta_{p'}}}  \leq \norm{\fpbp-\f} +
\norm{\f-\f_{p',\beta_{p'}}}  \leq 2^{1-r} R \left(p^{-r}+p'^{-r}\right)  \leq
C_{p} p'^{-r},$$ and we have
$\norm{\f_{p',\beta_{p'}}}_{\mathcal{L}_{1}(\mathcal{D}_{p'})} \leq C_{p}$ by
definition of $C_{p}$. If $p' > p$, then $\ludp \subset
\mathcal{L}_{1}(\mathcal{D}_{p'})$ and $\fpbp \in
\mathcal{L}_{1}(\mathcal{D}_{p'})$ with
$\norm{\fpbp}_{\mathcal{L}_{1}(\mathcal{D}_{p'})} \leq \normup{\fpbp} \leq
C_{p}$ and $\norm{\fpbp-\fpbp} = 0 \leq C_{p} p'^{-r}$. So, $\fpbp \in \lur$.

\medskip \n Now, it only remains to choose a convenient $p \in \ensn$ so as to
prove that $f \in~\bqr(R)$.

\n Let us first give an upper bound of $\normlur{\fpbp}$ for all $p \in \ensn$.
By definition of $\normlur{\fpbp}$ and the above upper bounds, we have
$\normlur{\fpbp} \leq C_{p}$. So, we just have to bound $C_{p}$. Let $p' \in
\ensn, p'\leq p$. From Lemma~\ref{gentmam}, we know that there exists $C_{q}
>0$ depending only on $q$ such that
$\norm{\f_{p',\beta_{p'}}}_{\mathcal{L}_{1}(\mathcal{D}_{p'})} \leq C_{q} R^q
\pb_{p'}^{1-q}$. So, we get from  \eqref{defbp} that
\begin{align*}
\norm{\f_{p',\beta_{p'}}}_{\mathcal{L}_{1}(\mathcal{D}_{p'})} &  \leq C_{q} R
\left(\sqrt{C_{q}}\, (p'+1)^{r}\right)^{\frac{2(q-1)}{2-q}}\\ & \leq C_{q} R
\left(\sqrt{C_{q}} (p+1)^{r}\right)^{\frac{2(q-1)}{2-q}}\\ & =
C_{q}^{\frac{1}{2-q}} R \left(2p\right)^{\frac{2(q-1)r}{2-q}},\end{align*} and we deduce
from \eqref{defcp} that $$C_{p} \leq \max\left(2^{2-r} R, C_{q}^{\frac{1}{2-q}}
R \left(2p\right)^{\frac{2(q-1)r}{2-q}}\right) \leq C_{q,r} R
p^{\frac{2(q-1)r}{2-q}}$$ where $C_{q,r}>0$ depends only on $q$ and $r$. Thus,
we have
\begin{equation}\label{ct}\normlur{\fpbp}\leq C_{q,r} R
p^{\frac{2(q-1)r}{2-q}}.\end{equation}

\medskip \n Then, we deduce from \eqref{lj} and \eqref{ct} that for all $p \in \ensn$ and
$\delta>0$,
\begin{align}\label{oeucq}\non  \inf_{\ft \in
\lur}{\norm{\f-\ft}+\delta\normlur{\ft}} &
\leq  \norm{\f-\fpbp}+\delta\normlur{\fpbp}\\ & \leq 2^{1-r} R p^{-r}+\delta
C_{q,r} R p^{\frac{2(q-1)r}{2-q}}.\end{align} We now choose $p \geq 2$ such
that $p^{-r} \approx \delta\, p^{\frac{2(q-1)r}{2-q}}$. More precisely, set
$p=2^J$ where $J=\left\lceil (2-q)(qr)^{-1} \log_{2}
(\delta^{-1})\right\rceil$. With this value of $p$, we get that there exists
$C^{\prime}_{q,r}>0$ depending only on $q$ and $r$ such that \eqref{oeucq} is
upper bounded by $C^{\prime}_{q,r} R\,
 \delta^{(2-q)/q} = C^{\prime}_{q,r} R\,
 \delta^{2 \alpha}$. This means that $\f
\in \bqr(C^{\, \prime}_{q,r}\, R)$, hence \eqref{eqcocno}.\hfill $\Box$

\smallskip
\begin{lemma} \label{gentmam} Assume that the dictionary $\mathcal{D}$ is an
orthonormal basis of the Hilbert space $\mathbb{H}$ and that there exist
$1<q<2$, $r>0$ and $R>0$ such that $\f \in \wlq(R) \cap \besov(R).$ For all $p
\in \ensn$ and $\beta
>0$, define $$\f_{p,\beta} := \argmin_{\ft \in
\ludp}\left(\norm{\f-\ft}^2+\pb\normup{\ft}\right).$$ Then, there exists
$C_{q}>0$ depending only on $q$ such that for all $\pp \in \ensn$ and
$\beta>0$,
$$\normup{f_{p,\beta}} \leq C_{q} R^{q} \blambdap^{1-q}$$ and
$$\norm{\f-\f_{p,\pb}} \leq  R (p+1)^{-r}+\sqrt{C_{q}}\, R^{q/2} \pb^{1-q/2}.$$
\end{lemma}
\n The proof of Lemma \ref{gentmam} uses the two following technical lemmas.

\begin{lemma}\label{lemme1} For all $a=\left(a_{1},\dots,a_{p}\right)\in
\mathbb{R}^p$ and $\gamma > 0$,  $$\sum_{j=1}^{p} a_{j}^2\ \mathds{1}_{\{\vert
a_{j}\vert \leq \gamma\}} \leq 2 \sum_{j=1}^{p} \int_{0}^{\gamma} t\,
\mathds{1}_{\{\vert a_{j}\vert > t\}}\ dt.$$
\end{lemma}

\begin{proof}
\begin{align*} &
2 \sum_{j=1}^{p} \int_{0}^{\gamma} t\, \mathds{1}_{\{\vert
a_{j}\vert
> t\}}\ dt\\  = {} &  2 \sum_{j=1}^{p}\left[ \left(\int_{0}^{\gamma}
t\, \mathds{1}_{\{\vert a_{j}\vert > t\}}\ dt\right)
\mathds{1}_{\{\vert a_{j}\vert
> \gamma\}} +  \left(\int_{0}^{\gamma} t\, \mathds{1}_{\{\vert
a_{j}\vert
> t\}}\  dt\right) \mathds{1}_{\{\vert a_{j}\vert \leq
\gamma\}}\right]\\  = {} & 2 \sum_{j=1}^{p}\left[
\left(\int_{0}^{\gamma} t\ dt\right) \mathds{1}_{\{\vert a_{j}\vert
> \gamma\}} +  \left(\int_{0}^{\vert a_{j}\vert} t\ dt\right)
\mathds{1}_{\{\vert a_{j}\vert \leq \gamma\}}\right]\\  = {} &
\sum_{j=1}^{p} \left(\gamma^2\, \mathds{1}_{\{\vert a_{j}\vert
> \gamma\}} +  a_{j}^2\,
\mathds{1}_{\{\vert a_{j}\vert \leq \gamma\}}\right)\\  \geq {} &
\sum_{j=1}^{p} a_{j}^2\, \mathds{1}_{\{\vert a_{j}\vert \leq
\gamma\}}.\end{align*}\end{proof}

\begin{lemma}\label{lemme2} For all $a=\left(a_{1},\dots,a_{p}\right)\in
\mathbb{R}^p$ and $\gamma > 0$,  $$\sum_{j=1}^{p} \vert a_{j}\vert
\mathds{1}_{\{\vert a_{j}\vert > \gamma\}} = \gamma \sum_{j=1}^{p}
\mathds{1}_{\{\vert a_{j}\vert > \gamma\}} + \sum_{j=1}^{p}
\int_{\gamma}^{+\infty} \mathds{1}_{\{\vert a_{j}\vert > t\}}\ dt.$$
\end{lemma}
\begin{proof}
$$
\sum_{j=1}^{p} \int_{\gamma}^{+\infty} \mathds{1}_{\{\vert
a_{j}\vert > t\}}\ dt  =  \sum_{j=1}^{p}
\left(\int_{\gamma}^{\vert a_{j} \vert} dt\right)
\mathds{1}_{\{\vert a_{j}\vert > \gamma\}}  =
\sum_{j=1}^{p}\left(\vert a_{j}\vert-\gamma\right)
\mathds{1}_{\{\vert a_{j}\vert > \gamma\}}.$$ \end{proof}

\paragraph{Proof of Lemma \ref{gentmam}.}\label{proofgentmam} Let denote by
$\{\theta^{*}_{j}\}_{j\in\mathbb{N}^*}$ the coefficients of the target function
$\f$ in the basis $\mathcal{D}=\{\vp_{j}\}_{j\in\mathbb{N}^*}$, $\f =
\theta^{*}.\vp = \sum_{j\in \mathbb{N}^*} \cf_{j}\, \vp_{j}$. We introduce for
all $p\in \ensn$,
$$\Theta_{p} := \left\{\theta =(\theta_{j})_{j\in \mathbb{N}^*},\ \theta =
 \left(\theta_{1},\dots,\theta_{\pp},0,\dots,0,\dots\right)\right\}.$$
Let $\blambdap>0$. Since $f_{p,\beta} \in \ludp$, there exists a unique
$\theta^{p,\beta} \in \Theta_{p}$ such that $f_{p,\beta} =
\theta^{p,\beta}.\vp$. Moreover, from \eqref{defnormemodelpu} and using the
orthonormality of the basis functions~$\vp_{j}$, we have
\begin{equation}\label{douze}\theta^{p,\beta} =  \argmin_{\theta \in
\Theta_{p}}\left(\norm{\theta^*.\vp-\theta.\vp}^2+\blambdap\norm{\theta}_{1}\right) =
\argmin_{\theta \in \Theta_{p}}\left(\norm{\theta^*-\theta}^2+\blambdap\norm{\theta}_{1}\right).
 \end{equation} By
calculating the subdifferential of the function $\theta \in \mathbb{R}^{\pp}
\mapsto \norm{\theta^*-\theta}^2+\blambdap\norm{\theta}_{1}$, we get that the
solution of the convex minimization problem \eqref{douze} is $\thetalp =
(\thetalp_{1},\dots,\thetalp_{\pp},0,\dots,0,\dots)$ where for all
$j=1,\dots,\pp$,
$$\thetalp_{j}\ = \
 \left\{\begin{array}{ll}
\cf_{j} - \blambdap/2 & \quad \text{if } \cf_{j} > \blambdap/2,\\
\cf_{j} + \blambdap/2 & \quad \text{if } \cf_{j} < -\blambdap/2,\\
0 & \quad \text{else} .\end{array}\right.$$ Then, we have
\begin{align}\label{fleurr}\non
\norm{\f-\f_{p,\beta}}^{2} & =  \norm{\theta^*-\thetalp}^2\\ \non & =  \sum_{j=1}^{\infty}
\left(\cf_{j}-\thetalp_{j}\right)^2 \\ \non & =
\sum_{j=\pp+1}^{\infty} {\cf_{j}}^2 +  \sum_{j=1}^{\pp}
{\cf_{j}}^2 \mathds{1}_{\{\vert \cf_{j} \vert \leq \blambdap/2\}} +
 \sum_{j=1}^{\pp} \frac{\blambdap^2}{4}\, \mathds{1}_{\{\vert \cf_{j}
\vert > \blambdap/2\}}\\ &  \leq  \underbrace{\sum_{j=\pp+1}^{\infty} {\cf_{j}}^2}_{(i)} +
\underbrace{\sum_{j=1}^{\pp} {\cf_{j}}^2 \mathds{1}_{\{\vert
\cf_{j} \vert \leq \blambdap/2\}}}_{(ii)} +  \frac{\blambdap}{2} \underbrace{
\sum_{j=1}^{\pp} \vert \cf_{j} \vert \mathds{1}_{\{\vert \cf_{j}
\vert>\blambdap/2\}}}_{(iii)}\ \text{ .}
\end{align} while
\begin{align}\label{fleurv} \non
\normup{f_{p,\beta}} & =   \sum_{j=1}^{\infty} \vert\thetalp_{j}\vert\\ \non &  =
\sum_{j=1}^{\pp}
\left(\vert \cf_{j} \vert - \frac{\blambdap}{2} \right)
\mathds{1}_{\{\vert \cf_{j} \vert
>\blambdap/2\}}\\  &  \leq
\sum_{j=1}^{\pp} \vert \cf_{j} \vert \mathds{1}_{\{\vert \cf_{j}
\vert>\blambdap/2\}}  =  (iii) \text{ .}
\end{align}
Now, since $\f$ is assumed to belong to $\besov(R)$, we get from
\eqref{defbesovr} that $(i)$ is bounded by
\begin{equation}\label{supm} \sum_{j=\pp+1}^{\infty} {\cf_{j}}^2  \leq
R^2 (\pp+1)^{-2r}.\end{equation} Let us now bound $(ii)$ and $(iii)$ thanks to
the assumption $\f \in \wlq(R)$. By applying Lemma~\ref{lemme1}   and Lemma
\ref{lemme2}  with $a_{j} = \cf_{j}$ for all $j=1,\dots,p$ and $\gamma =
\blambdap/2$, and by using the fact that $\sum_{j=1}^{p} \mathds{1}_{\{\vert
\cf_{j}\vert > t\}} \leq \sum_{j=1}^{\infty} \mathds{1}_{\{\vert \cf_{j}\vert >
t\}} \leq R^q t^{-q}$ for all $t
>0$ if $\f \in \wlq(R)$,  we get that $(ii)$ is bounded by
\begin{align}\label{eq11}\non \sum_{j=1}^{p} {\cf_{j}}^2
\mathds{1}_{\{\vert\cf_{j}\vert \leq \blambdap/2\}}  &\leq  2
\sum_{j=1}^{p} \int_{0}^{\blambdap/2} t \mathds{1}_{\{\vert
\cf_{j}\vert
> t\}}\ dt\\ \non & \leq  2 R^q \int_{0}^{\blambdap/2} t^{1-q}\ dt\\ & =
\frac{2^{q-1}}{2-q}\ R^{q} \blambdap^{2-q},\end{align} while $(iii)$ is bounded by
\begin{align}\label{eq12}\non  \sum_{j=1}^{p} \vert\cf_{j}\vert \mathds{1}_{\{\vert\cf_{j}\vert
> \blambdap/2\}} & =  \frac{\blambdap}{2} \sum_{j=1}^{p}
\mathds{1}_{\{\vert \cf_{j}\vert > \blambdap/2\}} + \sum_{j=1}^{p}
\int_{\blambdap/2}^{+\infty} \mathds{1}_{\{\vert \cf_{j}\vert > t\}}\
dt\\ \non &  \leq
R^q \left(\frac{\blambdap}{2}\right)^{1-q} +
R^q \int_{\blambdap/2}^{+\infty}
 t^{-q}\ dt \\ &  =   \frac{q\, 2^{q-1}}{q-1}\, R^{q} \blambdap^{1-q}.\end{align}
Gathering together \eqref{fleurv} and \eqref{eq12} on the one hand and
\eqref{fleurr}, \eqref{supm}, \eqref{eq11} and \eqref{eq12} on the other hand,
we get that there exists $C_{q}>0$ depending only on $q$ such that
$$\normup{f_{p,\beta}} \leq C_{q} R^{q} \blambdap^{1-q}$$
and $$\norm{\f-f_{p,\beta}}^{2} \leq R^2 (\pp+1)^{-2r} + C_{q} R^q
\blambdap^{2-q}.$$ Finally,
$$\norm{\f-\f_{p,\pb}} \leq \sqrt{R^2
(\pp+1)^{-2r} + C_{q} R^q \blambdap^{2-q}} \leq  R
(p+1)^{-r}+\sqrt{C_{q}}\, R^{q/2} \pb^{1-q/2}.$$ \hfill $\Box$

\paragraph{Proof of Proposition \ref{minimax}.}\label{preuveminimax} Let us define
$$\R = \eps \sqrt{u\ln\left(\rb\, \eps^{-1}\right)}, \quad p = 2^{J},\quad d =
2^{K},$$ with
$$J = \left\lfloor \frac{2-q}{2r}\ \log_{2}
\left(\rw\R^{-1}\right)\right\rfloor$$ and
$$K = \left\lfloor q \log_{2} \left(\rw \R^{-1}\right)\right\rfloor.$$
Let us first  check that $\R$ is well-defined and that $d\leq p$ under the
assumptions of Proposition \ref{minimax}. Under the assumption $r<1/q-1/2$, we
have $u>0$, and since $R\eps^{-1} \geq \text{\upshape{e}}^2 \geq
\text{\upshape{e}}$, $\R$ is well-defined. Moreover, since $r<1/q-1/2$, we have
$(2-q)/(2r) > q$, so it only remains to check that $R\R^{-1} \geq
\text{\upshape{e}}$ so as to prove that $d \leq p$. We shall in fact prove the
following stronger result:

\medskip
\n Result $(\diamondsuit)$: If $R\eps^{-1} \geq\max(\text{\upshape{e}}^2,u^2)$,
then $R\eps^{-1}/\left(\ln(R\eps^{-1})\right) \geq u$.

\medskip \n This result indeed implies that, under the assumption $R\eps^{-1} \geq\max(\text{\upshape{e}}^2,u^2)$,
$$R\R^{-1} = \frac{R\eps^{-1}}{\sqrt{u \ln \left(R\eps^{-1}\right)}} =
\sqrt{R\eps^{-1}}\,
\sqrt{\frac{R\eps^{-1}}{u \ln \left(R\eps^{-1}\right)}}
 \geq \text{\upshape{e}} \times 1 \geq  \text{\upshape{e}}.$$

\n Let us prove Result $(\diamondsuit)$. Introduce the function $$g:\
]0,+\infty[ \mapsto \mathbb{R},\ x \mapsto \frac{x}{\ln x}.$$ It is easy to
check that $g$ is non-decreasing on $[\text{\upshape{e}},+\infty[$ and that
$g(x^2)\geq x$ for all $x>0$. Now, assume that $R\eps^{-1}\geq
\max(\text{\upshape{e}}^2,u^2)$. Using the properties of $g$, we deduce that if
$u\geq \text{\upshape{e}}$ then $R\eps^{-1}\geq u^2 \geq \text{\upshape{e}}^2
\geq \text{\upshape{e}}$ and
$$\frac{R\eps^{-1}}{\ln(R\eps^{-1})} = g(R\eps^{-1}) \geq g(u^2) \geq u,$$
while if $u < \text{\upshape{e}}$ then  $R\eps^{-1}\geq \text{\upshape{e}}^2
\geq \text{\upshape{e}}$ and
$$\frac{R\eps^{-1}}{\ln(R\eps^{-1})} = g(R\eps^{-1}) \geq g(\text{\upshape{e}}^2)
\geq \text{\upshape{e}} > u,$$ hence Result $(\diamondsuit)$.

\medskip \n Now, consider the following hypercube $\Theta(p,d,\R)$ defined by
\begin{align*} & \left\{\sum_{j=1}^{\infty}\theta_{j}\,
\vp_{j},\ \left(\theta_{1},\dots,\theta_{p}\right) \in [0,\R]^p,\ \theta_{j} = 0
\text{ for } j \geq p+1,\ \sum_{j=1}^{p} \mathds{1}_{\{\theta_{j}\neq0\}} = d\right\}\\
 = {}&  \left\{\R \sum_{j=1}^{\infty}\beta_{j}\, \vp_{j},\ \left(\beta_{1},\dots,\beta_{p}\right)
\in [0,1]^p,\ \beta_{j} = 0 \text{ for } j \geq p+1,\ \sum_{j=1}^{p}
\mathds{1}_{\{\beta_{j}\neq0\}} = d\right\}.\end{align*} The essence of the
proof is just to ckeck that $\Theta(p,d,\R) \subset \swlq(\rw) \cap
\besov(\rb)$, which shall enable us to bound from below the minimax risk over
$\swlq(\rw) \cap \besov(\rb)$ by the lower bound of the minimax risk over
$\Theta(p,d,\R)$ provided in~\cite{BM99g}.

\medskip \n Let $h \in \Theta(p,d,\R)$. We write $h  = \sum_{j=1}^{\infty} \theta_{j}
\vp_{j} = \R \sum_{j=1}^{\infty} \beta_{j} \vp_{j}$. $$\sum_{j=1}^{\infty}
\vert\theta_{j}\vert^q  =  \R^{q} \sum_{j=1}^{p} \beta_{j}^q\,
\mathds{1}_{\{\beta_{j}\neq0\}}\leq  \R^{q}  \sum_{j=1}^{p}
\mathds{1}_{\{\beta_{j}\neq0\}} \leq
  \R^{q} d \leq  \R^{q} \left(\rw \R^{-1}\right)^q \leq  \rw^{q}.$$ Thus, $h \in \swlq(\rw)$.

\medskip \n Let $J_{0} \in \mathbb{N}^{*}$. If $J_{0}>p$, then,
$$ J_{0}^{2r} \sum_{j=J_{0}}^{\infty} \theta_{j}^2  \leq  J_{0}^{2r}
\sum_{j=p+1}^{\infty} \theta_{j}^2 = 0 \leq \rb^2.$$ Now consider $J_{0} \leq p$.
 Then, \begin{align*} J_{0}^{2r} \sum_{j=J_{0}}^{\infty} \theta_{j}^2
 &  =   J_{0}^{2r} \R^2 \sum_{j=J_{0}}^{p}
\beta_{j}^2 \mathds{1}_{\{\beta_{j}\neq0\}}\\   &  \leq   J_{0}^{2r} \R^2
\sum_{j=J_{0}}^{p} \mathds{1}_{\{\beta_{j}\neq0\}}\\ &  \leq   p^{2r} \R^2 d \\
& \leq  \left(\rb \R^{-1}\right)^{2-q} \R^2 \left(\rw \R^{-1}\right)^q\\
& \leq  \rb^2.
\end{align*} Thus, $h \in \besov(\rb)$. Therefore, $\Theta(p,d,\R) \subset \swlq(\rw) \cap \besov(\rb)$
and  \begin{equation}\label{jard}\inf_{\fct} \sup_{\f \in \swlq(\rw) \cap
\besov(\rb)} \mathbb{E}\left[\norm{\f-\fct}^2\right]  \geq  \inf_{\fct}
\sup_{\f \in \Theta\left(p,d,\R\right)}
\mathbb{E}\left[\norm{\f-\fct}^2\right].\end{equation} Now, from Theorem 5 in
\cite{BM99g}, we know that the minimax risk over $\Theta\left(p,d,\R\right)$
satisfies
\begin{align}\label{pacon}\non \inf_{\fct} \sup_{\f \in
\Theta\left(p,d,\R\right)} \mathbb{E}\left[\norm{\f-\fct}^2\right]& \geq
\kappa d \min\left(\R^2, \eps^2\left(1+\ln\left(\frac{p}{d}\right)\right)\right)\\ \non
 & \geq  \kappa \frac{\left(\rw \R^{-1}\right)^q}{2}
 \min\left(\R^2, \eps^2\left(1+\ln\left(\frac{p}{d}\right)\right)\right)\\  & \geq
\kappa' \rw^{q} \R^{-q}
 \min\left(\R^2, \eps^2\left(1+\ln\left(\frac{p}{d}\right)\right)\right),\end{align} where $\kappa>0$
 and $\kappa'>0$ are absolute constants.  Moreover, we have
\begin{align}\label{accli}\non \eps^2\left(1+\ln\left(\frac{p}{d}\right)\right) & \geq
\eps^{2} \left(1+\ln\left[\frac{\left(\rb \R^{-1}\right)^{\frac{2-q}{2r}}}
{2\left(\rw \R^{-1}\right)^q}\right]\right)\\ \non & =   \eps^{2}
\left(1+\ln\left[\left(\rb \R^{-1}\right)^{u}\right]-\ln2\right)\\ \non &
\geq   \eps^{2} \ln\left[\left(\rb \R^{-1}\right)^{u}\right]\\ \non& =
\eps^{2} \ln\left[\left(\rb \eps^{-1}\right)^{u} \left(\eps
\R^{-1}\right)^{u}\right]\\ \non & =  \R^2 + \eps^{2} \ln\left[\left(\eps
\R^{-1}\right)^{u}\right]\\  & =  \R^2 - \frac{u}{2}\, \eps^{2}
\ln\left[u\ln\left(R \eps^{-1}\right)\right].\end{align} But the assumption
$R\eps^{-1}\geq \max(\text{\upshape{e}}^2,u^2)$ implies that \eqref{accli} is
greater than $\R^{2}/2$. Indeed, first notice that
\begin{equation}\label{fetm}  \R^2 - \frac{u}{2}\, \eps^{2}
\ln\left[u\ln\left(R \eps^{-1}\right)\right] \geq \R^2/2\ \Leftrightarrow\
\frac{R\eps^{-1}}{\ln(R\eps^{-1})} \geq u,\end{equation} and then apply Result $(\diamondsuit)$ above.
Thus, we deduce from \eqref{jard}, \eqref{pacon}, \eqref{accli} and
\eqref{fetm}  that there exists $\kappa''>0$ such that
$$ \inf_{\fct} \sup_{\f \in \swlq(\rw) \cap \besov(\rb)}
\mathbb{E}\left[\norm{\f-\fct}^2\right] \geq \kappa'' \rw^{q} \R^{2-q} =
\kappa'' u^{1-\frac{q}{2}} \rw^{q} \left(\eps\sqrt{\ln(R\eps^{-1})}\right)^{2-q}.$$ \hfill $\Box$

\end{document}